\documentclass[a4paper,12pt,reqno,draft]{amsart}
\usepackage{amssymb,amsmath,array,amscd,amsthm,hhline}
\usepackage[mathscr]{euscript}

\usepackage{color}

\usepackage{mathrsfs}

\usepackage[dvipdfm,backref=false,final=true]{hyperref}

\hypersetup{bookmarks=true}

\renewcommand{\labelenumi}{{\rm \theenumi}}
\renewcommand{\theenumi}{{\rm(\arabic{enumi})}}
\renewcommand\epsilon{\varepsilon}
\renewcommand\emptyset{\varnothing}

\def\lm{\lambda}

\def\<{\langle}
\def\>{\rangle}

\renewcommand\d{{\rm d}}
\def\rk{\mathop{\rm rk}}
\def\grk{\rk\nolimits'}

\def\B{\mathscr{B}}
\def\W{\mathcal W}
\def\S{\mathcal S}
\def\T{\mathcal T}
\def\M{\mathscr{M}}
\def\N{\mathscr{N}}
\def\Z{\mathbb Z}
\def\C{\mathcal C}
\def\G{\mathcal G}
\def\P{\mathcal P}
\def\PP{\mathscr P}
\def\D{\mathcal D}
\def\Q{\mathcal Q}
\def\I{\mathcal I}
\def\m{\mathfrak m}

\def\V{\mathcal V}
\def\E{\mathcal E}
\def\F{\mathbb F}
\def\b{\mathbf b}
\def\VF{\mathsf V}

\newtheorem{theorem}{Theorem}[subsection]
\newtheorem{proposition}[theorem]{Proposition}
\newtheorem{lemma}[theorem]{Lemma}
\newtheorem{corollary}[theorem]{Corollary}
\newtheorem{definition}[theorem]{Definition}
\newtheorem{example}[theorem]{Example}

\newtheorem{conjecture}[theorem]{Conjecture}

\def\le{\leqslant}
\def\ge{\geqslant}

\def\supp{\mathop{\rm supp}}
\def\rad{\mathop{\rm rad}}

\def\ev{\mathop{\rm ev}\nolimits}
\def\suchthat{\mathbin{\rm |}}
\def\and{\,\mathbin{\&}\,}

\def\X{\mathcal X}
\def\Y{\mathcal Y}
\def\s{{\mathbf s}}

\def\sp{{\ast}}

\def\c{\mathop{\mathbf c}}
\def\r{\mathop{\mathbf r}}

\renewcommand\phi{\varphi}

\def\={\equiv}

\renewcommand{\(}{\left(}
\renewcommand{\)}{\right)}

\def\vertexwithonearrow[#1#2]{
\put(0,0){\circle*{4}}
\put(0,-30){\vector(0,1){28}}
\put(-6,-18.5){$\scriptstyle#1$}
\put(4,-18.5){$\scriptstyle#2$}
}

\def\vertexwithtwoarrows[#1#2][#3#4]{
\put(0,0){\circle*{4}}
\put(-30,-30){\vector(1,1){28.6}}
\put(30,-30){\vector(-1,1){28.6}}
\put(-22,-15){$\scriptstyle#1$}
\put(-14.5,-21.5){$\scriptstyle#2$}
\put(17,-15){$\scriptstyle#3$}
\put(8,-21.5){$\scriptstyle#4$}
}

\makeatletter

\def\linefill@#1{\m@th\setboxz@h{$#1-$}\ht\z@\z@
    $#1\copy\z@\mkern-6mu\cleaders
    \hbox{$#1\mkern-2mu\box\z@\mkern-2mu$}\hfill
    \mkern-2mu\vphantom{\rightarrow}$}

  \def\rightarrowfill@#1{\m@th\setboxz@h{$#1-$}\ht\z@\z@
    $#1\copy\z@\mkern-6mu\cleaders
    \hbox{$#1\mkern-2mu\box\z@\mkern-2mu$}\hfill
    \mkern-6mu\mathord\rightarrow$}

\def\edge#1#2#3{
\global\bigaw@1.8pc
\setboxz@h{$\m@th\scriptstyle\;{#2}\;\;$}%
  \ifdim\wdz@>\bigaw@\global\bigaw@\wdz@\fi
#1\stackrel{#2}{\mathop{\hbox to \bigaw@{\linefill@\displaystyle}}}#3
}
\def\edgeright#1#2#3{
\global\bigaw@1.8pc
\setboxz@h{$\m@th\scriptstyle\;{#2}\;\;$}%
  \ifdim\wdz@>\bigaw@\global\bigaw@\wdz@\fi
#1\stackrel{#2}{\mathop{\hbox to \bigaw@{\rightarrowfill@\displaystyle}}}#3
}
\makeatother

\def\dotcup{\mathbin{\dot\cup}}

\title{On decomposition of Bott-Samelson sheaves}
\author{Vladimir Shchigolev}
\email{shchigolev\_vladimir@yahoo.com}

\begin{document}

\maketitle

\begin{abstract}
We give an exact algorithm to calculate (under some GKM-restriction) the matrix describing the embedding
$\B(\s)_x\subset\B(\s)^x$, where the first module is the costalk and the second one is the stalk at $x$
of a Bott-Samelson module (sheaf) $\B(\s)$. This allows us to calculate the first few
terms of the decomposition of $\B(\s)$ into a sum of indecomposable modules (sheaves) and to calculate the characters
of Braden--MacPherson sheaves in some previously unknown cases.
\end{abstract}


\section{Introduction}

In~\cite{Fiebig_sheaves_on_affine_Schubert_varieties} and~\cite{Fiebig_Lusztig’s_conjecture},
Peter Fiebig developed a connection between Lusztig's conjecture on the characters of
irreducible rational representations of reductive algebraic groups over a field~$\F$
of positive characteristic and the theory of $\F$-sheaves on moment graphs.
He showed that Lusztig's conjecture follows from the conjecture on the characters
of the Braden--MacPherson sheaves (with coefficients in $\F$) on an affine
moment graph (cf.~\cite[Conjecture 4.4]{Fiebig_Lusztig’s_conjecture}).

Fiebig showed in~\cite{Fiebig_An_upper_bound} that for every element $w$ of the affine Weyl group,
there exists some explicitly defined number $U(w)$ such that for all $\mathop{{\rm char}}\F>U(w)$,
the character of the Braden--MacPherson sheaf $\B(w)$ with coefficients in $\F$
is given by the corresponding Kazhdan-Lusztig element of the Hecke algebra.
This result is actually obtained by considering decompositions of Bott--Samelson modules into direct
sums of indecomposables. The parameter that governs this decomposition in~\cite{Fiebig_An_upper_bound}
is the Lefschetz datum. This approach however works only if $\mathop{{\rm char}}\F$ is bigger than
some number depending on the heights of roots and not only on the GKM-property.

In view of this, we introduce here a different parameter called {\it defect}
(Definition~\ref{definition:6}). Unlike Lefschetz datum, this parameter applies not to modules
but to certain sheaves on moment graphs that we call {\it projective} (Definition~\ref{definition:5}).
This notion is motivated by Jantzen's lectures~\cite{Jantzen_moment_graphs}, where he defines F-projective
sheaves\footnote{F stands for Fiebig} (\cite[Section~3.8]{Jantzen_moment_graphs}).
Whereas Jantzen considers only finite moment graphs in his lectures, all his definitions and
results apply to affine moment graphs as well if we additionally require the supports of sheaves to be finite.

One thing that we are very interested in is the possibility to apply translation functors
directly to moment graphs. The corresponding construction is given by
Fiebig~\cite[Section~2.9]{Fiebig_sheaves_on_affine_Schubert_varieties}.
As predicted by results from~\cite{Fiebig_The_combinatorics_of_Coxeter_categories},
translation functors should take projective sheaves to projective sheaves.
This turns out to be true (Theorem~\ref{theorem:8}) if we slightly modify the definition of $\vartheta^s_{on}$
(Section~\ref{functor_vartheta_on}) and impose some GKM-restriction.
Thus we can apply translation functor to the trivial sheaf repeatedly and
get the sheaf $\B(\s)$ ($\s$ is a sequence of simple reflections), which we call
the {\it Bott--Samelson sheaf} (Definition~\ref{definition:BS_sheaf}).
This sheaf is well-defined and projective under some GKM-restriction.
By~\cite[Proposition 3.12]{Jantzen_moment_graphs}, $\B(\s)$
decomposes into a direct sum of indecomposable projective sheaves
(Braden--MacPherson sheaves) and the defect of $\B(\s)$ tells us how exactly.
So we may consider the calculation of this defect (or proving its independence of
$\mathop{{\rm char}}\F$) our main problem.

To this end, we use the modules $\X(\s)$ constructed by Fiebig in~\cite[Section~6]{Fiebig_An_upper_bound},
which are isomorphic to Bott--Samelson modules (cf.~\cite[Proposition 6.14(2)]{Fiebig_An_upper_bound}).
We prove in this paper a different version of this isomorphism: the isomorphism $\Gamma(\B(\s))\stackrel{\sim}\to\X(\s)$
that induces the isomorphisms of stalks in a way compatible with restrictions (Theorem~\ref{theorem:3}).
The advantage is two-fold: we get the isomorphisms of costalks (Corollary~\ref{corollary:2})
and the result saying how elements of $\X(\s)$ behave with respect to edges of the underlying moment graph
(Corollary~\ref{corollary:3}).

The last result is very important to construct a basis of the stalk of a Bott-Samelson sheaf (module) at a fixed point $x$
(cf. Corollary~\ref{corollary:5}). This basis is constructed in terms of the tree $T(\s,x)$ that tells us how
the element $x$ is represented with respect to the Bruhat order as products of entries of subsequences of $\s$
(Section~\ref{trees}). Hence, we can get the matrix describing the inclusion of the costalk in the stalk
$\B(\s)_x\subset\B(\s)^x$ by the same method as~\cite{Fiebig_An_upper_bound} and previously in~\cite{Haerterich}.

Along with the costalk $\B(\s)_x$, which is the intersection of the kernels of the projections $\rho_{x,E}$
for all edges $E$ incident with $x$, one can consider the submodule $\B(\s)_{[x]}$,
which is the intersection of the same kernels but only for edges starting at $x$.
The defect at $x$ is defined by the zero degree part of the matrix $\Phi(\s,x)$ describing the inclusion
$\B(\s)_{[x]}\subset\B(\s)^x$ (Corollary~\ref{corollary:1}), which is given by an exact combinatorial algorithm
(Theorem~\ref{theorem:4}).

Although the matrix $\Phi(\s,x)$ is quite complicated, we still can make use of it
when the ungraded rank of the stalk $\B(\s)^x$ does not exceed $3$ (Section~\ref{low_rank_cases}).
Remarkably, we don't use any restrictions on the characteristic in our calculations other than the GKM-property.
In particular, we can write down the first terms of the decomposition of a Bott-Samelson module (sheaf) into a
direct sum of indecomposables as in Corollary~\ref{corollary:6}
(here $|I(\s)_x|$ is the number of subsequences of $\s$ giving $x$). This result and the known zero characteristic case
allow us to prove Fiebig's conjecture~\cite[Conjecture 4.4]{Fiebig_Lusztig’s_conjecture} on the characters of
the Braden--MacPherson sheaves $\B(w)$ for 3-reachable elements $w$ (Definition~\ref{definition:14})
of the affine Weyl group. I~believe that similar calculations are possible for other low ranks.
However, already the case where the ungraded rank $\B(\s)^x$ is 4 contains a large number of subcases.

\section{Notation and definitions}

\subsection{General} If we consider a field $\F$ of characteristic $p>0$, then we identify the residue
field $\Z/p\Z$ with the minimal subfield of $\F$. When there is now confusion about which field we take,
we write $\overline n$ for the residue class $n+p\Z\in\Z/p\Z\subset\F$.

We shall consider products $\prod_{i\in\mathcal I} M_i$ of sets, which consist of all functions $f$ on $\I$ such that
$f_i\in\M_i$ for any $i\in\mathcal I$. For any $\mathcal J\subset\mathcal I$, we consider the restriction
$f|_{\mathcal J}\in\prod_{i\in\mathcal J} M_i$. If $\mathcal I$ is finite, then we often write
$\bigoplus_{i\in\mathcal I} M_i$ instead of $\prod_{i\in\mathcal I} M_i$.

We write $A=B\dotcup C$ to say that set $A$ is the disjoint (i.e. $A\cap B=\emptyset$) union of sets $B$ and $C$.

\subsection{Poset topology}\label{poset_topology} Let $\V$ be a poset (partially ordered set). A subset $U\subset\V$ is called {\it open}
if $x\in U$ and $y\ge x$ imply $y\in U$. Obviously, the subsets
$$
\V_{\ge x}:=\{y\in\V\suchthat y\ge x\},\quad \V_{>x}:=\{y\in\V\suchthat y>x\}
$$
are open and $U=\bigcup_{x\in U}\V_{\ge x}$ for any open $U\subset\V$. Unions and intersections of open subsets
are also open, so we get a topology on $\V$. We write [X] for the closure of $X\subset\V$. Obviously, the subsets
$$
\V_{\le x}:=\{y\in\V\suchthat y\le x\},\quad \V_{<x}:=\{y\in\V\suchthat y<x\}
$$
are closed and $[X]=\bigcup_{x\in X}\V_{\le x}$. Note that any unions including infinite ones of closed subsets are closed
in our topology!

In the sequel, we always suppose that $\V_{\le x}$ is finite for any $x\in\V$.
Then the closure of any finite subset is also finite.

\subsection{Moment graphs}\label{moment_graphs} Throughout this paper $\F$ denotes a field of characteristic
distinct from 2. Let $\VF$ be a finite dimensional vector space over $\F$. Consider the symmetric algebra $S:=S(\VF)$ with $\Z$-grading such that elements of $\VF$
have degree $2$. For brevity, we shall always say in this paper ``graded'' meaning ``$\Z$-graded''.

\begin{definition}\label{definition:1}
A moment graph $\G$ is given by the following data:
\begin{enumerate}
\item An oriented graph $(\V,\E)$ with set of vertices $\V$ and set of edges $\E$.
\item A map $l:\E\to \VF\setminus\{0\}$ called the labelling.
\item A partial order $\le$ on $\V$ such that the following holds: if there is an edge from $x$ to $y$ then $x<y$.
\end{enumerate}
\end{definition}
We write $\edge x{}y$ if we want to say that $x$ and $y$ are connected by an edge (in any direction)
and $\edgeright x{}y$ if we want to say that there is an edge from $x$ to $y$.
If we want to specify the label of this edge we write $\edge x\alpha y$ and $\edgeright x\alpha y$ and
if we want to call this edge we add its name followed by $:$ on the left, e.g. $E:\edgeright x\alpha y$.

We say that $\widehat\G$ satisfies the {\it GKM-property} if the the labels $\alpha$ and $\beta$ of any two
edges having a common vertex are not proportional, that is, $\beta\notin\F\alpha$.

For any subset $\I\subset\V$, we define the {\it full moment subgraph $\G_\I$} of $\G$ by taking $\I$ for the set of vertices
and considering only those edges of $\G$ that connect elements of $\I$.
We often abbreviate
$\G_{\le x}:=\G_{\V_{\le x}}$.

We keep this notation throughout the paper --- $\V$ usually denotes the set of vertices,\linebreak $\E$ the set of edges
and $l$ the labelling.

\subsection{Finite root system} Let $V$ be a finite dimensional $\mathbb Q$-vector space and
let $R\subset V$ be an irreducible, reduced, finite root system.
We denote by $V^*$ the $\mathbb Q$-space of all $\mathbb Q$-linear functions $f:V\to\mathbb Q$ and write
$\<\cdot,\cdot\>:V\times V^*\to\mathbb Q$ for the natural pairing: $\<v,f\>:=f(v)$.
For $\alpha\in R$, we denote by $\alpha^\vee\in V^*$
the corresponding coroot. Let $R^\vee := \{\alpha^\vee\,\suchthat\,\alpha\in R\}$ be the dual root system and
$$
X:=\{\lm\in V\suchthat\<\lm,\alpha^\vee\>\in\Z\text{ for all }\alpha\in R\}
$$
be the weight lattice. We fix a decomposition into positive and negative roots $R=R^+\dotcup R^-$ and the corresponding
set $\Pi$ of simple roots.

\subsection{Affine Weyl group}\label{Affine_Weyl_group} For $\alpha\in R$ and $n\in\Z$, we define the affine reflection \linebreak$s_{\alpha,n}:V^*\to V^*$ by
$$
s_{\alpha,n}(v)=v-(\<\alpha,v\>-n)\,\alpha^\vee.
$$
Note that $s_\alpha:=s_{\alpha,0}$ is the usual reflection and
$s_{\alpha,n}=s_{\beta,m}$ if and only if either $\alpha=\beta$ and $n=m$ or $\alpha=-\beta$ and $n=-m$.

The affine Weyl group $\widehat\W$ is the subgroup of transformations of $V^*$ that is generated by all $s_{\alpha,n}$.
We can linearize the above affine action by adding an additional dimension. Set $\widehat V:=V\oplus\mathbb Q$ and
$\widehat V^*:=V^*\oplus\mathbb Q$. This pair of spaces have the pairing $\<\cdot,\cdot\>:\widehat V\times \widehat V^*\to\mathbb Q$
given by $\<(\lm,m),(v,n)\>=\<\lm,v\>+mn$. We let $s_{\alpha,n}$ act linearly on $\widehat V^*$ by
$$
s_{\alpha,n}(v,m)=(v-(\<\alpha,v\>-mn)\,\alpha^\vee,m).
$$
This action extends to the linear action of the whole $\widehat\W$.
The first level space $\widehat V^*_1:=\{(v,1)\suchthat v\in V^*\}$ is stabilized by this action of $\widehat\W$ and
we have
$$
w(v,1)=(wv,1)
$$
for any $w\in\widehat W$ and $v\in V^*$.

We also have the dual action of $\widehat\W$ on $\widehat V$ that is characterized by the duality rule
$$
\<w(\alpha),v\>=\<\alpha,w^{-1}(v)\>
$$
for any $\alpha\in\widehat V$ and $v\in\widehat V^*$. Explicitly it is given by
$$
w(0,\nu)=(0,\nu),\quad s_{\alpha,n}(\lm,0)=(s_\alpha(\lm),n\<\lm,\alpha^\vee\>).
$$
In the next section, we shall give a friendlier interpretation of this action.

\subsection{Affine root system}\label{Affine_root_system} The hyperplane of fixed points of $s_{\alpha,n}$ in $\widehat V^*$ is
$$
\widehat H_{\alpha,n}:=\{(v,m)\suchthat\<\alpha,v\>=mn\}.
$$
We can rewrite the defining equation of this hyperplane as $\<(\alpha,-n),\hat v\>=0$.
We define $s_{(\alpha,-n)}:=s_{\alpha,n}$ and call the pair $(\alpha,-n)$ an {\it affine root}.
The set of all affine roots is denoted~$\widehat R$.
Clearly, $s_\gamma=s_\tau$ for two affine roots $\gamma$ and $\tau$ if and only if $\gamma=\pm\,\tau$
(cf. Section~\ref{Affine_Weyl_group}).

Now let us choose the system $\widehat R^+\subset\widehat R$ of positive roots as follows
$$
\widehat R^+:=\{(\alpha,n)\suchthat\alpha\in R,n>0\}\dotcup\{(\alpha,0)\suchthat\alpha\in R^+\}.
$$
We set $\widehat R^-:=-\widehat R^+$ and call it the set of {\it negative roots}.
Clearly $\widehat R=\widehat R^-\dotcup \widehat R^+$.
We write $\gamma>0$ (resp. $\gamma<0$) to say that $\gamma\in R^+$ or $\gamma\in\widehat R^+$
(resp. $\gamma\in R^-$ or $\gamma\in\widehat R^-$).
The set of {\it simple affine roots} is
$$
\widehat\Pi:=\{(\alpha,0)\suchthat\alpha\in\Pi\}\dotcup\{(-\tilde\alpha,1)\},
$$
where $\tilde\alpha$ is the highest root of $R$. Then any root of $\widehat R^+$ is a sum (possibly with repetitions)
of roots of $\widehat\Pi$. The corresponding set of {\it simple reflections} is
$$
\widehat\S:=\{s_\alpha\suchthat\alpha\in\Pi\}\dotcup\{s_{\tilde\alpha,1}\}.
$$
Then $\widehat\S$ is a set of generators of $\widehat\W$ and $(\widehat\W,\widehat\S)$ is a Coxeter system.
Therefore we have the length function $\ell$ and the Bruhat order $\le$ on $\widehat\W$
(cf.~\cite[Chapter~2]{Bjorner_Brenti}) .



\begin{proposition}[\text{\cite[Proposition 4.4(c)]{Humphreys_reflection_groups}}]\label{proposition:0}
For any $w\in\widehat\W$ and $\gamma\in\widehat\Pi$, we have $\ell(ws_\gamma)=\ell(w)+1$ if and only if $w(\gamma)>0$.
\end{proposition}
\noindent
Along with the affine roots $\widehat R$, we can consider the affine weights $\widehat X:=X\oplus\Z$.

Finally, note that the above action of $\widehat\W$ on $\widehat V$
satisfies
\begin{equation}\label{eq:0}
ws_\alpha w^{-1}=s_{w(\alpha)}
\end{equation}
for any $\alpha\in\widehat R$ and $w\in\widehat\W$ and is given by
$$
s_\alpha(\beta)=\beta-\<\{\beta\},\{\alpha\}^\vee\>\alpha
$$
for any $\alpha\in\widehat R$ and $\beta\in\widehat V$,
where $\{\cdot\}:\widehat V\to V$ is the map $(\lm,n)\mapsto\lm$.
An easy observation is that $\{s_\alpha(\beta)\}=s_{\{\alpha\}}\{\beta\}$ for any $\alpha,\beta\in R$.
Hence $\<\{w(\alpha)\},\{w(\beta)\}^\vee\>=\<\{\alpha\},\{\beta\}^\vee\>$
for any $\alpha,\beta\in R$ and $w\in\widehat\W$.

Using the abbreviation $\<\alpha,\beta\>':=\<\{\alpha\},\{\beta\}^\vee\>$, we
shall rewrite the above formulas as
\begin{equation}\label{eq:1}
s_\alpha(\beta)=\beta-\<\beta,\alpha\>'\alpha
\end{equation}
and
\begin{equation}\label{eq:0.25}
\<w(\alpha),w(\beta)\>'=\<\alpha,\beta\>'.
\end{equation}
\subsection{Properties of the Bruhat order} We outline the main properties of this order (valid for any Coxeter group)
that we shall use in the sequel.
Concise and self-contained proofs can be found in~\cite{Bjorner_Brenti}.
Consider the set of reflections $\widehat\T:=\{wsw^{-1}\suchthat w\in\widehat\W\}$.

\begin{proposition}[Subword Property]\label{proposition:subword_property}
Let $w=s_1\cdots s_q$ be a reduced expression. We have $u\le w$ in the Bruhat order if and only if there exists a reduced
expression $u=s_{i_1}\cdots s_{i_k}$ for some $1\le i_1<\cdots<i_k\le q$.
\end{proposition}

\begin{proposition}[Exchange Property]\label{proposition:exchange_property}
Suppose $w=s_1\cdots s_k$, where $s_i\in\widehat\S$ and $t\in\widehat\T$.
If $\ell(wt)<\ell(w)$, then $wt=s_1\cdots\widehat s_i\cdots s_k$ for some $i=1,\ldots,k$.
\end{proposition}

\begin{proposition}[Lifting property]\label{proposition:lifting_property}
Suppose $u,v\in\widehat\W$ and $s\in\widehat S$ be such that $u<w$
and $ws<w$ and $us>u$. Then $us\le w$ and $u\le ws$.
\end{proposition}

\begin{corollary}\label{corollary:0}
Suppose $u,v\in\widehat\W$ and $s\in\widehat S$ be such that $u<w$.
\begin{enumerate}
\itemsep=2pt
\item\label{corollary:0:part:1} If $ws<w$ then $us\le w$.
\item\label{corollary:0:part:2} If $us>u$ then $u\le ws$.
\end{enumerate}
\end{corollary}

\begin{proposition}[Deletion Property]\label{proposition:deletion_property}
If $w=s_1\cdots s_k$, where $s_i\in\widehat\S$ and  $\ell(w)<k$, then $w = s_1\cdots\widehat s_i\cdots\widehat s_j\cdots s_k$
for some $1\le i<j\le k$.
\end{proposition}

\subsection{Associated moment graph}\label{associated_moment_graph} 
We define the moment graph $\widehat\G$ with set of vertices $\widehat\W$ and the Bruhat order on it.
There is an edge from $x$ to $y$ if and only if $x<y$ and $y=s_{\alpha}x$ for some $\alpha\in\widehat R^+$.
We endow this edge with the (nonzero) label $\bar\alpha:=\alpha\otimes1\in\widehat X\otimes_\Z\F=:\VF$.
We define the graded symmetric algebra $S:=S(\VF)$ as in Section~\ref{moment_graphs}.

We can shift edges as follows. If we have an edge $E:\edge x\alpha y$ and $s\in\widehat\W$,
then there is also an edge $Es:\edge{xs}\alpha{ys}$. Note that this operation does not preserve the direction of edges.

For $I\subset\widehat\S$, we can consider the parabolic subgroup $\widehat\W_I$, which is the subgroup of $\widehat\W$
generated by $I$. It is known that $(\widehat\W_I,I)$ is a Coxeter system with the same length function
(cf.~\cite[Proposition 2.4.1]{Bjorner_Brenti}). We introduce the quotient moment graph $\widehat\G^I$ as follows.
We take $\widehat\W^I:=\widehat\W/\widehat\W_I$ (left cosets) for the set of vertices with the induced order.
This means the following: $x<y$, where $x,y\in\widehat\W^I$, if and only if $x\ne y$ and
there are some coset representatives $u\in x$ and $v\in y$ such that $u<v$.
By~\cite[Proposition~2.5.1]{Bjorner_Brenti} this order is well defined and we can take
elements of minimal length in $x$ and $y$ for $u$ and $v$ respectively.

To get the edges of $\widehat\G^I$, we virtually repeat the definition of the order on $\widehat\W^I$:
there is an edge $\edgeright x{}y$, where $x,y\in\widehat\W^I$, if and only if $x\ne y$ and
there are some coset representatives $u\in x$ and $v\in y$ such that $\edgeright u{}v$.
This edge receives the same labelling as $\edgeright u{}v$.
In other words, there is an edge $\edgeright x{\bar\alpha}y$ if and only if $x<y$,
$\alpha\in\widehat R^+$ and $y=s_\alpha x$.

We are especially interested in the case $I=\{e,s\}$ for some $s\in\widehat\S$.
We abbreviate $\widehat\W^s:=\widehat\W^{\{e,s\}}$ and $\widehat\G^s:=\widehat\G^{\{e,s\}}$.
We usually write $\bar x=\{x,xs\}$ for the image of $x\in\widehat\W$ in $\widehat\W^s$,
when it is clear which $s$ we mean.
We also set $\overline\Omega:=\{\bar x\,|\,x\in \Omega\}$ for any subset $\Omega\subset\widehat\W$.
We denote by $\pi_s$ the natural projection $\widehat\W\to\widehat\W^s$ (i.e. $\pi_s(x)=\{x,xs\}=\bar x$).

We also denote by $\bar E:\edgeright{\bar x}{}{\bar y}$ the edge of $\widehat\G^s$
corresponding to an edge $E:\edgeright x{}y$.

\section{Projective sheaves}

\subsection{Sheaves} A {\it homomorphism} $f:M\to N$ of graded spaces,
is a homogeneous linear map of degree zero, that is a linear map satisfying $f(M_i)\subset N_i$ for any $i\in\Z$.
Similarly, a {\it homomorphism of graded $S$-modules} is a usual (ungraded) homomorphism of $S$-modules
that is a homomorphism of graded spaces.
We define now the main object of our study.

\begin{definition}\label{definition:2}
A sheaf $\M$ on the moment graph $\G$ is given by the following data:
\begin{enumerate}
\itemsep=2pt
\item\label{definition:2:part:1} a graded $S$-module $\M^x$ for each vertex $x$;
\item\label{definition:2:part:2} a graded $S$-module $\M^E$ for each vertex $E$ such that $l(E)\,\M^E=0$;
\item\label{definition:2:part:3} a homomorphism $\rho_{x,E}:\M^x\to\M^E$ {\rm(}restriction{\rm)} of graded $S$-modules for each edge~$E$ and vertex $x$ lying on $E$.
\end{enumerate}
\end{definition}
\noindent
We sometimes write $\rho_{x,E}^\M$ to underline that this restriction map is for $\M$.

The {\it support} of $\M$ is $\supp\M:=\{x\in\V\suchthat\M^x\ne0\}$. We call $\M^x$ the {\it stalk} of $\M$ at $x$.
Let us define $\E^x\subset\E$ to be the set of edges that contain the vertex $x$
and $\E^{\delta x}\subset\E$ to be the set of edges that start at $x$.
We denote by $\V^{\delta x}$ the set of ends of all edges of $\E^{\delta x}$.
We define the following graded submodules of $\M^x$:
$$
\arraycolsep=2pt
\begin{array}{rcl}
\M_x&:=&\{m\in\M^x\suchthat\rho_{x,E}(m)=0\text{ for all }E\in\E^x\},\\[8pt]
\M_{[x]}&:=&\{m\in\M^x\suchthat\rho_{x,E}(m)=0\text{ for all }E\in\E^{\delta x}\}.
\end{array}
$$
Clearly, $\M_x\subset\M_{[x]}\subset\M^x$.

If $M$ is a graded module and $n\in\Z$, then we denote by $M\<n\>$ the graded module with shifted grading
$M\<n\>_i=M_{i+n}$. The similar notation applies to maps: if $\pi:M\to N$ is a map from one graded module
to another, then $\pi\<n\>$ denotes the map from $M\<n\>$ to $N\<n\>$ acting as $\pi$ elementwise.
Shifts are applied to sheaves $\M$ as well by applying $\cdot\<n\>$ to all modules $\M^x$ and $\M^E$
and restriction maps $\rho_{x,E}$. We denote this new sheaf by $\M\<n\>$.

We can also form the direct sum $\M\oplus\N$ of two sheaves $\M$ and $\N$ on $\G$ by
$$
(\M\oplus\N)^x:=\M^x\oplus\N^x,\quad(\M\oplus\N)^E:=\M^E\oplus\N^E,\quad\rho_{x,E}^{\M\oplus\N}:=\rho_{x,E}^\M\oplus\rho_{x,E}^\N.
$$
This definition obviously implies $(\M\oplus\N)_x=\M_x\oplus\N_x$ and $(\M\oplus\N)_{[x]}=\M_{[x]}\oplus\N_{[x]}$.
The zero element with respect to this summation is the {\it zero sheaf} --- the sheaf $\M$ such that
$\M^x=\M^E=0$ for all vertices $x$ and edges $E$.

The sheaves on $\G$ form a category. A morphism $f:\M\to\N$ between two sheaves
is a pair of families $\left(\{f_x\}_{x\in\V},\{f_E\}_{E\in\E}\right)$, where $f_x:\M^x\to\N^x$
and $f_E:\M^E\to\N^E$ are homomorphisms of graded $S$-modules such that the diagram
$$
\begin{CD}
\M^x @>f_x>> \N^x\\
@V\rho^\M_{x,E}VV                   @VV\rho^\N_{x,E}V   \\
\M^E @>f_E>>            \N^E
\end{CD}
$$
whenever $x$ lies on $E$. If all homomorphisms $f_x$ and $f_E$ are isomorphisms, then we call $f:\M\to\N$ an isomorphism of sheaves.

For a subset $\mathcal I\subset\V$ and a sheaf $\M$ on $\G$, we can consider the restriction $\M_\I$
of $\G$ to $\G_\I$.

\subsection{Category $\C$} We outline here the main points of~\cite[Section 2.21]{Jantzen_moment_graphs}.
The objects of category $\C$ are finitely generated graded $S$-modules and the morphisms are
homomorphism of graded $S$-modules.
We denote $\m:=\bigoplus_{i>0}S_i$ (the maximal graded ideal of $S$).
Note that $S/\m S\cong S_0=\F$ and so any module $M/\m M$ is an $\F$-linear space.
For any $M$ in $\C$, we denote by $\rad_\C M$ the intersection of all maximal graded submodules of $M$.

A {\it free object} in $\C$ is a finite sum $S\<r_1\>\oplus\cdots\oplus S\<r_n\>$.
A {\it projective cover} of an object $M$ in $\C$ is a pair $(P,\pi)$,
where $P$ is a free object in $\C$ and $\pi:P\to M$ a surjective homomorphism in $\C$
such that $\pi(N)\ne M$ for any graded submodule $N$ of $P$ with $N\ne P$.

Note that a projective cover exists for any module $M$ in $\C$ and is defined uniquely up to an isomorphism.
In this way, one proves that any projective object in $\C$ is free.

\begin{proposition}[\text{\cite[Lemma 2.21]{Jantzen_moment_graphs}}]\label{proposition:1}
{\rm(a)} We have $\rad_\C=\m M$ for any $M$ in $\C$.\\[-9pt]

\noindent
{\rm(b)} Let $\pi:P\to M$ be a homomorphism in $\C$ such that $P$ is free.
Then $(P,\pi)$ is a projective cover of $M$ in $\C$ if and only if the induced map
$\bar\pi: P/\m P\to M/\m M$ is an isomorphism.
\end{proposition}

\subsection{Defect} We consider here the situation when we have a surjective homomorphism $\pi:P\to M$ in $\C$
that is not necessarily a projective cover. Obviously, the induced homomorphism $\bar\pi: P/\m P\to M/\m M$
is also surjective but not an isomorphism in general. Therefore it is worth considering its kernel $\ker\bar\pi$.
If it is zero and $P$ is free, then $(P,\pi)$ is a projective cover by Proposition~\ref{proposition:1}(b).

\begin{lemma}\label{lemma:1} Let $\pi:P\to M$ be a surjective homomorphism in $\C$.
Consider the induced map $\bar\pi:P/\m P\to M/\m M$. Then $(\ker\bar\pi)_i=(\ker\pi)_i+\m P$ for all $i\in\Z$.
\end{lemma}
\begin{proof} Any element of $(\ker\bar\pi)_i$ can be written as $x+\m P$, where $x\in P_i$.
Then $\pi(x)\in\m M$. Let us write $\pi(x)=r_1y_1+\cdots+r_ny_n$ for some homogeneous $r_j\in\m$ and $y_j\in M$.
Since $\pi$ is surjective, there is for each $y_j$ an element $z_j\in P$ of the same degree such that $\pi(z_j)=y_j$.
Hence $x-r_1z_1-\cdots-r_nz_n\in(\ker\pi)_i$ and $x+\m P\in(\ker\pi)_i+\m P$.
The inverse inclusion is obvious.
\end{proof}

We fix a variable $v$ and consider the ring $\Z[v,v^{-1}]$.
\begin{definition}\label{definition:3}
Let $\pi:P\to M$ be a surjective homomorphism in $\C$. The defect of $\pi$ is the following element of $\Z[v,v^{-1}]${\rm:}
$$
\d(\pi):=\sum_{i\in\Z}\dim_\F(\ker\bar\pi)_i\,v^{-i}.
$$
\end{definition}
We have $\d(\pi\<n\>)=v^n\d(\pi)$. Lemma~\ref{lemma:1} prompts the following way to calculate defects.

\begin{corollary}\label{corollary:1}
Let $\pi:P\to M$ be a surjective homomorphism in $\C$ such that $P$ and $\ker\pi$ are free.
Suppose that $\{v^{(n)}_j\}_{n\in\Z,1\le j\le k_n}$ and $\{u^{(n)}_i\}_{n\in\Z,1\le i\le l_n}$
are bases of $P$ and $\ker\pi$ respectively labelled in such a way that $v^{(n)}_j$ and $u^{(n)}_i$ have degree $n$.
Let
$$
u^{(n)}_i=\sum\nolimits_{m\in\Z, 1\le j\le k_m}a_{j,i}^{(m,n)}v^{(m)}_j
$$
for corresponding homogeneous $a_{j,i}^{(m,n)}\in S$. For each $n\in\Z$, we denote by $A^{(n)}$
the $k_n\times l_n$-ma\-trix whose $ji^{\text{th}}$ entry is $a_{j,i}^{(n,n)}\in\F$.
Then $\d(\pi)=\sum_{n\in\Z}\rk_\F A^{(n)}v^{-n}$.
\end{corollary}

\subsection{Sections of sheaves}\label{Sections_of_sheaves} The {\it space of global sections} of a sheaf $\M$ is
$$
\Gamma(\M):=\Biggl\{m\in\prod_{x\in\V}\M^x\;\biggl|\;\rho_{x,E}(m_x)=\rho_{y,E}(m_y)\text{ for all edges }E:\edge x{}y\Biggl\}.
$$
For a subspace $\mathcal I\subset\V$ we define the {\it space of sections over $\I$} by
$$
\Gamma(\mathcal I,\M):=\Biggl\{m\in\prod_{x\in\mathcal I}\M^x\;\biggl|\;\rho_{x,E}(m_x)=\rho_{y,E}(m_y)\text{ for all edges }E:\edge x{}y\text{ with }x,y\in\mathcal I\Biggl\}.
$$
If ${\mathcal I\,}'\subset\mathcal I$, the we have the natural restriction map $\Gamma(\mathcal I,\M)\to\Gamma({\mathcal I\,}',\M)$
We have the following sheaf property of sections.

\begin{lemma}\label{lemma:1.5}
Let $\mathcal I\subset\V$ be an open subset and $\mathcal I=\bigcup_{j\in J}{\mathcal I}_i$ be an open covering.
Then $m\in\prod_{x\in\mathcal I}\M^x$ is a section of $\M$ if and only if $m|_{{\mathcal I}_j}$ is a section of $\M$ for every $j\in J$.
\end{lemma}
\begin{proof}
Let $E:\edgeright x{}y$ be an edge such that $x,y\in\mathcal I$. Then there is an index $j\in J$
such that $x\in{\mathcal I}_j$. As ${\mathcal I}_j$ is open,
we get $y\in{\mathcal I}_j$, whence $\rho_{x,E}(m_x)=\rho_{y,E}(m_y)$.
\end{proof}

We often consider the case where $\mathcal I=\{x,y\}$ is a two-element set.
Then we write $(a,b)$ for the element of $\M^x\oplus\M^y$
whose value at $x$ is $a$ and at $y$ is $b$.
To stipulate that the first position corresponds to $x$ and the second to $y$,
we write the set of sections as $\Gamma(\{x,y\},\M)$.

We consider also the following important special case. The {\it structure sheaf} $\mathscr Z$ is the sheaf
defined by
\begin{itemize}
\item $\mathscr Z^x:=S$ for any vertex $x$;
\item $\mathscr Z^E:=S/l(E)S$ for any edge $E$;
\item every $\rho_{x,E}$ is the natural projection $S\to S/l(E)$.
\end{itemize}
We set $\mathcal Z:=\Gamma(\mathscr Z)$. One checks elementarily that $\mathcal Z$ is a
commutative, associative, unital $S$-algebras with respect to coordinatewise addition and multiplication.
We call $\mathcal Z$ the {\it struc\-ture algebra} of $\widehat\G$.
Clearly, $\Gamma(\M)$ is a $\mathcal Z$-module under componentwise action.

We rewrite the definition of $\mathcal Z$ in the case of the moment graph $\widehat\G$ defined in Section~\ref{associated_moment_graph}
as follows:
$$
\mathcal Z=\Biggl\{z\in\prod_{x\in\widehat\W}S\;\biggl|\; z_x\=z_{s_\alpha x}\!\!\!\!\!\pmod{\bar\alpha}\text{ for all }x\in\widehat\W\text{ and }\alpha\in\widehat R\Biggl\}.
$$
The most prominent elements of $\mathcal Z$ are elements $c^\lambda$, where $\lambda\in\widehat X$, defined by
$c^\lambda_x:=\overline{x(\lambda)}$ for any $x\in\widehat W$. We have $c^\alpha\in\mathcal Z$ by~(\ref{eq:1}).

We shall also use the {\it local structure algebra} $\mathcal Z(\I):=\Gamma(\mathscr Z_\I)$ or in a different form
$$
\mathcal Z(\I)=\Biggl\{z\in\prod_{x\in\I}S\;\biggl|\; z_x\=z_{s_\alpha x}\!\!\!\!\!\pmod{\bar\alpha}\text{ for all }x\in\I\text{ and }\alpha\in\widehat R\text{ with }s_\alpha x\in\I\Biggl\}.
$$
For any sheaf $\M$ on $\G$ such that $\supp\M\subset\I$,
the space of global sections $\Gamma(\M)$ is a $\mathcal Z(\I)$-module.
If we make this $\mathcal Z(\I)$-module into a $\mathcal Z(\mathcal I')$-module for some $\mathcal I'\supset\mathcal I$,
using the natural homomorphism $\mathcal Z(\mathcal I')\to\mathcal Z(\I)$,
then we obviously obtain the usual action of $\mathcal Z(\mathcal I')$ on $\Gamma(\M)$.


\subsection{Conditions on sheaves} We consider here sheaves satisfying various properties.

\begin{definition}\label{definition:4}

\begin{enumerate}
\itemsep=3pt
\item\label{gbgs} A sheaf $\M$ is called generated by global sections if for any $x\in\V$ and $a\in\M^x$,
      there exists a section $m\in\Gamma(\M)$ with $m_x=a$.
\item\label{flabby} A sheaf $\M$ is called flabby if the natural restriction $\Gamma(\mathcal I,\M)\to\Gamma({\mathcal I\,}',\M)$
      is surjective for any open subsets ${\mathcal I\,}'\subset\mathcal I\subset\V$.
\end{enumerate}
\end{definition}

Suppose that $\supp\M$ is finite. Then $\M$ is flabby if and only if the natural restriction
$\Gamma(\V_{\ge x},\M)\to\Gamma(\V_{>x},\M)$ is surjective for any $x\in\V$.
Indeed, suppose the last statement is true and ${\mathcal I\,}'\subset\mathcal I$ are open subsets of $\V$.
Take a section $m\in\Gamma({\mathcal I\,}',\M)$. By Lemma~\ref{lemma:1.5}, we can extend it to
${\mathcal I\,}'':={\mathcal I\,}'\cup(\mathcal I\setminus[\supp\M])$ by setting $u_y:=0$
for all $y\in\mathcal I\setminus[\supp\M]$.

Note that ${\mathcal I\,}'\subset{\mathcal I\,}''\subset\mathcal I$ and $\mathcal I\setminus{\mathcal I\,}''\subset[\supp\M]$.
Since $\supp\M$ is finite, its closure $[\supp\M]$ is finite as we stipulated in Section~\ref{poset_topology}.
So $\mathcal I\setminus{\mathcal I\,}''$ is also finite.
Choose a maximal (with respect to inclusion) open subset $\mathcal U\subset\V$
such that ${\mathcal I\,}''\subset\mathcal U\subset\mathcal I$ and $m$ is extendable to $\mathcal U$.
If $\mathcal U\ne\mathcal I$, then we take for $x$ the maximal element of $\mathcal I\setminus\mathcal U$.
We have $\V_{>x}\subset\mathcal U$ and we can extend $m$ over $x$, which contradicts the maximality
of $\mathcal U$. This argument is given in~\cite[Lemma~3.2]{Jantzen_moment_graphs}, where the set $\V$
is however supposed to be finite.

Now we define the class of sheaves most important to us. This definition is taken from~\cite[Section~3.8]{Jantzen_moment_graphs}
to which we added a finiteness condition.

\begin{definition}\label{definition:5}
A sheaf $\PP$ is called projective if the following conditions are satisfied:
{
\renewcommand{\labelenumi}{{\rm \theenumi}}
\renewcommand{\theenumi}{{\rm(P\arabic{enumi})}}
\begin{enumerate}
\itemsep=2pt
\item\label{proj:1} $\supp\PP$ is finite;
\item\label{proj:2} each $\PP^x$ is a free module in $\C$;
\item\label{proj:3} for any edge $E:\edgeright x\alpha y$, the restriction $\rho_{y,E}$ is surjective with kernel $\alpha\PP^y$.
\item\label{proj:4} $\PP$ is generated by global sections and flabby.
\end{enumerate}}

\end{definition}

\begin{lemma}\label{lemma:2}
Any restriction $\rho_{x,E}:\PP^x\to\PP^E$ for a projective sheaf $\PP$ is surjective.
\end{lemma}
\begin{proof}
Let $\PP$ be a projective sheaf. By~\ref{proj:3}, it suffices to consider the case of an edge
$E:x\stackrel{\alpha}\longrightarrow y$. Take an arbitrary $b\in\PP^E$. By~\ref{proj:3},
there is some $a\in\PP^y$ such that $\rho_{y,E}(a)=b$.
Since $\PP$ is generated by global section, there is some $m\in\Gamma(\PP)$ with $m_y=a$.
By definition of a global section, we have $\rho_{x,E}(m_x)=\rho_{y,E}(m_y)=\rho_{y,E}(a)=b$.
\end{proof}

\begin{proposition}[\text{\cite[Lemma 2.7]{Fiebig_An_upper_bound}}]\label{proposition:7}
Let $\PP$ be a projective sheave on $\G$ and $x\in\W$.
Suppose that $\G_{\supp\PP}$ satisfies the GKM-property. Denote by $\alpha_1,\ldots,\alpha_l$
the labels of all edges that end at $x$. Then $\PP_x=\alpha_1\ldots\alpha_l\PP_{[x]}$.
\end{proposition}
\begin{proof} The inclusion $\alpha_1\ldots\alpha_l\PP_{[x]}\subset\PP_x$ is obvious and holds without any GKM-restriction.
To prove the inverse inclusion, it suffices to consider the case $\PP^x\ne0$.
Take any $m\in\PP^x$.
By Lemma~\ref{lemma:2}, all edges ending at $x$ are edges of $\G_{\supp\PP}$.
Therefore their labels $\alpha_1,\ldots,\alpha_l$ are pairwise not proportional and by property~\ref{proj:3} of projective sheaves,
there exists some $m'\in\PP^x$ such that $m=\alpha_1\cdots\alpha_lm'$.
Now take an edge $E:\edgeright x\beta y$. If $\PP^y=0$ then again by~\ref{proj:3}, we get $\PP^E=0$ and $\rho_{x,E}(m')=0$.
If $\PP^y\ne0$ then $E$ is an edge of $\G_{\supp\PP}$ and $\alpha_1,\ldots,\alpha_l,\beta$ are all pairwise
not proportional. Thus $0=\rho_{x,E}(m)=\alpha_1\cdots\alpha_l\rho_{x,E}(m')$ and we can cancel out
$\alpha_1\cdots\alpha_l$ in $\PP^E\cong\PP^y/\beta\PP^y$. Hence $\rho_{x,E}(m')=0$.
\end{proof}

Consider the module $\M^{\delta x}\subset\prod_{E\in\E^{\delta x}}\M^E$ consisting
of all $r\in\prod_{E\in\E^{\delta x}}\M^E$ having the form
$r_E=\rho_{y,E}(m_y)$ for all edges $E:\edgeright x{}y$, where $m\in\Gamma(\V_{>x},\M)$.
Moreover, consider the projection $\rho_{x,\delta x}:\M^x\to\prod_{E\in\E^{\delta x}}\M^E$, taking
an element $a\in\M^x$ to $r\in\prod_{E\in\E^{\delta x}}\M^E$, where
$r_E=\rho_{x,E}(a)$ for all edges $E:\edgeright x{}y$, .
Obviously, $\ker\rho_{x,\delta x}=\M_{[x]}$.

Directly from our definitions and the remark above we get the following simple fact.
\begin{lemma}[\text{\cite[Lemma~3.3]{Jantzen_moment_graphs}}]\label{lemma:3}
Let $\M$ be a sheaf on $\G$ with finite support. The sheaf $\M$ is flabby and generated by global sections if and only if
$\M^{\delta x}=\rho_{x,\delta x}(\M^x)$ for any $x\in\V$.

\end{lemma}

\begin{definition}\label{definition:6}
Let $\PP$ be a projective sheaf. Consider the $\Z[v,v^{-1}]$-module $\Z[v,v^{-1}]\V$
of all formal linear combinations $f_1x_1+\cdots+f_nx_n$, where $f_i\in\Z[v,v^{-1}]$ and $x_i\in\V$.
The defect of $\PP$ is the following element of $\Z[v,v^{-1}]\V$:
$$
\d(\PP):=\sum_{x\in\V}\d(\rho_{x,\delta x})x.
$$
\end{definition}
\vspace{-5pt}
\noindent
Here $\rho_{x,\delta x}$ is considered as surjective homomorphism $\PP^x\to\PP^{\delta x}$
and $\d$ in the right-hand side denotes the defect of this map (cf. Definition~\ref{definition:3}).

We obviously have $\d(\PP\<r\>)=v^r\d(\PP)$ and $\d(\PP\oplus\PP')=\d(\PP)+\d(\PP')$.

%

\subsection{Braden–-MacPherson sheaves} The following proposition is actually a definition
(cf.~\cite{BM} and also~\cite[Theorem 4.2]{Fiebig_Lusztig’s_conjecture}
and~\cite[Section 3.5]{Jantzen_moment_graphs}).

\begin{proposition}\label{proposition:2}
For each $w\in\V$, there is a unique up to isomorphism sheaf $\B(w)$ on~$\G$, called the Braden–-MacPherson sheaf, with the following properties:
\begin{enumerate}
\itemsep=2pt
\item\label{proposition:2:part:1} $\supp\B(w)\subset\V_{\le w}$ and $\B(w)^w\cong S$.
\item\label{proposition:2:part:2} For any edge $E:\edgeright x\alpha y$, the restriction $\rho_{y,E}$ is surjective with kernel $\alpha\B(w)^y$.
\item\label{proposition:2:part:3} For any $x\in\V$, the image of $\rho_{x,\delta x}$ is $\B(w)^{\delta x}$. For any $x\in\V\setminus\{w\}$,
      the homomorphism $\rho_{x,\delta x}:\B(w)^x\to\B(w)^{\delta x}$ is a projective cover in $\C$.
\end{enumerate}
\end{proposition}

By Lemma~\ref{lemma:3}, Braden-MacPherson sheaves are projective and by~\cite[Proposition~3.5(c)]{Jantzen_moment_graphs}
indecomposable. We clearly have $\d(\B(w))=w$ (the last statement of part~\ref{proposition:2:part:3}
in the above proposition is wrong for $x=w$).


\begin{proposition}[\text{\cite[Proposition~3.12]{Jantzen_moment_graphs}}]\label{proposition:3}
For any be a projective sheaf $\PP$, there is an isomorphism of sheaves
$$
\PP\cong\B(z_1)\<r_1\>\oplus\B(z_2)\<r_2\>\oplus\cdots\oplus\B(z_s)\<r_s\>
$$
with suitable $z_1,z_2,\ldots,z_s\in\V$ and $r_1,r_2,\ldots, r_s\in\Z$.
The pairs $(z_i, r_i)$ are determined uniquely up to order by $\PP$.
\end{proposition}
\begin{proof} As $\V_{\le x}$ is finite for any $x\in\V$, the set $[\supp\PP]$ is finite.
Then the existence of such a decomposition for $\PP_{[\supp\PP]}$ follows
from~\cite[Proposition~3.12]{Jantzen_moment_graphs}. We extend it to all of $\V$,
using the fact that $[\supp\PP]$ is closed and applying property~\ref{proj:3} of projective sheaves.

To prove uniqueness, we can apply $\d$ to both sides of our decomposition. We get
\begin{equation}\label{eq:def}
\d(\PP)=\d\bigl(\B(z_1)\<r_1\>\bigr)+\d\bigl(\B(z_2)\<r_2\>\bigr)+\cdots+\d\bigl(\B(z_s)\<r_s\>\bigr)=v^{r_1}z_1+v^{r_2}z_2+\cdots+v^{r_s}z_s.
\end{equation}
Hence we can reconstruct up to order the pairs $(z_i, r_i)$ using only the defect $\d(\PP)$.
\end{proof}

The above formula shows that the defect is useful for decomposing a projective sheaf into a direct sum of
indecomposable sheaves. We shall use Corollary~\ref{corollary:1} for calculating the required defects.

\begin{example}\label{example:1}\rm Let us explain why we consider only sheaves with finite support.
Consider the moment graph $\G$ with set of vertices $\Z$ with the usual order and edges of the form
$\edgeright i{}{i+1}$ labelled in an arbitrary way. Take a sheaf $\M$ on $\G$ such that $\M^i=S$
for any $i\in\Z$, $\M^E=S/l(E)S$ for any edge $E:\edgeright i{}{i+1}$ and $\rho_{i,E}$ is the natural projection.
This sheaf obviously satisfies properties~\ref{proj:2}--\ref{proj:4} of Definition~\ref{definition:5}.
However, the projection $\rho_{i,\delta i}:\M^i\to\M^{\delta i}$ is a projective cover for any $i\in\Z$.
So the defect (should we define it for sheaves with infinite support) of $\M$ is zero, whereas $\M\ne0$.
\end{example}

\subsection{Functor $\vartheta^s_{on}$}\label{functor_vartheta_on} Let us focus on sheaves on the moment graph $\widehat\G$
defined by an irreducible root system $R$ and a field $\F$ (cf. Section~\ref{associated_moment_graph}).
In this section, we define the functor $\vartheta^s_{on}$ from the category of projective sheaves on $\widehat\G$
to the category of projective sheaves on $\widehat\G^s$. This definition is borrowed
from~\cite[Section~2.9]{Fiebig_sheaves_on_affine_Schubert_varieties} with one small change that
we need to ensure that this functor takes projective sheaves to projective sheaves.

Let $\PP$ be a projective sheaf on $\widehat\G$ and $s\in\widehat\S$.
We define the sheaf $\N=\vartheta^s_{on}\PP$ on $\widehat\G$ as follows.
For a vertex $\bar x=\{x,xs\}\in\widehat\W^s$, we set $\N^{\bar x}=\Gamma(\{x,xs\},\PP)$.

Now consider an edge $\bar E:\edgeright{\bar x}{}{\bar y}$ of $\widehat\G^s$ corresponding to an edge
$E\edgeright x{}y$ in $\widehat\G$.
Denote by $\N^{\bar E}$ the image of $\N^{\bar y}=\Gamma(\{y,ys\},\PP)$
under the projection $\rho_{y,E}\oplus\rho_{ys,Es}:\PP^y\oplus\PP^{ys}\to\PP^E\oplus\PP^{Es}$.
The restriction of this projection on $\N^{\bar y}$ is denoted by $\rho_{\bar y,\bar E}$.
This is a natural definition in view of property~\ref{proj:3} of projective sheaves.
However, we still need to define~$\rho_{\bar x,\bar E}$.

\begin{lemma}\label{lemma:6} Let $\PP$ be a projective sheaf on $\widehat\G$.
\begin{enumerate}
\itemsep=2pt
\item\label{lemma:6:part:1} The restriction $\Gamma(\PP)\to\Gamma(\{x,xs\},\PP)$ is surjective.
\item\label{lemma:6:part:2} 
In the above notation,
$\rho_{x,E}\oplus\rho_{xs,Es}\bigl(\N^{\bar x}\bigr)=\N^{\bar E}$.
\end{enumerate}
\end{lemma}
\begin{proof}\ref{lemma:6:part:1}
Let $s=s_\alpha$ for the corresponding $\alpha\in\widehat\Pi$. Without loss of generality we can assume that $x<xs$.
Then $x(\alpha)>0$ by Proposition~\ref{proposition:0}. Consider the edge $F:\edgeright x{\overline{x(\alpha)}} xs$.
Let $(u,v)\in\Gamma(\{x,xs\},\PP)$.
Let us extend $u$ to a global section $q\in\Gamma(\PP)$ (that is $q_x=u$).
Then $(u-q_x,v-q_{xs})=(0,v-q_{xs})\in\Gamma(\{x,xs\},\PP)$.
Hence $\rho_{xs,F}(v-q_{xs})=\rho_{x,F}(0)=0$.
Therefore by property~\ref{proj:3} of the projective sheaves, we get $v-q_{xs}\in\overline{x(\alpha)}\PP^{xs}$.
So we can write $v-q_{xs}=\overline{x(\alpha)}v'$ for a suitable $v'\in\PP^{xs}$.
Continue $v'$ to a global section $r\in\Gamma(\PP)$
and set
$$
q':=q-\frac{c^{\alpha}-\overline{x(\alpha)}}2\,r.
$$
This is a global section of $\PP$. We have
$$
q'_x=q_x-\frac{\overline{x(\alpha)}-\overline{x(\alpha)}}2\,r_x=q_x=u
$$
and
$$
q'_{xs}=q_{xs}-\frac{\overline{xs(\alpha)}-\overline{x(\alpha)}}2\,r_{xs}=q_{xs}+\overline{x(\alpha)}v'=q_{xs}+v-q_{xs}=v.
$$

\ref{lemma:6:part:2} Follows from~\ref{lemma:6:part:1}.
\end{proof}
\noindent
Now we can take the restriction $\rho_{x,E}\oplus\rho_{xs,Es}$ to $\N^{\bar x}$ for $\rho_{\bar x,\bar E}$.

\begin{lemma}\label{lemma:2.25} Let $\PP$ be a projective sheaf on $\widehat\G$.
Suppose that $\widehat\G_{\supp\PP}$ satisfies the GKM-property.
Then $\vartheta^s_{on}\PP$ is projective.
\end{lemma}
\begin{proof}
Let us take $\alpha\in\widehat\Pi$ such that $s=s_\alpha$.
Set $\N:=\vartheta^s_{on}\PP$.
We are going to check for $\N$ properties~\ref{proj:1}--\ref{proj:4} from the definition of a projective sheaf.

\ref{proj:1} It follows from $\supp\vartheta^s_{on}\PP\subset\overline{\supp\PP}=\pi_s(\supp\PP)$.

\ref{proj:2} We claim that for any $x\in\widehat\W$ such that $x<xs$, the following isomorphism holds:
\begin{equation}\label{eq:1.5}
\N^{\bar x}=\Gamma(\{x,xs\},\PP)\cong\PP^x\oplus\PP^{xs}\<-2\>.
\end{equation}
Let $F:\edgeright x{\pm\overline{x(\alpha)}}xs$ be the corresponding edge. Consider the diagram

\begin{center}
\begin{center}
\begin{picture}(0,0)
\put(6,-2){\vector(0,-1){40}}
\put(0,0){$\PP^x$}
\put(0,-53){$\PP^F$}
\put(11,-22){$\scriptstyle\rho_{x,F}$}
\put(20,-49){\vector(1,0){30}}
\put(54,-53){$0$}
\put(-36,-49){\vector(1,0){36}}
\put(-59,-53){$\PP^{xs}$}
\put(-28,-60){$\scriptstyle\rho_{xs,F}$}
\end{picture}
\end{center}
\end{center}
\vspace{60pt}
where the bottom line is exact. Since $\PP^x$ is projective, there exists a homomorphism (homogeneous of degree $0$)
$f_{x,xs}:\PP^x\to\PP^{xs}$ such that the diagram

\begin{center}
\begin{picture}(0,0)
\put(6,-2){\vector(0,-1){40}}
\put(0,0){$\PP^x$}
\put(0,-53){$\PP^F$}
\put(11,-22){$\scriptstyle\rho_{x,F}$}
\put(20,-49){\vector(1,0){30}}
\put(54,-53){$0$}
\put(-36,-49){\vector(1,0){36}}
\put(-59,-53){$\PP^{xs}$}
\put(-28,-60){$\scriptstyle\rho_{xs,F}$}
\put(-1,-2){\vector(-1,-1){40}}
\put(-39,-17){$\scriptstyle f_{x,xs}$}
\end{picture}
\end{center}
\vspace{60pt}
is commutative. We define the map $\phi:\PP^x\oplus\PP^{xs}\<-2\>\to\Gamma(\{x,xs\},\PP)$ by
$$
\phi(u,v):=\bigl(u,f_{x,xs}(u)+\overline{x(\alpha)}\,v\bigr).
$$
From the above commutative diagram and property~\ref{definition:2:part:2} from Definition~\ref{definition:2}, we get
$$
\rho_{xs,F}\bigl(f_{x,xs}(u)+\overline{x(\alpha)}\,v\bigr)=\rho_{xs,F}\circ f_{x,xs}(u)+\overline{x(\alpha)}\,\rho_{xs,F}(v)=\rho_{x,F}(u),
$$
which proves that the image of $\phi$ is indeed contained in $\Gamma(\{x,xs\},\PP)$.
It is clear that $\phi$ is a homogeneous of degree $0$ homomorphism of $S$-modules.

Conversely, take an arbitrary $(a,b)\in\Gamma(\{x,xs\},\PP)$. We have
$$
\rho_{xs,F}(b-f_{x,xs}(a))=\rho_{xs,F}(b)-\rho_{xs,F}\circ f_{x,xs}(a)=\rho_{xs,F}(b)-\rho_{x,F}(a)=0.
$$
Hence $b-f_{x,xs}(a)\in\ker\rho_{xs,F}=\overline{x(\alpha)}\,\PP^{xs}$ and we can write $b-f_{x,xs}(a)=\overline{x(\alpha)}\,v$
for some $v\in\PP^{xs}$. Thus $\phi(a,v)=(a,b)$.

Finally, suppose that $\phi(u,v)=0$. Then $u=0$ and $\overline{x(\alpha)}\,v=u+\overline{x(\alpha)}\,v=0$.
As $\PP^{xs}$ is free and $\overline{x(\alpha)}\ne0$, we get $v=0$.

\ref{proj:3} Consider an edge $\bar E:\edgeright{\bar x}\gamma{\bar y}$ of $\widehat\G^s$ corresponding to an edge
$E:\edgeright x\gamma y$ in $\widehat\G$. The restriction $\rho_{\bar y,\bar E}$ is surjective by construction.
We need to prove that its kernel is $\gamma\N^{\bar y}$.

We denote by $F$ the edge $\edge y{\pm\overline{y(\alpha)}}ys$.
Let $(u,v)\in\N^{\bar y}=\Gamma(\{y,ys\},\PP)$ be a section such that $\rho_{\bar y,\bar E}(u,v)=0$.
This means that $\rho_{y,E}(u)=0$ и $\rho_{ys,Es}(v)=0$.
By property~\ref{proj:3} of projective sheaves, we get $u=\gamma u'$ and $v=\gamma v'$ for some
$u'\in\PP^y$ and $v'\in\PP^{ys}$.

Since $(u,v)=\gamma(u',v')$, it suffices to prove $(u',v')\in\Gamma(\{y,ys\},\PP)$.
It is enough to consider the case $y,ys\in\supp\PP$, since otherwise we would get $\PP^F=0$ by Lemma~\ref{lemma:2}.
Hence $\PP^E\cong\PP^y/\gamma\PP^y\ne0$. Since $\rho_{x,E}$ is surjective by Lemma~\ref{lemma:2},
we get $\PP^x\ne0$ and $x\in\supp\PP$. Therefore $E$ and $F$ are edges of $\widehat\G_{\supp\PP}$
and their labels $\gamma$ and $\pm\overline{y(\alpha)}$ are not proportional by hypothesis.
Therefore we can cancel out $\gamma$ in $\PP^F$.
We have
$$
\gamma\rho_{y,F}(u')=\rho_{y,F}(u)=\rho_{ys,F}(v)=\gamma\rho_{ys,F}(v').
$$
Cancelling out $\gamma$, we get $\rho_{y,F}(u')=\rho_{ys,F}(v')$, whence $(u',v')\in\Gamma(\{y,ys\},\PP)=\N^{\bar y}$.

\ref{proj:4} The sheaf $\N$ is generated by global sections by Lemma~\ref{lemma:6} and is flabby,
since $\PP$ is flabby and the full preimage of an open subset of $\widehat\W^s$
under the natural projection $\widehat\W\to\widehat\W^s$ is open.
\end{proof}

\subsection{Functor $\vartheta^s_{out}$}\label{functorvarthetasout} Let us remind the definition of this functor given
in~\cite[Section~2.9]{Fiebig_sheaves_on_affine_Schubert_varieties}. Let $s=s_\alpha\in\widehat\S$, where $\alpha\in\widehat\Pi$.
Recall that $\bar x=\{x,xs\}$ and $\overline\Omega$ denotes the image of $\Omega\subset\widehat\W$
under the natural projection $\pi_s:\widehat\W\to\widehat\W^s$.

Let $\PP$ be a projective sheaf on $\widehat\G^s$.
We define the sheaf $\M:=\vartheta^s_{out}\PP$ as follows.
For a vertex $x\in\widehat\W$, we set $\M^x:=\PP^{\bar x}$.
Let $E$ be an edge of $\widehat\G$.
If $E$ connects vertices $x$ and $xs$ for some $x\in\widehat\W$, then we set $\M^E:=\PP^{\bar x}/l(E)\PP^{\bar x}$ and take
for $\rho_{x,E}$ and $\rho_{xs,E}$ the natural projections. Suppose now that $E$ connects vertices $x$ and $y$
such that $x\ne ys$. In this case $\bar x$ and $\bar y$ are connected by the edge $\bar E$.
We set $\M^E:=\PP^{\overline E}$, $\rho_{x,E}:=\rho_{\bar x,\overline E}$ and $\rho_{y,E}:=\rho_{\bar y,\overline E}$.

From this definition, we get the following obvious formula
\begin{equation}\label{eq:2}
\supp\vartheta^s_{out}\PP=\pi_s^{-1}(\supp\PP).
\end{equation}

In this section, we shall use the following notation. Let $\Omega\subset\widehat\W$ be such that $\Omega=\Omega s$ and
$\PP$ be a sheaf on $\widehat\G^s$. For any section $a\in\Gamma(\overline\Omega,\PP)$,
we denote by $a_{out}$ the section of $\Gamma(\Omega,\vartheta^s_{out}\PP)$ defined by $(a_{out})_x=a_{\bar x}$
for any $x\in\Omega$.

\begin{proposition}[\text{\cite[Lemma 2.6]{Fiebig_sheaves_on_affine_Schubert_varieties}}]\label{proposition:4}
Let $\PP$ be a projective sheaf on $\widehat\G^s$ and $\Omega\subset\widehat\W$ be such that $\Omega=\Omega s$.
Suppose that $\widehat\G_{\Omega\cap\pi_s^{-1}(\supp\PP)}$
satisfies the GKM-property. Then any section $u\in\Gamma(\Omega,\vartheta^s_{out}\PP)$
has a unique representation $u=a_{out}+c^\alpha b_{out}$, where $a,b\in\Gamma(\overline\Omega,\PP)$.
\end{proposition}
\begin{proof}
As noted in the proof of Lemma 2.6 from~\cite{Fiebig_sheaves_on_affine_Schubert_varieties},
some GKM-restriction is needed in the case our ground field has positive characteristic. We only
need to specify which exactly (assuming that $\PP$ is projective). In order to make our notation compatible with Fiebig's notation,
we assume $\mathscr G=\PP$ and $\mathscr F=\vartheta_{out}^s\PP$.

The only place where this equality emerges is the prove of the equality $\rho_{x,E}(m'_x)=\rho_{y,E}(m'_y)$.
Here $E$ is the edge connecting points $x,y\in\Omega$.
We may obviously assume that $\mathscr F^E\ne0$. Then by definition, $\mathscr G^{\overline E}\ne0$.
Since $\mathscr G$ is projective, we have by Lemma~\ref{lemma:2} that $\mathscr G^{\bar x}\ne0$ and
$\mathscr G^{\bar y}\ne0$. Hence $\bar x,\bar y\in\supp\mathscr G$. Therefore all elements $x,xs,y,ys$ belong to
$\pi_s^{-1}(\supp\mathscr G)$ and of course to $\Omega$. Thus we can apply the GKM-property of the hypothesis of the
current lemma, say, to edge connecting $x$ and $y$ to the edge connecting $x$ and $xs$.
\end{proof}

\begin{lemma}\label{lemma:7.5}
Let $\PP$ be a projective sheaf on $\widehat{\G}^s$.
Suppose that $\widehat\G_{\pi_s^{-1}(\supp\PP)}$ satisfies the GKM-pro\-perty. Then $\vartheta^s_{out}\PP$
is a projective sheaf on $\widehat{\G}$. If $\PP$ is indecomposable, then so is $\vartheta^s_{out}\PP$.
\end{lemma}
\begin{proof}
Let us take $\alpha\in\widehat\Pi$ such that $s=s_\alpha$.
Set $\M:=\vartheta^s_{out}\PP$.
We are going to check for $\M$ properties~\ref{proj:1}--\ref{proj:4} from the definition of a projective sheaf.

\ref{proj:1} It follows from~(\ref{eq:2}).

\ref{proj:2} Since $\M^x=\PP^{\bar x}$, then any $\M^x$ is a free module.

\ref{proj:3} It is satisfied by construction.

\ref{proj:4} First, we prove that $\M$ is generated by global sections.
Take any $m\in\M^x$. Then $m\in\PP^{\bar x}$. Since $\PP$ is generated by global sections (as a projective sheaf)
there is some global section $u\in\Gamma(\PP)$ such that $u_{\bar x}=m$.
Consider the global section $u_{out}\in\Gamma(\M)$. We have $(u_{out})_x=u_{\bar x}=m$.

It remains to prove that $\M$ is flabby. As we noted after Definition~\ref{definition:4},
it is enough to prove that any section $u\in\Gamma(\widehat\W_{>x},\M)$ can be extended to a section of
$\Gamma(\widehat\W_{\ge x},\M)$.

{\it Case~1: $x<xs$}. Consider the set $\Omega:=\widehat\W_{>x}\setminus\{xs\}$.
We claim $\Omega s=\Omega$. Indeed, take any $y\in\Omega$.
Then by Corollary~\ref{corollary:0}\ref{corollary:0:part:2}, we get $x\le ys$.
Moreover, we have $ys\ne x$ and $ys\ne xs$. Thus we have proved $\Omega\subset\Omega s$.
The inverse inclusion is obtained from this one by multiplying it by $s$ on the right.

By proposition~\ref{proposition:4}, we get $u|_\Omega=a_{out}+c^\alpha b_{out}$ for some
$a,b\in\Gamma(\overline\Omega,\PP)$.
Note that $\overline\Omega=\widehat\W^s_{>\bar x}$. Indeed, take any $y\in\Omega$. Then $y>x$ and $y\ne xs$.
The first inequality implies $\bar y\ge\bar x$. The equality is impossible, as otherwise we would get $y=x$ or $y=xs$.
Thus $\bar y\in \widehat\W^s_{>\bar x}$. Let on the contrary $\bar y>\bar x$.
Then $y\ne xs$ and $\min\{y,ys\}>x$. In any case $y>x$.
Therefore $y\in\Omega$ and $\bar y\in\overline\Omega$.

Since $\PP$ is flabby, we can extend $b$ from $\overline\Omega=\widehat\W^s_{>\bar x}$ to $\widehat\W^s_{\ge\bar x}$
We define the function $u\in\prod_{y\in\widehat\W_{\ge x}}\M^y$ by letting $u$ take the old values on $\widehat\W_{>x}$
and setting $u_x:=u_{xs}+2\overline{x(\alpha)}b_{\bar x}$.
We claim $u\in\Gamma(\widehat\W_{\ge x},\M)$.
Since the edge $F:\edgeright x{}xs$ has label $\pm\overline{x(\alpha)}$, we have $\rho_{x,F}(u_x)=\rho_{xs,F}(u_{xs})$.
Now take an edge $E:\edgeright x{}y$ distinct from $F$. We have then $y,ys\in\Omega$.
Note that $\overline{x(\alpha)}\equiv\overline{y(\alpha)}\pmod{l(E)\M^E}$ by~(\ref{eq:1}).
We get
\begin{multline*}
 \rho_{x,E}(u_x)=\rho_{x,E}(u_{xs}+2\overline{x(\alpha)}b_{\bar x})=\rho_{xs,Es}(u_{xs})+2\overline{x(\alpha)}\rho_{\bar x,\bar E}(b_{\bar x})\\[6pt]
=\rho_{ys,Es}(u_{ys})+2\overline{y(\alpha)}\rho_{\bar y,\bar E}(b_{\bar y})
=\rho_{ys,Es}\Bigl((a_{out})_{ys}+\overline{ys(\alpha)}(b_{out})_{ys}\Bigr)+2\overline{y(\alpha)}\rho_{\bar y,\bar E}(b_{\bar y})\\
\shoveright{=\rho_{ys,Es}\Bigl((a_{out})_{ys}\Bigr)-\overline{y(\alpha)}\rho_{ys,Es}\Bigl((b_{out})_{ys}\Bigr)+2\overline{y(\alpha)}\rho_{\bar y,\bar E}(b_{\bar y})\!\!\!\!\!\!}\\[6pt]
\shoveleft{=\rho_{\bar y,\bar E}(a_{\bar y})-\overline{y(\alpha)}\rho_{\bar y,\bar E}(b_{\bar y})+2\overline{y(\alpha)}\rho_{\bar y,\bar E}(b_{\bar y})
=\rho_{\bar y,\bar E}(a_{\bar y})+\overline{y(\alpha)}\rho_{\bar y,\bar E}(b_{\bar y})}\\[6pt]
=\rho_{y,E}\Bigl((a_{out})_y\Bigr)+\overline{y(\alpha)}\rho_{y,E}\Bigl((b_{out})_y\Bigr)
=\rho_{y,E}\Bigl((a_{out})_y+\overline{y(\alpha)}(b_{out})_y\Bigr)=\rho_{y,E}(u_y).
\end{multline*}

For the remaining two cases, we have to prove that $\Omega$ is open.
Indeed, suppose that $y\in\Omega$ and $z>y$.
Then $z>y>x$. We can not have $z=xs$, since otherwise we would get $xs>y>x$, which is impossible
in view of $\ell(xs)=\ell(x)+1$.

{\it Case~2: $x>xs$}. Set $\Omega:=\widehat\W_{>xs}\setminus\{x\}$.
Note that $\widehat\W_{>x}\subset\Omega$ and $\Omega=\Omega s$ as is shown in case~1.
First, we extend  $u$ to a section of $\Gamma(\Omega,\M)$.
Consider the open subset $\Omega':=\widehat\W_{>x}\cup(\Omega\setminus[\supp\M])$.
By Lemma~\ref{lemma:1.5}, we can extend $u$ to $\Omega'$ by setting $u_y:=0$ for all $y\in\Omega\setminus[\supp\M]$.
Obviously $\widehat\W_{>x}\subset\Omega'\subset\Omega$ and $\Omega\setminus\Omega'\subset[\supp\M]$.
Since $\supp\M$ is finite, $\Omega\setminus\Omega'$ is also finite.
Choose a maximal (with respect to inclusion) open subset $\mathcal U\subset\widehat\W$
such that $\Omega'\subset\mathcal U\subset\Omega$ and $m$ is extendable to $\mathcal U$.
It is enough to consider the case $\mathcal U\ne\Omega$.

Choose a maximal element $y\in\Omega\setminus\mathcal U$. Then $\widehat\W_{>y}\subset\mathcal U$.
We claim that $y<ys$. Indeed, suppose on the contrary that $y>ys$. Then by the lifting property,
we derive from $xs<y$ that $x\le y$. By the definition of $\Omega$, we get $x<y$ and $y\in\Omega'\subset\mathcal U$,
which is a contradiction.

Now that we know $y<ys$, we can extend $u|_{\widehat\W_{>y}}$ to $\widehat\W_{\ge y}$ by case~1.
Hence we have extended $u$ to the open subset $\mathcal U\dotcup\{y\}$, which contradicts the choice of $\mathcal U$.
Therefore, we have proved $\mathcal U=\Omega$.

We have obtained an extension of $u$ to $\Omega$, which we denote also by $u$.
By Proposition~\ref{proposition:4}, we have the representation
$u|_\Omega=a_{out}+c^\alpha b_{out}$ for some $a,b\in\Gamma(\overline\Omega,\PP)$.
By case~1, we have $\overline\Omega=\widehat W^s_{>\bar x}$. As $\PP$ is flabby, we can extend $a$ and $b$ to
$\widehat W^s_{\ge\bar x}$ and set $u_x:=a_{\bar x}+\overline{x(\alpha)}b_{\bar x}$.
We claim that $u$ so defined is a section of $\Gamma(\widehat\W_{\ge x},\M)$.
Indeed, take an edge $E:\edgeright x{}y$. We have
\begin{multline*}
 \rho_{x,E}(u_x)
=\rho_{\bar x,\bar E}(a_{\bar x})+\overline{x(\alpha)}\rho_{\bar x,\bar E}(b_{\bar x})
=\rho_{\bar y,\bar E}(a_{\bar y})+\overline{y(\alpha)}\rho_{\bar y,\bar E}(b_{\bar y})\\[6pt]
=\rho_{y,E}\Bigl((a_{out})_y\Bigr)+\overline{y(\alpha)}\rho_{y,E}\Bigl((b_{out})_y\Bigr)
=\rho_{y,E}\Bigl((a_{out})_y+\overline{y(\alpha)}(b_{out})_y\Bigr)=\rho_{y,E}(u_y).
\end{multline*}

Now suppose that $\PP$ is indecomposable. By Proposition~\ref{proposition:3},
we get $\PP\cong\B^s(\bar w)\<r\>$ for some $\bar w\in\widehat\W^s$.
For definiteness, we assume that $w\in\widehat\W$ is chosen so that $ws<w$.
We claim that $\M=\vartheta^s_{out}\PP\cong\B(w)\<r\>$.

Suppose that $\M^x\ne0$. Then $\PP^{\bar x}\ne0$ and $\bar x\le\bar w$.
If $\bar x=\bar w$, then either $x=w$ or $x=ws<w$. If $\bar x<\bar w$, then $\min\{x,xs\}<ws$.
Thus either $x<ws<w$ or $xs<ws<w$ and $xs<x$. Applying the lifting property in the second case to the inequality $xs<w$,
we get $x\le w$. Thus we have proved $\supp\M\subset\widehat\G_{\le w}$.
Moreover, we have $\M^w=\PP^{\bar w}\cong\B(\bar w)\<r\>^{\bar w}\cong S\<r\>$.

It remains to prove that $\rho_{x,\delta x}:\M^x\to\M^{\delta x}$ is a projective cover for any $x\ne w$.
By our construction, we get
$$
\ker\rho_{x,\delta x}=
\left\{
\arraycolsep=1pt
\begin{array}{ll}
\ker\rho_{\bar x,\delta\bar x}&\text{ if }xs<x;\\
\ker\rho_{\bar x,\delta\bar x}\cap\overline{x(\alpha)}\PP^{\bar x}&\text{ if }x<xs.
\end{array}
\right.
$$
In the second case, we get $\ker\rho_{x,\delta x}\subset\overline{x(\alpha)}\PP^{\bar x}\subset\m\M^x$,
whence $\rho_{x,\delta x}$ is a projective cover. In the first case, we have $\bar x\ne\bar w$.
Hence $\ker\rho_{x,\delta x}=\ker\rho_{\bar x,\delta\bar x}\subset\m\PP^{\bar x}=\m\M^x$
and $\rho_{x,\delta x}$ is again a projective cover.
\end{proof}


\subsection{Functor $\vartheta^s$} We define $\vartheta^s:=\vartheta^s_{out}\circ\vartheta^s_{on}$.
This functor is applicable to projective sheaves $\PP$ such that $\vartheta^s_{on}\PP$ is also projective.
The combination of Lemmas~\ref{lemma:2.25} and~\ref{lemma:7.5} yields the following result.

\begin{theorem}\label{theorem:8}
Let $\PP$ be a projective sheaf such that $\widehat\G_{\supp\PP\cup(\supp\PP)s}$ satisfies the GKM-property.
Then the sheaf $\vartheta^s\PP$ is well defined and projective with support contained in $\supp\PP\cup(\supp\PP)s$.
\end{theorem}
\begin{proof} By Lemma~\ref{lemma:2.25}, the sheaf $\vartheta^s_{on}\PP$ is projective.
As is noted in part~\ref{proj:1} of the proof of this lemma,
$\supp\vartheta^s_{on}\PP\subset\overline{\supp\PP}$.
Since $\pi_s^{-1}\Bigl(\overline{\supp\PP}\Bigr)=\supp\PP\cup(\supp\PP)s$, Lemma~\ref{lemma:7.5} implies that
$\vartheta^s\PP=\vartheta^s_{out}(\vartheta^s_{on}\PP)$ is also projective.
\end{proof}

\subsection{Characters of sheaves} We define the {\it graded character} of free modules in $\mathcal C$ by
$$
\grk S\<l_1\>\oplus\cdots\oplus S\<l_k\>=v^{l_1}+\cdots+v^{l_k}.
$$
Obviously, $\grk0=0$ and $\grk M\<r\>=v^r\grk M$ for any free $M$ in $\mathcal C$.

Let us recall the definition of the {\it Hecke algebra associated to the Coxeter system} $(\widehat\W,\widehat\S)$
(cf.~\cite{Kazhdan_Lusztig} and~\cite{Soergel}). Consider the free $\Z[v,v^{-1}]$-module
$$
\mathcal H:=\bigoplus_{x\in\widehat\W}\Z[v,v^{-1}]T_x.
$$
The structure of an associative $\Z[v,v^{-1}]$-algebra on $\mathcal H$ is given by the multiplication
$$
\arraycolsep=2pt
\begin{array}{rcl}
T_xT_y&=&T_{xy}\;\text{ if }\ell(xy)=\ell(x)+\ell(y);\\[5pt]
 T_s^2&=&v^{-2}T_e+(v^{-2}-1)T_s\;\text{ for any }s\in\widehat\S.
\end{array}
$$
The unit of this algebra is $T_e$.

It is more convenient to work with elements $H_x:=v^{\ell(x)}T_x$.
There is exactly one ring homomorphism $\overline{\,\cdot}:\mathcal H\to\mathcal H$ given by
$\bar v=v^{-1}$ and $\overline{H_x}=(H_{x^{-1}})^{-1}$.
Elements $H\in\mathcal H$ such that $\overline H=H$ are called {\it self-dual}.

\begin{proposition}[\text{\cite{Kazhdan_Lusztig}, \cite{Soergel}}]
For any $w\in\widehat\W$, there exists a unique element $\underline H_w=\sum_{x\in\widehat\W}h_{x,w} H_x\in\mathcal H$
with the following properties:
\begin{enumerate}
\itemsep=2pt
\item $\underline H_w$ is self-dual;
\item $h_{x,w}=0$ if $x\not\le w$ and $h_{w,w}=1$;
\item $h_{x,w}\in v\Z[v]$ for $x<w$.
\end{enumerate}
\end{proposition}
\noindent
In particular, the element $\underline H_e=H_e$ is the unit of $\mathcal H$ and $\underline H_s=H_s+v$ for any $s\in\widehat\S$.

For any sheaf $\M$ on $\widehat\G$ with finite support whose all stalks $\M^x$ are free modules in $\mathcal C$ and $w\in\widehat\W$,
we define
$$
h(\M):=\sum_{x\in\widehat\W}v^{-l(x)}\cdot\grk\!\M^x\cdot H_x
$$

P. Fiebig formulated the following conjecture on the characters of Braden--MacPherson sheaves,
which as he showed in~\cite{Fiebig_sheaves_on_affine_Schubert_varieties} and~\cite{Fiebig_Lusztig’s_conjecture}
is equivalent to Lusztig's conjecture.

\begin{conjecture}[\text{\cite[Conjecture 2.10]{Fiebig_An_upper_bound}}] If $w\in\widehat\W$ is such that
$\widehat\G_{\le w}$ satisfies the GKM-property, then $v^{\ell(w)}h(\B(w))=\underline H_w$.
\end{conjecture}
\noindent
Theorem 4.7 from~\cite{Fiebig_An_upper_bound} shows that this conjecture is true if  the base field
$\F$ has characteristic~0.
From the proof of Lemma~\ref{lemma:2.25}, we easily derive the following fact
(cf.~\cite[Lemma 4.5]{Fiebig_sheaves_on_affine_Schubert_varieties})

\begin{lemma}\label{htransl}
For any $\PP$ projective sheaf on $\widehat\G$ and $s\in\widehat\S$, we have
$$
h(\vartheta^s\PP)=v^{-1}h(\PP)\underline H_s.
$$
\end{lemma}
\begin{proof} Recall the following formula from~\cite{Soergel}: for any $x\in\widehat\W$ and $s\in\widehat\S$, we have
\begin{equation}\label{eq:3}
H_x\underline H_s=
\left\{
\arraycolsep=2pt
\begin{array}{ll}
H_{xs}+vH_x&\text{ if }xs>x;\\[3pt]
H_{xs}+v^{-1}H_x&\text{  if }xs < x.
\end{array}
\right.
\end{equation}
We set ${\widehat\W\,}':=\{x\in\widehat\W\suchthat x<xs\}$. Clearly, $\widehat\W={\widehat\W\,}'\dotcup{\widehat\W\,}'s$.
By~(\ref{eq:1.5}), we get
\begin{multline*}
h(\vartheta^s\PP)=\sum_{x\in\widehat\W}v^{-\ell(x)}\grk(\vartheta^s_{on}\PP)^{\bar x}H_x
=\sum_{x\in\widehat\W}v^{-\ell(x)}\grk\Gamma(\{x,xs\},\PP)H_x\\
=\sum_{x\in{\widehat\W\,}'}v^{-\ell(x)}\grk\Gamma(\{x,xs\},\PP)H_x+\sum_{x\in{\widehat\W\,}'}v^{-\ell(xs)}\grk\Gamma(\{xs,x\},\PP)H_{xs}\\
=\sum_{x\in{\widehat\W\,}'}v^{-\ell(x)}(\grk\PP^x+v^{-2}\grk\PP^{xs})(H_x+v^{-1}H_{xs}).
\end{multline*}
On the other hand, we by~(\ref{eq:3}), we get
\begin{multline*}
h(\PP)\underline H_s=\sum_{x\in\widehat\W}v^{-\ell(x)}\grk\PP^xH_x\underline H_s\\
=\sum_{x\in{\widehat\W\,}'}v^{-\ell(x)}\grk\PP^x(H_{xs}+vH_x)+\sum_{x\in{\widehat\W\,}'}v^{-\ell(xs)}\grk\PP^{xs}(H_x+v^{-1}H_{xs})\\
=v\sum_{x\in{\widehat\W\,}'}v^{-\ell(x)}(H_x+v^{-1}H_{xs})(\grk\PP^x+v^{-2}\grk\PP^{xs}).
\end{multline*}
The comparison of the above formulas gives the required result.
\end{proof}

\section{Bases for Bott-–Samelson sheaves}\label{Bases_for_Bott-–Samelson_sheaves}

In this section, we are going to look closer at Fiebig's realization (\cite[Section~6]{Fiebig_An_upper_bound})
of Bott-–Samelson modules.

\subsection{Sequences} We reserve the symbol $\sp$ to denote the blank space in sequences.
Let $\s=(s_1,\ldots,s_l)$ be a sequence in $\widehat\S$.
Any sequence obtained from~$\s$ by replacing some of its entries with $\sp$ is called a {\it subsequence} of $\s$.
The set of all subsequences $\sigma$ of $\s$ is denoted by $I(\s)$.
If we multiply the entries of $\sigma$ respecting their order and ignoring symbols $\sp$,
then we get the element of $\widehat\W$ denoted by $\ev(\sigma)$.
For any $\I\subset\widehat\W$, we set
$$
I(\s)_\I:=\{\sigma\in I(\s)\suchthat\ev(\sigma)\in\I\}.
$$
Of particular interest is the set $I(\s)_x:=I(\s)_{\{x\}}=\{\sigma\in I(\s)\suchthat\ev(\sigma)=x\}$.


We multiply sequences by concatenation and truncate by ${}'$: if $\sigma=(\sigma_1,\ldots,\sigma_l)$ and
$\tau=(\tau_1,\ldots,\tau_m)$, then $\sigma\tau=(\sigma_1,\ldots,\sigma_l,\tau_1,\ldots,\tau_m)$
and $\sigma'=(\sigma_1,\ldots,\sigma_{l-1})$. In the notation of this multiplication, we prefer to write $\sp$ instead
of $(\sp)$ and $s$ instead of $(s)$. We also apply these operations to subsets elementwise.

{\bf Example.} Let $\s=(s_1,s_2,s_1)$, where $s_1$ and $s_2$ are distinct simple reflections. Then
$$
I(\s)=\{(\sp,\sp,\sp),(s_1,\sp,\sp),(\sp,s_2,\sp),(s_1,s_2,\sp),(\sp,\sp,s_1),(s_1,\sp,s_1),(\sp,s_2,s_1),(s_1,s_2,s_1)\}
$$
and $I(\s)_e=\{(\sp,\sp,\sp),(s_1,\sp,s_1)\}$.
Below are some examples how we extend and truncate elements and sets:
\begin{multline*}
(\sp,\sp,\sp)'=(\sp,\sp),\quad (\sp,s_2,s_1)'=(\sp,s_2),\quad (s_1,\sp)\sp(\s_1,\s_2)=(s_1,\sp,\sp,s_1,\s_2),\\[6pt]
\shoveleft{(\sp,s_1)s_2(s_2,\sp,s_1)=(\sp,s_1,s_2,s_2,\sp,s_1),\quad
\{(s_1,s_2,\sp),(\sp,\sp,s_1)\}'=\{(s_1,s_2),(\sp,\sp)\},}\\[6pt]
\shoveleft{(s_1,\sp,\sp)s_1=(s_1,\sp,\sp,s_1),\quad\{(s_1,\sp,s_1),(s_1,s_2,\sp)\}\sp=\{(s_1,\sp,s_1,\sp),(s_1,s_2,\sp,\sp)\}.}\\[-12pt]
\end{multline*}

Note that $\s\in I(\s)$ and $|I(\s)|=2^{|\s|}$. We denote the empty sequence by $\emptyset$ (just like the empty set).
For example, $I(\emptyset)=\{\emptyset\}$ and
\begin{equation}\label{eq:0.5}
I(\emptyset)_x=\left\{
\begin{array}{cl}
\{\emptyset\}&\text{ if }x=e;\\[3pt]
\emptyset &\text{ otherwise}.
\end{array}
\right.
\end{equation}
Another useful formula, which we shall often use, is
\begin{equation}\label{eq:20}
I(\s)_z=I(\s')_z\sp\,\dotcup\,I(\s')_{zs}s,
\end{equation}
where $\s\ne\emptyset$ and $s$ is the last (rightmost) element of $\s$.
We also have
\begin{equation}\label{eq:1.25}
I(\s)=I(\s')\sp\,\dotcup\,I(\s')s.
\end{equation}

\subsection{Bott-–Samelson sheaves} Let $\s=(s_1,\ldots,s_l)$ be a sequence in $\widehat\S$.
We denote $J(\s)=\{\ev(\sigma)\suchthat\sigma\in I(\s)\}$. In particular, $J(\emptyset)=\{e\}$.
The deletion and subword properties (Propositions~\ref{proposition:deletion_property}
and~\ref{proposition:subword_property}) imply the following property of $J(\s)$: if $x\in J(\s)$ and $y<x$,
then $y\in J(\s)$.

\begin{definition}\label{definition:BS_sheaf}
Let $\s=(s_1,\ldots,s_l)$ be a sequence in $\widehat\S$ such that $\widehat\G_{J(\s)}$ satisfies the GKM-property.
The Bott-–Samelson sheaf corresponding to this sequence is $\B(\s):=\vartheta^{s_l}\circ\cdots\circ\vartheta^{s_1}\B(e)$.
\end{definition}
In this definition, $\B(e)$ is the Braden-MacPherson sheaf with defect $e$. It is a very simple sheaf defined by
$\B(e)^e=S$, $\B(e)^x=0$ for $x\ne e$, $\B(e)^E=0$ and $\rho_{x,E}=0$ for all edges $E$.
Thus $\B(\emptyset)=\B(e)$.


\begin{lemma}\label{lemma:9}
Let $\s=(s_1,\ldots,s_l)$ be a sequence in $\widehat\S$ such that $\widehat\G_{J(\s)}$ satisfies the GKM-pro\-perty.
Then $\B(\s)$ is well-defined and projective
with support contained in $J(\s)$ and  $v^lh(\B(\s))=(H_{s_1}+v)\cdots(H_{s_l}+v)$.
\end{lemma}
\begin{proof}
We prove by induction on $i=0,\ldots,l$ that $\vartheta^{s_i}\circ\cdots\circ\vartheta^{s_1}\B(e)$
is a well-defined projective sheaf with support contained in $J((s_1,\ldots,s_i))$.
The induction starts trivially with the case $i=0$, so suppose that $i<l$ and the claim for $i$ is true.
Then our claim is also true for $i+1$ by Theorem~\ref{theorem:8},
since $J((s_1,\ldots,s_i))\cup J((s_1,\ldots,s_i))s_{i+1}\subset J((s_1,\ldots,s_{i+1}))$.

The second statement follows from Lemma~\ref{htransl}.
\end{proof}

\subsection{Module $\bigoplus_{\sigma\in I(\s)}Q$} As we noted in Section~\ref{associated_moment_graph}, $\bar\alpha\ne 0$ for any $\alpha\in\widehat R$.
Therefore we can consider the localization of $S$ with respect to all these elements:
$$
Q:=S[\bar\alpha^{-1}\suchthat\alpha\in\widehat R].
$$
We consider the direct sum $\bigoplus_{\sigma\in I(s)}Q$ as an $S$-module with the componentwise action of $S$
as well as a $\mathcal Z$-module with the following action
\begin{equation}\label{eq:1.125}
(zf)_\sigma:=z_{\ev(\sigma)}f_\sigma
\end{equation}
for any $z\in\mathcal Z$ and $f\in\bigoplus_{\sigma\in I(s)}Q$.
For example,
\begin{equation}\label{eq:21}
(c^\lm f)_\sigma=\overline{\ev(\sigma)(\lambda)}\,f_\sigma.
\end{equation}
or in a different form $(c^\lm f)_\sigma=\overline{x(\lm)}f_\sigma$ for $\sigma\in I(\s)_x$.

For any nonempty sequence $\s=(s_1,\ldots,s_l)$ in $\widehat\S$, we have the {\it diagonal embedding}\linebreak
$\Delta:\bigoplus_{\sigma\in I(\s')}Q\to\bigoplus_{\sigma\in I(\s)}Q$ defined by
$$
\Delta(f)_{\sigma\sp}=f_\sigma\quad\text{ and }\quad\Delta(f)_{\sigma s_l}=f_\sigma\quad\text{ for any}\quad\sigma\in I(\s')
$$
and the {\it antidiagonal embedding} $\Delta^-:\bigoplus_{\sigma\in I(\s')}Q\to \bigoplus_{\sigma\in I(\s)}Q$ defined by
$$
\Delta^-(f)_{\sigma\sp}=f_\sigma\quad\text{ and }\quad\Delta^-(f)_{\sigma s_l}=-f_\sigma\quad\text{ for any}\quad\sigma\in I(\s').
$$

Let $\iota:\bigoplus_{\sigma\in I(\s)}Q\to\bigoplus_{\sigma\in I(\s)}Q$ be the map
given by $\iota(f)_{\sigma\sp}=f_{\sigma s_l}$, $\iota(f)_{\sigma s_l}=f_{\sigma\sp}$.
As $2$ is invertible in $\F$, we have the decomposition
\begin{equation}\label{eq:19}
\bigoplus_{\sigma\in I(\s)}Q=\(\bigoplus_{\sigma\in I(\s)}Q\)^{\!\iota}\oplus\(\bigoplus_{\sigma\in I(\s)}Q\)^{\!-\iota},
\end{equation}
where $(\cdot)^\iota$ denotes $\iota$-invariant elements and $(\cdot)^{-\iota}$ denotes the $\iota$-antiinvariant elements.
The first summand is the image of $\Delta$ and the second summand is the image of $\Delta^-$.

We can translate any $S$-module endomorphism $\phi$ of $\bigoplus_{\sigma\in I(\s')}Q$ to the $S$-module endomorphism
$\Delta^\phi$ of $\bigoplus_{\sigma\in I(\s)}Q$ by
\begin{equation}\label{eq:18}
\Delta^\phi\bigl(\Delta(f)\bigr)=\Delta\bigl(\phi(f)\bigr),\quad \Delta^\phi\bigl(\Delta^-(f)\bigr)=\Delta^-\bigl(\phi(f)\bigr)
\end{equation}
for any $f\in\bigoplus_{\sigma\in I(\s')}Q$.

We shall actually need the submodule $\bigoplus_{\sigma\in I(\s)}S\subset\bigoplus_{\sigma\in I(\s)}Q$.
It is an $S$-submodule as well as a $\mathcal Z$-submodule. The embeddings $\Delta$ and $\Delta^-$ restrict
well to these submodules.

We also introduce the structure of a $\mathcal Z(\I)$-module on $\bigoplus_{\sigma\in I(s)}Q$ for any
$\I\subset\widehat\W$ containing $J(\s)$
by the same formula~(\ref{eq:1.125}). If we apply the natural homomorphism $\mathcal Z(\mathcal I')\to\mathcal Z(\I)$
for any $\mathcal I'\supset\mathcal I$, then we get the usual action of $\mathcal Z(\mathcal I')$ on this space.

Directly from the definitions, we get the following simple result.
\begin{proposition}\label{proposition:5}
Let $\s=(s_1,\ldots,s_l)$ be a sequence in $\widehat\S$ and $\I$ be a subset of $\widehat\W$ containing $J(\s)$ such that $\I s_l=\I$.
We have $\Delta(zf)=z\Delta(f)$ for any $f\in\bigoplus_{\sigma\in I(\s')}Q$ and $z\in\mathcal Z(\I)$ such that
$z_x=z_{xs_l}$ for any $x\in\I$.
\end{proposition}

\subsection{Module $\X(\s)$} The following definition is due to P. Fiebig.

\begin{definition}[\text{\cite[Definition 6.1.]{Fiebig_An_upper_bound}}]\label{definition:7}
We define for all sequences $\s=(s_1,\ldots,s_l)$ in $\widehat\S$ the $S$-submodule $\X(\s)\subset\bigoplus_{\sigma\in I(\s)}S$
by the following inductive rule:
\begin{enumerate}
\itemsep=2pt
\item\label{definition:7:part:1} $\X(\emptyset):=\bigoplus_{\sigma\in I(\emptyset)}S\cong S$;
\item\label{definition:7:part:2} if $\s=(s_1,\ldots,s_l)$ is not empty, then
$$
\X(\s):=\Delta(\X(\s'))+c^{\alpha_l}\Delta(\X(\s')),
$$
where $\alpha_l\in\widehat\Pi$ is such that $s_l=s_{\alpha_l}$.
\end{enumerate}

\begin{proposition}[\text{\cite[Proposition 6.14(1)]{Fiebig_An_upper_bound}}]
$\X(\s)$ is stable under the action of $\mathcal Z$.
\end{proposition}
\end{definition}

The second part of Proposition 6.14 from~\cite{Fiebig_An_upper_bound} says that $\X(\s)$ is isomorphic as a $\mathcal Z$-module
to the Bott-Samelson module corresponding to $\s$. We want to prove a similar result stating that
under some GKM-restriction there is an isomorphism from $\Gamma(\B(\s))$ to $\X(\s)$ that agrees well with restrictions
(Theorem~\ref{theorem:3}).

\subsection{Comparison theorem} Following Fiebig's paper~\cite{Fiebig_An_upper_bound}, we define for any
$\I\subset\widehat\W$ the $S$-modules $M^\I$ and $M_\I$ for any $S$-submodule
$M\subset\bigoplus_{\sigma\in I(\s)}Q$ as follows:
$$
M^\I:=\Bigl\{f|_{I(\s)_\I}\,\Big|\,f\in M\Bigr\},\quad M_\I:=\Bigl\{f|_{I(\s)_\I}\,\Big|\,f\in M\text{ and }f|_{I(\s)\setminus I(\s)_\I}=0\Bigr\}.
$$
By definition $M_\I\subset M^\I\subset\bigoplus_{\sigma\in I(\s)_\I}Q$.
Note that $M_\I=M^\I=0$ if $I(\s)_\I=\emptyset$.
We also abbreviate $M^x:=M^{\{x\}}$ and $M_x:=M_{\{x\}}$.

For an element $x\in\widehat\W$, a sheaf $\M$ on $\widehat\G$ and
an $S$-submodule $M\subset\bigoplus_{\sigma\in I(\s)}Q$, we define
the natural restrictions $r^\M_x:\Gamma(\M)\to\M^x$ and $r^M_x:M\to M^x$ by
$$
r^\M_x(m):=m_x,\quad r^M_x(f):=f|_{I(\s)_x}.
$$

\begin{theorem}\label{theorem:3}
Let $\s$ be a sequence in $\widehat\S$ such that $\widehat\G_{J(\s)}$ satisfies the GKM-property.
Then there exist a $\mathcal Z(J(\s))$-module isomorphism $\phi(\s):\Gamma(\B(\s))\to\X(\s)$ and
$S$-module isomorphisms $\phi_x(\s):\B(\s)^x\to\X(\s)^x$ for each $x\in\widehat\W$ such that
the following diagrams are commutative:
$$
\begin{CD}
\Gamma(\B(\s)) @>\phi(\s)>> \X(\s)\\
@Vr^{\B(\s)}_xVV                   @VVr^{\X(\s)}_xV   \\
\B(\s)^x @>\phi_x(\s)>>            \X(\s)^x
\end{CD}
$$
\end{theorem}{
\begin{proof}
We apply induction on the length of $\s$. First, consider the case $\s=\emptyset$.
The isomorphism $\phi(\emptyset)$ is defined by
$(\phi(\s)(m))_{\emptyset}=m_e$. One checks easily that this is indeed a $\mathcal Z$-module isomorphism.
We also define $\phi_e(\emptyset)$ by $(\phi_e(\emptyset)(a))_\emptyset=a$ for any $a\in\B(\emptyset)^e=S$
and set $\phi_x(\emptyset)=0$ for any $x\ne e$. Obviously, the above diagram is commutative
for the isomorphisms so defined.

Now suppose that $\s=(s_1,\ldots,s_l)$ is a nonempty sequence in $\widehat\S$ and that the
required isomorphisms are constructed for $\s'=(s_1,\ldots,s_{l-1})$. Let us choose $\alpha_l\in\widehat\Pi$
so that $s=s_{\alpha_l}$. We set $\mathfrak c^{\alpha_l}:=c^{\alpha_l}|_{J(\s)}\in\mathcal Z(J(\s))$.

We can identify $\Gamma(\vartheta^s_{on}\B(\s'))$ and $\Gamma(\B(\s'))$ in the obvious way.
So in the notation of Section~\ref{functorvarthetasout}, we have for any section $a\in\Gamma(\B(\s'))$
the section $a_{out}\in\Gamma(\B(\s))$ defined by $(a_{out})_x=(a_x,a_{x{s_l}})\in\Gamma(\{x,xs_l\},\B(\s'))$.
Now Proposition~\ref{proposition:4} for $\Omega=\widehat\W$ states that
$$
\Gamma(\B(\s))=\Gamma(\B(\s'))_{out}\oplus\mathfrak c^{\alpha_l}\Gamma(\B(\s'))_{out}.
$$
We define $\phi(\s)$ separately on each summand as follows:
\begin{equation}\label{eq:4}
\phi(\s)(a_{out}):=\Delta(\phi(\s')(a)),\quad \phi(\s)(\mathfrak c^{\alpha_l}a_{out}):=\mathfrak c^{\alpha_l}\Delta(\phi(\s')(a))
\end{equation}
for any $a\in\Gamma(\B(\s'))$. Since every module $\B(\s')^x$ is free, any $a\in\Gamma(\B(\s'))$ is uniquely determined
by $a_{out}$ as well as by $\mathfrak c^{\alpha_l}a_{out}$. So $\phi(\s)$ is well-defined as a map.

We claim that $\phi(\s)$ is a $\mathcal Z(J(\s))$-isomorphism. Indeed, take any $z\in\mathcal Z(J(\s))$.
Consider $z'\in\mathcal Z(J(\s))$ defined by $(z')_x:=z_{xs_l}$ for any $x\in J(\s)$.
Then we have
$$
(z-z')_x=z_x-z_{xs_l}\equiv0\!\!\!\pmod{\overline{x(\alpha_l)}}
$$
for any $x\in J(\s)$.
Therefore, we can consider the quotient $(z-z')/(2\,\mathfrak c^{\alpha_l})\in\prod_{x\in J(\s)}S$.
Since $\widehat\G_{J(\s)}$ satisfies the GKM-property, we obtain $(z-z')/(2\,\mathfrak c^{\alpha_l})\in\mathcal Z(J(\s))$.
Evaluation at an arbitrary $x\in\widehat\W$, proves the following formulas:
\begin{equation}\label{eq:5}
za_{out}=\(\frac{z+z'}2\,a\)_{\!\!out}+\mathfrak c^{\alpha_l}\!\(\frac{z-z'}{\;2\,\mathfrak c^{\alpha_l}}\,a\)_{\!\!out}\\[6pt]
\end{equation}
\begin{equation}\label{eq:6}
z\mathfrak c^{\alpha_l}a_{out}=\(\mathfrak c^{\alpha_l}\,\frac{z-z'}2\,a\)_{\!\!out}+\mathfrak c^{\alpha_l}\!\(\frac{z+z'}2\,a\)_{\!\!out}
\end{equation}
for any $a\in\Gamma(\B(\s'))$.

We actually need to prove these formulas for $x\in J(\s)$ as otherwise both sides of each formula equals $0$.
The right-hand side of~(\ref{eq:5}) evaluated at $x$ equals the following element of $\Gamma(\{x,xs_l\},\B(\s'))$:
\begin{multline*}
\(\frac{z_x+z'_x}2\,a_x,\frac{z_{xs_l}+z'_{xs_l}}2\,a_{xs_l}\)+\overline{x(\alpha_l)}\(\frac{z_x-z'_x}{\;2\,\overline{x(\alpha_l)}\;}\,a_x,\frac{z_{xs_l}-z'_{xs_l}}{\;2\,\overline{xs_l(\alpha_l)}\;}\,a_{xs_l}\)\\[2pt]
=\(\frac{z_x+z_{xs_l}}2\,a_x,\frac{z_{xs_l}+z_x}2\,a_{xs_l}\)+\overline{x(\alpha_l)}\(\frac{z_x-z_{xs_l}}{\;2\,\overline{x(\alpha_l)}\;}\,a_x,\frac{z_{xs_l}-z_x}{\;-2\,\overline{x(\alpha_l)}\;}\,a_{xs_l}\)\\[8pt]
=(z_xa_x,z_xa_{xs_l})=(za_{out})_x.
\end{multline*}
The right-hand side of~(\ref{eq:6}) evaluated at $x$ equals the following element of $\Gamma(\{x,xs_l\},\B(\s'))$:
\begin{multline*}
\(\overline{x(\alpha_l)}\,\frac{z_x-z'_x}2\,a_x,\overline{xs_l(\alpha_l)}\,\frac{z_{xs_l}-z'_{xs_l}}2\,a_{xs_l}\)+\overline{x(\alpha_l)}\!\(\frac{z_x+z'_x}2\,a_x,\frac{z_{xs_l}+z'_{xs_l}}2\,a_{xs_l}\)=\\[2pt]
=\(\overline{x(\alpha_l)}\,\frac{z_x-z_{xs_l}}2\,a_x,-\overline{x(\alpha_l)}\,\frac{z_{xs_l}-z_x}2\,a_{xs_l}\)+\(\overline{x(\alpha_l)}\,\frac{z_x+z_{xs_l}}2\,a_x,\overline{x(\alpha_l)}\,\frac{z_{xs_l}+z_x}2\,a_{xs_l}\)\\[6pt]
=\(\overline{x(\alpha_l)}z_xa_x,\overline{x(\alpha_l)}z_xa_{xs_l}\)=(z\mathfrak c^{\alpha_l}a_{out})_x.
\end{multline*}

Let us take arbitrary $a\in\Gamma(\B(\s'))$ and $z\in\mathcal Z(J(\s))$.
By~(\ref{eq:4}),~(\ref{eq:5}), the inductive hypothesis and Proposition~\ref{proposition:5}, we get
\begin{multline*}
\phi(\s)(za_{out})=\Delta\!\(\phi(\s')\(\frac{z+z'}2\,a\)\)+\mathfrak c^{\alpha_l}\Delta\!\(\phi(\s')\(\frac{z-z'}{\;2\,\mathfrak c^{\alpha_l}}\,a\)\)\\[8pt]
\hspace{-130pt}=\Delta\!\(\frac{z+z'}2\,\phi(\s')(a)\)+\mathfrak c^{\alpha_l}\Delta\!\(\frac{z-z'}{\;2\,\mathfrak c^{\alpha_l}}\,\phi(\s')(a)\)\\[8pt]
=\frac{z+z'}2\,\Delta\!\(\phi(\s')(a)\)+\frac{z-z'}2\Delta\!\(\phi(\s')(a)\)
=z\Delta\!\(\phi(\s')(a)\)=z\phi(\s)(a_{out}).
\end{multline*}
By~(\ref{eq:4}),~(\ref{eq:6}), the inductive hypothesis and Proposition~\ref{proposition:5}, we get
\begin{multline*}
\phi(\s)(z\mathfrak c^{\alpha_l}a_{out})=\Delta\!\(\phi(\s')\(\mathfrak c^{\alpha_l}\,\frac{z-z'}2\,a\)\)+\mathfrak c^{\alpha_l}\Delta\!\(\phi(\s')\(\frac{z+z'}2\,a\)\)\\[8pt]
\hspace{-170pt}=\Delta\!\(\mathfrak c^{\alpha_l}\,\frac{z-z'}2\,\phi(\s')(a)\)+\mathfrak c^{\alpha_l}\Delta\!\(\frac{z+z'}2\,\phi(\s')(a)\)\\[8pt]
=\mathfrak c^{\alpha_l}\,\frac{z-z'}2\,\Delta\!\(\phi(\s')(a)\)+\mathfrak c^{\alpha_l}\frac{z+z'}2\,\Delta\!\(\phi(\s')(a)\)
=z\mathfrak c^{\alpha_l}\Delta\!\(\phi(\s')(a)\)=z\phi(\s)(\mathfrak c^{\alpha_l}a_{out}).
\end{multline*}
These calculations prove that $\phi(\s)$ is a $\mathcal Z(J(\s))$-homomorphism.
Definition~\ref{definition:7}\ref{definition:7:part:2} implies that $\phi(\s)$ is surjective.
Finally, $\phi(\s)$ is injective in view of the inductive hypothesis and decomposition~(\ref{eq:19}),
since the first formula of~(\ref{eq:4}) defines an $\iota$-invariant element,
the second formula of~(\ref{eq:4}) defines an $\iota$-antiinvariant element, $\Delta$ is injective
and $\mathfrak c^{\alpha_l}$ takes only nonzero values.

It remains to define isomorphisms $\phi_x(\s)$ for any $x\in\widehat\W$. There is actually no choice
as to how we cant do it. Take any $u\in\M^x$ and extend it to a global section $m\in\Gamma(\B(\s))$.
Then we set $\phi_x(\s)(u):=r^{\X(\s)}_x\circ\phi(\s)(m)=\phi(\s)(m)|_{I(\s)_x}$.

First, we have to prove the correctness of this definition, that is independence of an extension $m$ of $u$.
Since $\phi(\s)$ is linear it suffice to prove that $\phi(\s)(m)|_{I(\s)_x}=0$ if $m_x=0$.
Indeed, let $m=a_{out}+\mathfrak c^{\alpha_l}b_{out}$ for some $a,b\in\Gamma(\B(\s'))$.
Then $0=m_x=(a_x,a_{xs_l})+\overline{x(\alpha_l)}(b_x,b_{xs_l})$.
In other words,
$$
(a+\overline{x(\alpha_l)}b)_x=0,\qquad (a+\overline{x(\alpha_l)}b)_{xs_l}=0.
$$
By these formulas, the formula $(I(\s)_x)'=I(\s')_x\dotcup I(\s')_{xs_l}$ (following from~(\ref{eq:1.25})), formula~(\ref{eq:4})
and the inductive hypothesis , we have
$$
\phi(\s)(m)_\sigma=\phi(\s')(a)_{\sigma'}+\overline{x(\alpha_l)}\phi(\s')(b)_{\sigma'}=\phi(\s')\Bigl(a+\overline{x(\alpha_l)}b\Bigr)_{\!\sigma'}=0.
$$
for any $\sigma\in I(\s)_x$.

It is obvious that $\phi_x(\s)$ is an $S$-module homomorphism and surjective. It only remains to prove that $\phi_x(\s)$
is injective. Take any $u\in\ker\phi_x(\s)$. Let us extend it to a section $m\in\Gamma(\B(\s))$.
As usual, we write $m=a_{out}+\mathfrak c^{\alpha_l}b_{out}$ for some $a,b\in\Gamma(\B(\s'))$.
According to our definitions $\phi(\s)(m)|_{I(\s)_x}=0$. By~(\ref{eq:4}), we get
$$
0=\phi(\s)(m)_{\sigma\sp}=\phi(\s')(a)_\sigma+\overline{x(\alpha_l)}\phi(\s')(b)_\sigma=\phi(\s')\Bigl(a+\overline{x(\alpha_l)}b\Bigr)_{\!\sigma}
$$
for any $\sigma\in I(\s')_x$. Similarly,
$$
0=\phi(\s)(m)_{\sigma s_l}=\phi(\s')(a)_\sigma+\overline{x(\alpha_l)}\phi(\s')(b)_\sigma=\phi(\s')\Bigl(a+\overline{x(\alpha_l)}b\Bigr)_{\!\sigma}
$$
for any $\sigma\in I(\s')_{xs_l}$. These formulas imply
$$
\begin{array}{l}
0=r_x^{\X(\s')}\circ \phi(\s')\Bigl(a+\overline{x(\alpha_l)}b\Bigr)=\phi_x(\s')\circ r^{\B(\s')}_x\Bigl(a+\overline{x(\alpha_l)}b\Bigr)=\phi_x(\s')\Bigl(a_x+\overline{x(\alpha_l)}b_x\Bigr)\\[12pt]
0=r_{xs_l}^{\X(\s')}\circ \phi(\s')\Bigl(a+\overline{x(\alpha_l)}b\Bigr)=\phi_{xs_l}(\s')\circ r^{\B(\s')}_{xs_l}\Bigl(a+\overline{x(\alpha_l)}b\Bigr)=\phi_{xs_l}(\s')\Bigl(a_{xs_l}+\overline{x(\alpha_l)}b_{xs_l}\Bigr).
\end{array}
$$
By induction, we already know that $\phi_x(\s')$ and $\phi_{xs_l}(\s')$ are isomorphisms, so $a_x+\overline{x(\alpha_l)}b_x=0$ and $a_{xs_l}+\overline{x(\alpha_l)}b_{xs_l}=0$
Now we can write
$$
u=m_x=(a_x,a_{xs_l})+\overline{x(\alpha_l)}(b_x,b_{xs_l})=(a_x+\overline{x(\alpha_l)}b_x,a_{xs_l}+\overline{x(\alpha_l)}b_{xs_l})=0.
$$
\end{proof}

We needed the condition that $\phi(\s)$ is a $\mathcal Z(J(\s))$-module isomorphism only for the sake of induction.
In what follows, we shall only use the fact that $\phi(\s)$ is a $\mathcal Z$-module isomorphism.

\begin{corollary}\label{corollary:2}
Let $\s$ be a sequence in $\widehat\S$ such that $\widehat\G_{J(\s)}$ satisfies the GKM-property.
For any element $x\in\widehat\W$, the isomorphism $\phi(\s)_x$ in Theorem~\ref{theorem:3} restricts to the isomorphism
$\B(\s)_x\stackrel{\sim}\to\X(\s)_x$.
\end{corollary}
\begin{proof}
Take any $u\in\B(\s)_x$. Set $m_x:=u$ and $m_y:=0$ for any $y\ne x$. Then $m\in\Gamma(\B(\s))$.
By the commutativity of the diagram in Theorem~\ref{theorem:3}, we get $\phi(\s)(m)|_{I(\s)_y}=0$ for any $y\ne x$.
Hence $\phi(\s)(m)|_{I(\s)_x}\in\X(\s)_x$.

On the other hand, take any $f\in\X(\s)_x$. Then $f=g|_{I(\s)_x}$ for some $g\in\X(\s)$
such that $g|_{I(\s)\setminus I(\s)_x}=0$. Consider the section $m:=\phi(\s)^{-1}(g)$.
By the commutativity of the diagram in Theorem~\ref{theorem:3}, we get
$\phi_y(\s)(m_y)=0$, whence $m_y=0$, as  $\phi_y(\s)$ is an isomorphism.
Thus we have proved that $m_x\in\B(\s)_x$. We have $\phi_x(\s)(m_x)=f$.

Therefore, $\phi_x(\s)$ restricts to the  surjective homomorphism $\B(\s)_x\to\X(\s)_x$.
It is an isomorphism as $\phi_x(\s)$ has zero kernel.
\end{proof}

\begin{corollary}\label{corollary:3}
Let $\s=(s_1,\ldots,s_l)$ be a sequence in $\widehat\S$ such that $\widehat\G_{J(\s)}$ satisfies the GKM-property.
Suppose that $\edgeright y\alpha x$ is an edge of $\widehat\G$
and $f\in\X(\s)$ is such that $f|_{I(\s)_y}=0$.
Then $f|_{I(\s)_x}=\alpha\,g|_{I(\s)_x}$ for some $g\in\X(\s)$.
\end{corollary}
\begin{proof}
Consider the section $m:=\phi(\s)^{-1}(f)\in\Gamma(\B(\s))$. By Theorem~\ref{theorem:3}, we obtain $\phi_y(\s)(m_y)=0$.
Since $\phi(\s)$ is an isomorphism, we have $m_y=0$. Therefore, $\rho_{x,E}(m_x)=\rho_{y,E}(m_y)=0$.
Condition~\ref{proj:3} of Definition~\ref{definition:5} implies $m_x=\alpha u$ for some $u\in\B(\s)^x$.
Let us extend $u$ to a global section $n\in\Gamma(\B(\s))$. Then we have $(m-\alpha n)_x=0$.
Thus by Theorem~\ref{theorem:3}, we get $r^{\X(\s)}_x\circ\phi(\s)(m-\alpha n)=0$.
Hence $\bigl(f-\alpha\phi(\s)(n)\bigr)|_{I_x(\s)}=0$ and we can take $g:=\phi(\s)(n)$.
\end{proof}

{\bf Remark.} In this corollary, it suffices to require only that $\widehat\G_{J(\s')}$ satisfy the GKM-property.
There is a direct (that is without using Theorem~\ref{theorem:3}) proof of this stronger version.

\subsection{Extended bases of $\X(\s)^x$}\label{extended_bases}
This notion is useful to compute an actual $S$-basis of $\X(\s)^x$ in Section~\ref{basis_of_X}.

\begin{definition}\label{definition:8}
Let $\s$ be a sequence in $\widehat\S$ and $\I\subset\widehat\W$ be a subset such that the graded $S$-module $\mathcal X(\s)^\I$ is free.
A list of homogeneous functions $f_1,\ldots,f_m\in\mathcal X(\s)$ is called an extended basis of $\mathcal X(\s)^\I$
if the restrictions $f_1|_{I(\s)_\I},\ldots,f_m|_{I(\s)_\I}$ form an $S$-basis of $\mathcal X(\s)^\I$.
\end{definition}

In the following lemma and in the rest of the paper, we shall often use together the $S$-module and
the $\mathcal Z$-module structures on $\X(\s)$.

\begin{lemma}\label{lemma:9.5}
Let $\s=(s_1,\ldots,s_l)$ be a sequence in $\widehat\S$ be such that $\widehat\G_{J(\s)}$ satisfies the GKM-property
Take $x\in\widehat\W$ such that $x<xs_l$ and choose $\alpha_l\in\widehat\Pi$ so that $s_l=s_{\alpha_l}$.
If $f_1,\ldots,f_m$ is an extended basis of $\mathcal X(\s')^x$ and
$g_1,\ldots,g_k$ is an extended basis of $\mathcal X(\s')^{xs_l}$, then
$$
f_1,\ldots,f_m,\frac{\overline{x(\alpha_l)}-c^{\alpha_l}}2\,g_1,\ldots,\frac{\overline{x(\alpha_l)}-c^{\alpha_l}}2\,g_k
$$
is an extended basis of $\mathcal X(\s')^{\{x,xs_l\}}$.
\end{lemma}
\begin{proof}
First, we prove that the above elements generate  $\mathcal X(\s)^{\{x,xs_l\}}$ as an $S$-mo\-dule.
Take an arbitrary function $f\in\mathcal X(\s')$. By Definition~\ref{definition:8}, there exist homogeneous
$a_1,\ldots,a_m\in S$ such that
\begin{equation}\label{eq:10.5}
f|_{I(\s')_x}=(a_1f_1+\cdots+a_mf_m)|_{I(\s')_x}.
\end{equation}
As we have the edge $\edgeright x{\pm\overline{x(\alpha_l)}}{xs_l}$, the Corollary~\ref{corollary:3}
applied to it and the function $f-a_1f_1-\cdots-a_mf_m$ implies the existence of $g\in\mathcal X(\s)$ such that
\begin{equation}\label{eq:11}
(f-a_1f_1-\cdots-a_mf_m)|_{I(\s')_{xs_l}}=\overline{x(\alpha_l)}\,g|_{I(\s')_{xs_l}}.
\end{equation}
Now by Definition~\ref{definition:8}, there exist homogeneous $b_1,\ldots,b_k\in S$ such that
\begin{equation}\label{eq:12}
g|_{I(\s')_{xs_l}}=(b_1g_1+\cdots+b_kg_k)|_{I(\s')_{xs_l}}.
\end{equation}
We claim
$$
f\Bigl|_{I(\s')_x\dotcup I(\s')_{xs_l}}
=\!\!\(a_1f_1+\cdots+a_mf_m+b_1\frac{\overline{x(\alpha_l)}-c^{\alpha_l}}2\,g_1+\cdots+b_k\frac{\overline{x(\alpha_l)}-c^{\alpha_l}}2\,g_k\)\!\!\Biggl|_{I(\s')_x\dotcup I(\s')_{xs_l}}.
$$
Indeed, take any $\sigma\in I(\s')_x$. The right-hand side evaluated at $\sigma$
equals $a_1(f_1)_\sigma+\cdots+a_m(f_m)_\sigma$. By~(\ref{eq:10.5}), this sum is exactly $f_\sigma$.
Now take any $\sigma\in I(\s')_{xs_l}$. Then the right-hand side evaluated at $\sigma$ equals
$$
a_1(f_1)_\sigma+\cdots+a_m(f_m)_\sigma+\overline{x(\alpha_l)}\,b_1g_1+\cdots+\overline{x(\alpha_l)}\,b_kg_k.
$$
By~(\ref{eq:11}) and~(\ref{eq:12}), we again get $f_\sigma$.

It remains to prove that the restrictions to $I(\s)_x\dotcup I(\s)_{xs_l}$ of our functions are independent over $S$.
Suppose that
$$
\Bigl(a_1f_1+\cdots+a_mf_m+b_1\frac{\overline{x(\alpha_l)}-c^{\alpha_l}}2\,g_1+\cdots+b_k\frac{\overline{x(\alpha_l)}-c^{\alpha_l}}2\,g_k\Bigr)\Bigl|_{I(\s')_x\dotcup I(\s')_{xs_l}}=0.
$$
Restricting to $I(\s')_x$, we get
$$
(a_1f_1+\cdots+a_mf_m)|_{I(\s')_x}=0,
$$
whence $a_1=\cdots=a_m=0$ by Definition~\ref{definition:8}.
Now restricting to $I(\s')_{xs_l}$, we get
$$
\overline{x(\alpha_l)}\,(b_1g_1+\cdots+b_kg_k)|_{I(\s')_{xs_l}}=0,
$$
whence $b_1=\cdots=b_k=0$ by the same definition.
\end{proof}

\begin{lemma}\label{lemma:10}
Let $\s=(s_1,\ldots,s_l)$ be a sequence in $\widehat\S$ such that
$\widehat\G_{J(\s)}$ satisfies the GKM-property. Take $x\in\widehat\W$ such that $x<xs_l$
and choose $\alpha_l\in\widehat\Pi$ so that $s_l=s_{\alpha_l}$.
Let $f_1,\ldots,f_m$ and $g_1,\ldots,g_k$
be extended bases of $\X(\s')^x$ and $\X(\s')^{xs_l}$, respectively.
Then
$$
\Delta(f_1),\,\ldots,\Delta(f_m),\,\Delta\!\(\frac{c^{\alpha_l}-\overline{x(\alpha_l)}}2\,g_1\)\!,\,\ldots,\,\Delta\!\(\frac{c^{\alpha_l}-\overline{x(\alpha_l)}}2\,g_k\),
$$
is an extended basis of both $\X(\s)^x$ and $\X(\s)^{xs_l}$.
\end{lemma}
\begin{proof}We are going to prove the statement for $\mathcal X(\s)^x$ and indicate which alterations should be made
for $\X(\s)^{xs_l}$. We set for brevity
$$
f_{m+i}:=\frac{\overline{x(\alpha_l)}-c^{\alpha_l}}2\,g_i
$$
for $i=1,\ldots,k$ and shall prove that $\Delta(f_1),\ldots,\Delta(f_{m+k})$
is an extended basis of both $\X(\s)^x$ and $\X(\s)^{xs_l}$. This statement is equivalent
to the original one, as we only multiplied the last $k$ elements by $-1$.

First, we prove that the restrictions $f_1|_{I(\s)_x},\ldots,f_{m+k}|_{I(\s)_x}$ generate
$\X(\s)^x$ as an $S$-module. Take an arbitrary function $f\in\mathcal X(\s)$ and write it as
$f=\Delta(a)+c^{\alpha_l}\Delta(b)$ for suitable $a,b\in\X(\s')$.
Consider the function $f^+=a+\overline{x(\alpha_l)}\,b$.
We claim
\begin{equation}\label{eq:13}
f\bigl|_{I(\s)_x}=\Delta(f^{+})\bigl|_{I(\s)_x}.
\end{equation}
Indeed $I(\s)_x=I(\s')_x\sp\dotcup I(\s')_{xs_l}s_l$ by~(\ref{eq:20}).
First, take an arbitrary $\sigma\in I(\s')_x$. Then
$$
f_{\sigma\sp}=\Delta(a)_{\sigma\sp}+\bigl(c^{\alpha_l}\Delta(b)\bigr)_{\sigma\sp}
=\Delta(a)_{\sigma\sp}+\overline{\ev(\sigma\sp)(\alpha_l)}\,\Delta(b)_{\sigma\sp}
=a_\sigma+\overline{x(\alpha_l)}\,b_\sigma=f^+_\sigma=\Delta(f^+)_{\sigma\sp}.
$$
Now take an arbitrary $\sigma\in I(\s')_{xs_l}$.
Then
$$
f_{\sigma s_l}=\Delta(a)_{\sigma s_l}+\bigl(c^{\alpha_l}\Delta(b)\bigr)_{\sigma s_l}\!
=\Delta(a)_{\sigma s_l}+\overline{\ev(\sigma s_l)(\alpha_l)}\,\Delta(b)_{\sigma s_l}
=a_\sigma+\overline{x(\alpha_l)}\,b_\sigma=f^+_\sigma=\Delta(f^+)_{\sigma s_l}.
$$
Lemma~\ref{lemma:9.5} implies that there exist homogeneous $a_1,\ldots,a_{m+k}\in S$ such that
\begin{equation}\label{eq:14}
f^+\Bigl|_{I(\s')_x\dotcup I(\s')_{xs_l}}=\Bigl(a_1f_1+\cdots+a_{m+k}f_{m+k}\Bigr)\Bigl|_{I(\s')_x\dotcup I(\s')_{xs_l}}.
\end{equation}
Hence
\begin{equation}\label{eq:14.5}
\Delta(f^+)\Bigl|_{I(\s)_x}=\Bigl(a_1\Delta(f_1)+\cdots+a_{m+k}\Delta(f_{m+k})\Bigr)\Bigl|_{I(\s)_x}.
\end{equation}
Indeed, take any $\sigma\in I(\s')_x$. Then by~(\ref{eq:14}), we get
$$
\Delta(f^+)_{\sigma\sp}=f^+_\sigma=a_1(f_1)_\sigma+\cdots+a_{m+k}(f_{m+k})_\sigma
=\Bigl(a_1\Delta(f_1)+\cdots+a_{m+k}\Delta(f_{m+k})\Bigl)_{\sigma\sp}.
$$
Now take any $\sigma\in I(\s')_{xs_l}$. Then by~(\ref{eq:14}), we get
$$
\Delta(f^+)_{\sigma s_l}=f^+_\sigma=a_1(f_1)_\sigma+\cdots+a_{m+k}(f_{m+k})_\sigma
=\Bigl(a_1\Delta(f_1)+\cdots+a_{m+k}\Delta(f_{m+k})\Bigl)_{\sigma s_l}.
$$
It remains now to combine~(\ref{eq:13}) and~(\ref{eq:14.5}).

Now let us prove that $\Delta(f_1)|_{I(\s)_x},\ldots,\Delta(f_{m+k})|_{I(\s)_x}$ are linearly independent over $S$.
Suppose that
$$
a_1\Delta(f_1)|_{I(\s)_x}+\cdots a_{m+k}\Delta(f_{m+k})|_{I(\s)_x}=0
$$
for some $a_1,\ldots,a_{m+k}\in S$. Our computations in the proof of~(\ref{eq:14.5}) show that
$$
\Bigl(a_1f_1+\cdots+a_{m+k}f_{m+k}\Bigr)\Bigl|_{I(\s')_x\dotcup I(\s')_{xs_l}}=0.
$$
Hence by Lemma~\ref{lemma:9.5}, we get $a_1=\cdots=a_{m+k}=0$.

To prove the statement for $\X(\s)^{xs_l}$, we have to consider the function $f^-=a-\overline{x(\alpha_l)}\,b$
and prove the following analog of~(\ref{eq:13}):
$$
f\bigl|_{I(\s)_{xs_l}}=\Delta(f^-)\bigl|_{I(\s)_{xs_l}}.
$$
Then we apply Lemma~\ref{lemma:9.5} to obtain a decomposition of $f^-$ and proceed as for $\X(\s)^x$.
\end{proof}


Clearly, $1$ is the extended basis of $\X(\emptyset)^e\cong S$ and the empty list is a basis of
$\X(\emptyset)^x=0$ if $x\ne e$. Thus Lemma~\ref{lemma:10} allows us to construct inductively the extended
basis $\{\b(\s)^x_\sigma\}_{\sigma\in I(\s)_x}$ of each module $\X(\s)^x$.

The first thing obvious from this inductive construction is that the number of elements of
this extended basis equals the number of elements of $I(\s)_x$. This observation prompts us to parameterize elements of
our extended basis for $\X(\s)^x$ by elements of $I(\s)_x$. We set
$$
\b(\emptyset)^e_\emptyset:=1.
$$
Suppose now that
$x<xs_l$
and that we have already built elements $\{\b(\s')^x_\sigma\}_{\sigma\in I(\s')_x}$ of $\X(\s')$
and $\{\b(\s')^{xs_l}_\sigma\}_{\sigma\in I(\s')_{xs_l}}$
Then we set
\begin{align}
\b(\s)^x_{\sigma\sp}&:=\Delta(\b(\s')^x_\sigma)\text{ for }\sigma\in I(\s')_x;\label{eq:b:1}\\[3pt]
\b(\s)^x_{\sigma s_l}&:=\displaystyle\Delta\!\(\frac{c^{\alpha_l}-\overline{x(\alpha_l)}}2\;\b(\s')^{xs_l}_\sigma\!\)\text{ for }\sigma\in I(\s')_{xs_l};\label{eq:b:2}\\[3pt]
\b(\s)^{xs_l}_{\sigma s_l}&:=\Delta(\b(\s')^x_\sigma)\text{ for }\sigma\in I(\s')_x;\label{eq:b:3}\\[3pt]
\b(\s)^{xs_l}_{\sigma\sp}&:=\displaystyle\Delta\!\(\frac{c^{\alpha_l}-\overline{x(\alpha_l)}}2\;\b(\s')^{xs_l}_\sigma\!\)\text{ for }\sigma\in I(\s')_{xs_l}.\label{eq:b:4}
\end{align}


%
%
%

\begin{corollary}\label{corollary:4}
Let $\s$ be a sequence in $\widehat\S$ such that
$\widehat\G_{J(\s)}$ satisfies the GKM-property. Then $\{\b(\s)^x_\sigma\}_{\sigma\in I(\s)_x}$
is an extended basis of $\X(\s)^x$ for any $x\in\widehat\W$.
\end{corollary}

\subsection{Trees}\label{trees} There is a simple graphical interpretation of the above definition of $\b(\s)^x_\sigma$,
which we describe here with the help of trees.

In the remaining part of Section~\ref{Bases_for_Bott-–Samelson_sheaves},
we usually assume that $\s=(s_1,\ldots,s_l)$ is a sequence in $\widehat\S$ and $s_i=s_{\alpha_i}$,
where $\alpha_i\in\widehat\Pi$. 

For each element $x\in\widehat\W$,
we define the oriented tree $T(\s,x)$ inductively on the length of~$\s$.
Every edge $\gamma$ of this tree will be labelled by its color $\c(\gamma)\in\{-1,0,1\}$
and root $\r(\gamma)\in\widehat\Pi$.

First, we define $T(\emptyset,x)$ to be the tree with only one vertex\,
\begin{picture}(0,0)
\put(0,3){\circle*{4}}
\end{picture}\,
(and no edges) if $x=e$ and
to be the empty tree otherwise.

Now suppose that $l>0$. We write $\{x,xs_l\}=\{y,z\}$ where $y<z$.
First, we consider the case where both sets $I(\s')_x$ and $I(\s')_{xs_l}$ are nonempty.
Then we define

\begin{center}
\begin{picture}(0,0)
\put(-104,-18){$T(\s,x)\!:=$}
\vertexwithtwoarrows[1\alpha_l][0\alpha_l]
\put(-47,-41){$T(\s',y)$}
\put(13,-42){$T(\s',z)$}
\end{picture}
\end{center}

\vspace{43pt}

\noindent
if $xs_l=y<z=x$ and

\begin{center}
\begin{picture}(0,0)
\put(-104,-18){$T(\s,x)\!:=$}
\vertexwithtwoarrows[0\alpha_l][{-1}\alpha_l]
\put(-47,-41){$T(\s',y)$}
\put(13,-42){$T(\s',z)$}
\end{picture}
\end{center}

\vspace{43pt}

\noindent
if $x=y<z=xs_l$.

Now consider the case where $I(\s')_x=\emptyset$ or $I(\s')_{xs_l}=\emptyset$.
By (\ref{eq:20}), it is clear that $I(\s)_x=\emptyset$ if both these sets are empty and
we define $T(\s,x)$ to be the empty tree in this case.

Suppose that exactly one of these sets is empty.
Then we claim that $I(\s')_y\ne\emptyset$ and $I(\s')_z=\emptyset$.
Indeed suppose that on the contrary $I(\s')_y=\emptyset$ and $I(\s')_z\ne\emptyset$.
Then $z=s_{i_1}\cdots s_{i_k}$ for some integers $i_1,\ldots,i_m$ such that $1\le i_1<\cdots<i_m<l$.
As $zs_l=y<z$, the exchange property implies $y=s_{i_1}\cdots \widehat s_{i_k}\cdots s_{i_m}$ for some $k=1,\ldots,m$.
Thus
$I(\s')_y$ contains the sequence obtained from $(\sp,s_{i_1},\sp,\ldots,\sp,s_{i_m},\sp)$ by replacing
its ${i_k}^{\!\text{th}}$ entry with $\sp$, contrary to our assumption $I(\s')_y=\emptyset$.
Then we define

\begin{center}
\begin{picture}(0,0)
\put(-74,-20){$T(\s,x)\!:=$}
\vertexwithonearrow[0\alpha_l]
\put(-16,-41){$T(\s',y)$}
\end{picture}
\end{center}

\vspace{43pt}

\noindent
if $x=y$ and

\begin{center}
\begin{picture}(0,0)
\put(-74,-20){$T(\s,x)\!:=$}
\vertexwithonearrow[1\alpha_l]
\put(-16,-41){$T(\s',y)$}
\end{picture}
\end{center}

\vspace{43pt}

\noindent
if $x=z$.

Our inductive construction allows us to divide edges into {\it levels}, by saying that the uppermost edge
in our pictures above is of level $l$.
%
%
Note that $\r(\gamma)=\alpha_i$ for any edge of level~$i$.



We have similar definitions for vertices. A vertex $a$ is called a vertex of {\it level $i$} if there is an edge
of level $i$ ending at $a$. Vertices having no edges ending at them are of level $0$ and called {\it leaves}.
Moreover, we can associate to vertex $a$ the element $\ev(a)\in\widehat\W$
inductively by saying that $\ev$ applied to the uppermost vertex in our pictures above yields~$x$.

We call the unique vertex of $T(\s,x)$ having no edges beginning at it the {\it root}.
In other words, the root is the unique vertex of maximal level.
Each leaf is connected to the root by
a unique (oriented) path, which we call a {\it maximal path}. To any path $\pi$,
we associate the sequence $\c(\pi)$ of $0$, $1$ and $-1$ just by reading the colors of edges along it.
We shall also need the sequence $[\pi]$ that is obtained from $\pi$ as follows:
$$
[\emptyset]=\emptyset,\qquad [\pi_1\cdots \pi_i]=\left\{
\arraycolsep=0pt
\begin{array}{ll}
[\pi_1\cdots \pi_{i-1}]\,s_i&\;\text{ if }\c(\pi_i)\ne0;\\[6pt]
[\pi_1\cdots \pi_{i-1}]\,\sp &\;\text{ if }\c(\pi_i)=0
\end{array}
\right.
\qquad\text{ for }i>0.
$$
In other words, to get $[\pi]$, we start with the empty sequence and then go along $\pi$, adding $s_{\r(\pi_i)}=s_i$
to our sequence on the right every time we meet an edge of level $i$ having nonzero color and adding $\sp$ otherwise.


To each path $\pi$ in $T(\s,x)$, we can also associate an element of $\widehat\W$ just by evaluating $[\pi]$.
In this connection, we will use the abbreviation $\ev(\pi):=\ev([\pi])$.
We shall also use the notation $(\pi_1\cdots\pi_i)':=\pi_1\cdots\pi_{i-1}$ to indicate truncation of paths
(the loss of the last edge).

The set $I(\s)_x$ together with the lengths of $\ev(\tau)$, where $\tau$ is a beginning of some $\sigma\in I(\s)_x$,
can be read off the tree $T(\s,x)$ as follows.
\begin{proposition}\label{proposition:6}
\begin{enumerate}


\item\label{proposition:6:part:1} $\pi\mapsto[\pi]$ sets a one-to-one correspondence between the set of maximal paths in $T(\s,x)$ and the set $I(\s)_x$.\\[-8pt]

\item\label{proposition:6:part:2} $\ev(\pi)$ has length $\sum\c(\pi)$ for any path $\pi$ in $T(\s,x)$ starting at a leaf.\\[-8pt]

\item\label{proposition:6:part:3} Let $a$ be a vertex of $T(\s,x)$ of level $i$.
         Then the full subtree $T$ of $T(\s,x)$ with root $a$ is equal to $T((s_1,\ldots,s_i),\ev(a))$.
         In particular, $\ev(\pi)=\ev(a)$ for all paths $\pi$ starting at a leaf and ending at $a$.
\end{enumerate}
\end{proposition}

The first statement of part~\ref{proposition:6:part:3} is of principal importance,
as it allows us to apply induction on the length of $\s$ for our trees.
The identification of part~\ref{proposition:6:part:1} allows us to apply functions $f\in\bigoplus_{\sigma\in I(\s)_x}Q$
to maximal paths $\pi$ by $f(\pi):=f([\pi])$. The same formula identifies $\bigoplus_{\sigma\in I(\s)_x}Q$
with the set of functions mapping maximal paths $\pi$ in $T(\s,x)$ to $Q$.

\begin{example}\label{example:2}\rm Let $\s=(s_1,s_2,s_1,s_2,s_1)$, where $s_1$ and $s_2$ are simple reflections in the Weyl group of type $A_2$,
and $x=s_2s_1$. The tree $T(\s,x)$ is as follows

\vspace{120pt}

\def\cl{green}

\begin{center}
\begin{picture}(200,0)
\put(0,0){\circle*{4}}
\put(0,0){\vector(0,1){24}}
\put(0,25){\circle*{4}}
\put(-5.5,9){$\scriptstyle0$}
\put(2,9){$\scriptstyle s_1$}
\put(-8,-2){$\scriptstyle\textcolor{\cl} e$}
\put(-8,23){$\scriptstyle\textcolor{\cl} e$}
\put(50,0){\circle*{4}}
\put(50,0){\vector(0,1){24}}
\put(50,25){\circle*{4}}
\put(44.5,9){$\scriptstyle1$}
\put(52,9){$\scriptstyle s_1$}
\put(42,-2){$\scriptstyle\textcolor{\cl} e$}
\put(39,23){$\scriptstyle\textcolor{\cl}{s_1}$}
\put(100,0){\circle*{4}}
\put(100,0){\vector(0,1){24}}
\put(100,25){\circle*{4}}
\put(94.5,9){$\scriptstyle0$}
\put(102,9){$\scriptstyle s_1$}
\put(92,-2){$\scriptstyle\textcolor{\cl} e$}
\put(92,23){$\scriptstyle\textcolor{\cl} e$}
\put(150,0){\circle*{4}}
\put(150,0){\vector(0,1){24}}
\put(150,25){\circle*{4}}
\put(144.5,9){$\scriptstyle0$}
\put(152,9){$\scriptstyle s_1$}
\put(142,-2){$\scriptstyle\textcolor{\cl} e$}
\put(142,23){$\scriptstyle\textcolor{\cl} e$}
\put(200,0){\circle*{4}}
\put(200,0){\vector(0,1){24}}
\put(200,25){\circle*{4}}
\put(194.5,9){$\scriptstyle1$}
\put(202,9){$\scriptstyle s_1$}
\put(192,-2){$\scriptstyle\textcolor{\cl} e$}
\put(188,23){$\scriptstyle\textcolor{\cl}{s_1}$}
\put(0,25){\vector(0,1){24}}
\put(0,50){\circle*{4}}
\put(-5.5,34){$\scriptstyle0$}
\put(2,34){$\scriptstyle s_2$}
\put(-8,48){$\scriptstyle\textcolor{\cl} e$}
\put(50,25){\vector(0,1){24}}
\put(50,50){\circle*{4}}
\put(44.5,34){$\scriptstyle0$}
\put(52,34){$\scriptstyle s_2$}
\put(39,48){$\scriptstyle\textcolor{\cl}{s_1}$}
\put(100,25){\vector(0,1){24}}
\put(100,50){\circle*{4}}
\put(94.5,34){$\scriptstyle1$}
\put(102,34){$\scriptstyle s_2$}
\put(89,48){$\scriptstyle\textcolor{\cl}{s_2}$}
\put(150,25){\vector(0,1){24}}
\put(150,50){\circle*{4}}
\put(144.5,34){$\scriptstyle1$}
\put(152,34){$\scriptstyle s_2$}
\put(139,48){$\scriptstyle\textcolor{\cl}{s_2}$}
\put(200,25){\vector(0,1){24}}
\put(200,50){\circle*{4}}
\put(194.5,34){$\scriptstyle1$}
\put(202,34){$\scriptstyle s_2$}
\put(180,48){$\scriptstyle\textcolor{\cl}{s_1s_2}$}
\put(25,75){\circle*{4}}
\put(19,77){$\scriptstyle\textcolor{\cl} e$}
\put(0,50){\vector(1,1){24.2}}
\put(50,50){\vector(-1,1){24.2}}
\put(7,63){$\scriptstyle0$}
\put(13,58){$\scriptstyle s_1$}
\put(28,57){$\scriptstyle\text{--}1$}
\put(40,63){$\scriptstyle s_1$}
\put(100,50){\vector(0,1){24}}
\put(100,75){\circle*{4}}
\put(94.5,59){$\scriptstyle0$}
\put(102,59){$\scriptstyle s_1$}
\put(89,73){$\scriptstyle\textcolor{\cl}{s_2}$}
\put(150,50){\vector(0,1){24}}
\put(150,75){\circle*{4}}
\put(144.5,59){$\scriptstyle1$}
\put(152,59){$\scriptstyle s_1$}
\put(131,73){$\scriptstyle\textcolor{\cl}{s_2s_1}$}
\put(200,50){\vector(0,1){24}}
\put(200,75){\circle*{4}}
\put(194.5,59){$\scriptstyle1$}
\put(202,59){$\scriptstyle s_1$}
\put(172,73){$\scriptstyle\textcolor{\cl}{s_1s_2s_1}$}
\put(62.5,100){\circle*{4}}
\put(25,75){\vector(3,2){37}}
\put(100,75){\vector(-3,2){37}}
\put(39,89){$\scriptstyle1$}
\put(44,82){$\scriptstyle s_2$}
\put(82,89){$\scriptstyle s_2$}
\put(75,82){$\scriptstyle 0$}
\put(52,102){$\scriptstyle\textcolor{\cl}{s_2}$}
\put(175,100){\circle*{4}}
\put(150,75){\vector(1,1){24.2}}
\put(200,75){\vector(-1,1){24.2}}
\put(157,88){$\scriptstyle0$}
\put(163,83){$\scriptstyle s_2$}
\put(178,82){$\scriptstyle\text{--}1$}
\put(190,88){$\scriptstyle s_2$}
\put(178,102){$\scriptstyle\textcolor{\cl}{s_2s_1}$}
\put(118.25,118.75){\circle*{4}}
\put(175,100){\vector(-3,1){55}}
\put(62.5,100){\vector(3,1){55}}
\put(111.5,123.5){$\scriptstyle\textcolor{\cl}{s_2s_1}$}
\put(89,112.5){$\scriptstyle1$}
\put(92.5,104){$\scriptstyle s_1$}
\put(144,113){$\scriptstyle s_1$}
\put(142,103){$\scriptstyle 0$}
\end{picture}
\end{center}

\smallskip
\noindent
In this picture, we have also attached $\ev(a)$ to every vertex $a$. The operation $[\cdot]$ applied
to the maximal paths read from left to right yields the following sequences:
$(\sp,\sp,\sp,s_2,s_1)$, $(s_1,\sp,s_1,s_2,s_1)$, $(\sp,s_2,\sp,\sp,s_1)$, $(\sp,s_2,s_1,\sp,\sp)$,
$(s_1,s_2,s_1,s_2,\sp)$, which as one can easily see comprise $I(\s)_x$, i.e. all ways to obtain
$x$ from subsequences of $\s$.

\end{example}

The tree $T(\s,x)$ allows us to compute the extended basis $\b(\s)^x_\sigma$ constructed in Section~\ref{extended_bases}.
To do it, we classify edges as follows:

\hspace{70pt}
\begin{picture}(0,0)
\put(0,-30){\vector(0,1){28}}
\put(-16,-43){\small vertical}
\put(140,-30){\vector(1,1){28.6}}
\put(117,-43){\small right tilted}
\put(322,-30){\vector(-1,1){28.6}}
\put(293,-43){\small left tilted}
\end{picture}

\vspace{50pt}

We define elements $\b_\pi\in\X((s_1,\ldots,s_i))$ for all paths $\pi$ in $T(\s,x)$ of length $i$
starting at a leaf as follows.

\begin{definition}\label{definition:9} We set $\b_\emptyset:=1$. Let $\pi=\pi_1\cdots\pi_i$ be a path in $T(\s,x)$
of length $i>0$ starting at a leaf.
We define
$$
\b_\pi:=\Delta\bigl(\b_{\pi'}\bigr)
$$
if $\pi_i$ is vertical or right tilted and
$$
\b_\pi:=\Delta\!\left(\frac{{c^{\,\r(\pi_i)}+\overline{\ev(\pi')(\r(\pi_i))}}}2\,\b_{\pi'}\!\right)
$$
if $\pi_i$ is left tilted.
\end{definition}
\noindent
According to our convention, $\r(\pi_i)=\alpha_i$. We shall use this notation latter in our proofs.

\begin{lemma}\label{lemma:4}
$\b(\s)^x_{[\pi]}=\b_\pi$ for all maximal paths $\pi$ in $T(\s,x)$.
\end{lemma}
\begin{proof}
Induction on the length of $\s$ (the length of $\pi$). First suppose that $\s=\emptyset$. It suffices to consider
the case $x=e$, since otherwise $T(\s,x)$ is empty and there are no paths to consider.
By our definitions here and in Section~\ref{extended_bases}, we have $\b(\emptyset)^e_\emptyset=\b_\emptyset=1$.

Now suppose that the length of $\s$ is nonzero.
Take a maximal path $\pi=\pi_1\cdots \pi_l$ in $T(\s,x)$. 
We set $x':=\ev(\pi')$. So $[\pi']\in I(\s')_{x'}$. By Proposition~\ref{proposition:6}\ref{proposition:6:part:3},
we obtain that $\pi'$ is the maximal path in $T(\s',x')$, whence by the inductive hypothesis
\begin{equation}\label{eq:15}
\b(\s')^{x'}_{[\pi']}=\b_{\pi'}.
\end{equation}

{\it Case~1: $\pi_l$ is vertical or right tilted.} In this case $x'<x's_l$.
Suppose first that $\c(\pi_l)=0$. Then $x=x'$ and $[\pi]=[\pi']\sp$.
Then by~(\ref{eq:b:1}),~(\ref{eq:15}) and Definition~\ref{definition:9}, we get
$$
\b(\s)^x_{[\pi]}=\b(\s)^{x'}_{[\pi']\sp}=\Delta\bigl(\b(\s')^{x'}_{[\pi']}\bigr)
=\Delta\bigl(\b_{\pi'}\bigr)=\b_\pi.
$$
Now suppose that $\c(\pi_l)=1$. Then $x=x's_l$ and $[\pi]=[\pi']\,s_l$.
Then by~(\ref{eq:b:3}),~(\ref{eq:15}) and Definition~\ref{definition:9}, we get
$$
\b(\s)^x_{[\pi]}=\b(\s)^{x's_l}_{[\pi']s_l}=\Delta\bigl(\b(\s')^{x'}_{[\pi']}\bigr)
=\Delta\bigl(\b_{\pi'}\bigr)=\b_\pi.
$$

{\it Case~2: $\pi_l$ is left tilted.} In this case $x's_l<x'$.
Suppose first that $\c(\pi_l)=0$. Then $x=x'$ and $[\pi]=[\pi']\sp$.
Then by~(\ref{eq:b:4}) (where $x=x's_l$),~(\ref{eq:15}) and Definition~\ref{definition:9}, we get
\begin{multline*}
\b(\s)^x_{[\pi]}=\b(\s)^{(x's_l)s_l}_{[\pi']\sp}=\Delta\!\left(\frac{c^{\alpha_l}-\overline{x's_l(\alpha_l)}}2\,\b(\s')^{x'}_{[\pi']}\right)\\[6pt]
=\Delta\!\left(\frac{c^{\alpha_l}+\overline{x'(\alpha_l)}}2\,\b_{\pi'}\right)
=\Delta\!\left(\frac{c^{\alpha_l}+\overline{\ev(\pi')(\alpha_l)}}2\,\b_{\pi'}\right)
=\b_\pi.
\end{multline*}
Now suppose that $\c(\pi_l)=-1$. Then $xs_l=x'$ and $[\pi]=[\pi']s_l$.
Then by~(\ref{eq:b:2}),~(\ref{eq:15}) and Definition~\ref{definition:9}, we get
\begin{multline*}
\b(\s)^x_{[\pi]}=\b(\s)^x_{[\pi']s_l}=\Delta\!\(\frac{c^{\alpha_l}-\overline{x(\alpha_l)}}2\;\b(\s')^{x'}_{[\pi']}\!\)\\[6pt]
=\Delta\!\left(\frac{c^{\alpha_l}+\overline{x'(\alpha_l)}}2\,\b_{\pi'}\right)
=\Delta\!\left(\frac{c^{\alpha_l}+\overline{\ev(\pi')(\alpha_l)}}2\,\b_{\pi'}\right)
=\b_\pi.
\end{multline*}
\end{proof}

\subsection{Basis of $\X(\s)^x$}\label{basis_of_X} To calculate this basis, we introduce the product of paths.

\begin{definition}\label{definition:10}
Let $\pi$ and $\rho$ be two paths in $T(\s,x)$ starting from leaves and having equal lengths.
We set $(\pi,\rho):=\bigl(\b_\pi\bigr)_{[\rho]}$.
\end{definition}
\noindent
Hence and from Definition~\ref{definition:9}, we obtain immediately $(\emptyset,\emptyset)=1$.

Let $a_1,\ldots,a_n$ be the leaves of $T(\s,x)$ labelled from left to right.
Denote by $\pi^{(1)},\ldots,\pi^{(n)}$ the maximal paths connecting these vertices to the root of $T(\s,x)$.
%
From Corollary~\ref{corollary:4} and Lemma~\ref{lemma:4}, we deduce the following fact.

\begin{corollary}\label{corollary:5}
Let $\s$ be a sequence in $\widehat\S$ such that $\widehat\G_{J(\s)}$ satisfies the GKM-property. Then there exists a homogeneous
$S$-basis $v_1,\ldots,v_m$ of $\X(\s)^x$ such that
$v_i(\pi^{(j)})=(\pi^{(i)},\pi^{(j)})$ for any $i,j=1,\ldots,n$.
\end{corollary}
\begin{proof}
It suffices to set $v_i:=\b_{\pi^{(i)}}|_{I(\s)_x}$.
\end{proof}
\noindent
In this corollary, we apply elements $v_i\in\X(\s)^x$ directly to paths $\pi^{(j)}$
in view of the identification of Proposition~\ref{proposition:6}\ref{proposition:6:part:1}.

The following lemma shows how to calculate the product $(\pi,\rho)$ using a very simple rule.

\begin{lemma}\label{lemma:11}
Let $\pi$ and $\rho$ be paths in $T(\s,x)$ starting from leaves and having equal lengths.
Let $\lambda$ and $\mu$ be edges of $T(\s,x)$ such that $\pi\lm$ and $\rho\mu$ are again paths.
We~have
$$
(\pi\lm,\rho\mu)=\left\{
\arraycolsep=0pt
\begin{array}{ll}
(\pi,\rho)&\text{\; if }\lm\text{ is vertical or right tilted}\,;\\[8pt]
\displaystyle\frac{\overline{\ev(\pi)(\r(\lm))}+\overline{\ev(\rho)(\r(\mu))}}2\,(\pi,\rho)&\text{\; if }\lm\text{ is left tilted}.
\end{array}
\right.
$$
\end{lemma}
\begin{proof} Clearly, $\lm$ and $\mu$ are edges of the same level, say $i$.
We have $\r(\lm)=\r(\mu)=\alpha_i$ and also $[\rho\mu]'=[\rho]$.
Suppose first that $\lm$ is vertical or right tilted.
By Definitions~\ref{definition:10} and~\ref{definition:9}, we get
$$
(\pi\lm,\rho\mu)=\bigl(\b_{\pi\lm}\bigr)_{[\rho\mu]}=\Delta\bigl(\b_\pi\bigr)_{[\rho\mu]}
=\bigl(\b_\pi\bigr)_{[\rho]}=(\pi,\rho).
$$
Now suppose that $\lm$ is left tilted. By Definitions~\ref{definition:10} and~\ref{definition:9} and
formula~(\ref{eq:21}), we get
\begin{multline*}
(\pi\lm,\rho\mu)=(\b_{\pi\lm})_{[\rho\mu]}=\Delta\!\biggl(\frac{{c^{\alpha_i}+\overline{\ev(\pi)(\alpha_i)}}}2\,\b_\pi\!\biggr)_{\![\rho\mu]}
=\biggl(\frac{{c^{\alpha_i}+\overline{\ev(\pi)(\alpha_i)}}}2\,\b_\pi\!\biggr)_{\![\rho]}\\[8pt]
=\frac{\overline{\ev(\rho)(\alpha_i)}+\overline{\ev(\pi)(\alpha_i)}}2\,\bigl(\b_\pi\bigr)_{\![\rho]}
=\frac{\overline{\ev(\pi)(\r(\lm))}+\overline{\ev(\rho)(\r(\mu))}}2\,(\pi,\rho).
\end{multline*}
\end{proof}

Using our above notation, we denote by $E(\s,x)$ the $n\times n$ matrix whose $ij^{\text{th}}$ entry is $(\pi^{(i)},\pi^{(j)})$.
Under the hypothesis and in the notation of Corollary~\ref{corollary:5}, we have
\begin{equation}\label{eq:21.5}
\left(\arraycolsep=0pt\begin{array}{c} v_1\\\vdots\\v_n\end{array}\right)=E(\s,x)\!\left(\arraycolsep=0pt\begin{array}{c} e_1\\\vdots\\ e_n\end{array}\right),
\end{equation}
where $e_i:\{\pi^{(1)},\ldots,\pi^{(n)}\}\to Q$ is the function such that $e_i(\pi^{(j)})=\delta_{ij}$.

\begin{lemma}\label{lemma:5.5}
The matrix $E(\s,x)$ is upper triangular. Its diagonal entries are products of images $\bar\alpha$ of roots
$\alpha\in\widehat R$ and its first row is an array of 1's.
\end{lemma}
\begin{proof} Take integers $i,j$ such that $1\le j<i\le n$.
The maximal paths $\pi^{(i)}=\pi^{(i)}_1\cdots\pi^{(i)}_l$ and $\pi^{(j)}=\pi^{(j)}_1\cdots\pi^{(j)}_l$
merge at some vertex $a$. Denote by $k$ the level of $a$.
As $j\ne i$, we have $k>0$. By Lemma~\ref{lemma:11}, the product $(\pi^{(i)},\pi^{(j)})$
is proportional to the product $(\tilde\pi^{(i)}\pi^{(i)}_k,\tilde\pi^{(j)}\pi^{(j)}_k)$, where
$\tilde\pi^{(i)}=\pi^{(i)}_1\cdots\pi^{(i)}_{k-1}$ and $\tilde\pi^{(j)}=\pi^{(j)}_1\cdots\pi^{(j)}_{k-1}$.
Set $z:=\ev(\tilde\pi^{(i)})$ and $y:=\ev(\tilde\pi^{(j)})$.
We have $\alpha_k=\r(\pi^{(i)}_k)=\r(\pi^{(j)}_k)$. By construction of $T(\s,x)$, we have $y=zs_{\alpha_k}$.

Since $j<i$, the edge $\pi^{(i)}_k$ is left tilted.
Hence by Lemma~\ref{lemma:11}, we get
$$
(\tilde\pi^{(i)}\pi^{(i)}_k,\tilde\pi^{(j)}\pi^{(j)}_k)=\frac{\overline{z(\alpha_k)}+\overline{y(\alpha_k)}}2=0.
$$
The second and third statements follow immediately from Lemma~\ref{lemma:11}.
\end{proof}

Now that we have matrix $E(\s,x)$ together with the rule how to calculate it
(Corollary~\ref{corollary:5} and Lemma~\ref{lemma:11}), we can recover $\rk'\B(\s)^x$ from the tree $T(\s,x)$.
The {\it degree} of a path $\pi$ in $T(\s,x)$ is the number of left tilted edges of this path multiplied by $2$.
We denote this number by $\deg\pi$. We have
$$
\rk\nolimits'\B(\s)^x=\sum\{v^{-\deg\pi}\suchthat\pi\text{ is a maximal path in }T(\s,x)\}.
$$
In Example~\ref{example:2}, we the degrees of paths read from left to right are $0,2,2,2,4$ respectively.
So $\rk'\B((s_1,s_2,s_1,s_2,s_1))^{s_2s_1}=1+3v^{-2}+v^{-4}$.

\subsection{Automorphism $P(\s)$}\label{automorphism_P} We start with recollection of some constructions from Fiebig's paper~\cite{Fiebig_An_upper_bound}.
In~\cite[Definition~6.12]{Fiebig_An_upper_bound}), Fiebig defined the endomorphism $P(\s)$ of
$\bigoplus_{\sigma\in I(\s)}Q$ by the following rules:
\begin{enumerate}
\item $P(\emptyset)$ is the identity map.
\item If $\s\ne\emptyset$ and $P(\s')$ is already defined, then we set $P(\s):=c^{\alpha_l}\circ \Delta^{P(\s')}$.
\end{enumerate}
As is noted in that paper, the endomorphism $P(\s)$ is diagonal, whence a $\mathcal Z$-endomorphism.
This means that there exists a function $\P(\s)\in\bigoplus_{\sigma\in I(\s)}Q$ such that
$$
\bigl(P(\s)(f)\bigr)_\sigma=\P(\s)_\sigma f_\sigma
$$
for any $\sigma\in I(\s)$ and $f\in\bigoplus_{\sigma\in I(\s)}Q$.

\begin{lemma}\label{lemma:12}
$\P(\emptyset)_\emptyset=1$.
If $\s=(s_1,\ldots,s_l)\ne\emptyset$ then $\P(\s)_\sigma=\overline{\ev(\sigma)(\alpha_l)}\,\P(\s')_{\sigma'}$ for any $\sigma\in I(\s)$.
\end{lemma}
\begin{proof}
The first claim is clear as $\bigl(P(\emptyset)(f)\bigr)_\emptyset=f_\emptyset$.

Now take any $f\in\bigoplus_{\sigma\in I(\s')}Q$ and $\sigma\in I(\s')$.
By~(\ref{eq:1.25}),~(\ref{eq:21}) and~(\ref{eq:18}), we get
\begin{multline*}
\Bigl(P(\s)\bigl(\Delta(f)\bigr)\Bigr)_{\!\sigma\sp}
=\Bigl(c^{\alpha_l}\Delta^{P(\s')}\bigl(\Delta(f)\bigr)\Bigr)_{\!\sigma\sp}
=\overline{\ev(\sigma\sp)(\alpha_l)}\,\Delta\bigl(P(\s')(f)\bigr)_{\sigma\sp}\\[3pt]
=\overline{\ev(\sigma\sp)(\alpha_l)}\,\bigl(P(\s')(f)\bigr)_\sigma
=\overline{\ev(\sigma\sp)(\alpha_l)}\,\P(\s')_\sigma f_\sigma
=\overline{\ev(\sigma\sp)(\alpha_l)}\,\P(\s')_{(\sigma\sp)'}\,\Delta(f)_{\sigma\sp},
\end{multline*}
\begin{multline*}
\Bigl(P(\s)\bigl(\Delta(f)\bigr)\Bigr)_{\!\sigma s_l}
=\Bigl(c^{\alpha_l}\Delta^{P(\s')}\bigl(\Delta(f)\bigr)\Bigr)_{\!\sigma s_l}
=\overline{\ev(\sigma s_l)(\alpha_l)}\,\Delta\bigl(P(\s')(f)\bigr)_{\sigma s_l}\\[3pt]
=\overline{\ev(\sigma s_l)(\alpha_l)}\,\bigl(P(\s')(f)\bigr)_\sigma
=\overline{\ev(\sigma s_l)(\alpha_l)}\,\P(\s')_\sigma f_\sigma
=\overline{\ev(\sigma s_l)(\alpha_l)}\,\P(\s')_{(\sigma s_l)'}\,\Delta(f)_{\sigma s_l},
\end{multline*}
\begin{multline*}
\Bigl(P(\s)\bigl(\Delta^-(f)\bigr)\Bigr)_{\!\sigma\sp}=\Bigl(c^{\alpha_l}\Delta^{P(\s')}\bigl(\Delta^-(f)\bigr)\Bigr)_{\!\sigma\sp}
=\overline{\ev(\sigma\sp)(\alpha_l)}\,\Delta^-\bigl(P(\s')(f)\bigr)_{\sigma\sp}\\
=\overline{\ev(\sigma\sp)(\alpha_l)}\,\bigl(P(\s')(f)\bigr)_\sigma
=\overline{\ev(\sigma\sp)(\alpha_l)}\,\P(\s')_\sigma f_\sigma
=\overline{\ev(\sigma\sp)(\alpha_l)}\,\P(\s')_{(\sigma\sp)'}\,\Delta^-(f)_{\sigma\sp},
\end{multline*}
\begin{multline*}
\Bigl(P(\s)\bigl(\Delta^-(f)\bigr)\Bigr)_{\!\sigma s_l}
=\Bigl(c^{\alpha_l}\Delta^{P(\s')}\bigl(\Delta^-(f)\bigr)\Bigr)_{\!\sigma s_l}
=\overline{\ev(\sigma s_l)(\alpha_l)}\,\Delta^-\bigl(P(\s')(f)\bigr)_{\sigma s_l}\\
=-\,\overline{\ev(\sigma s_l)(\alpha_l)}\,\bigl(P(\s')(f)\bigr)_\sigma
=-\,\overline{\ev(\sigma s_l)(\alpha_l)}\,\P(\s')_\sigma f_\sigma
=\overline{\ev(\sigma s_l)(\alpha_l)}\,\P(\s')_{(\sigma s_l)'}\,\Delta^-(f)_{\sigma s_l}.
\end{multline*}
\end{proof}

This lemma shows that each $\P(\s)_\sigma$ is a product of images $\bar\alpha$ of $|\s|$ roots $\alpha\in\widehat R$,
whence an invertible element of $Q$. Therefore $P(\s)$ is a (homogeneous) automorphism of $\bigoplus_{\sigma\in I(\s)}Q$
with the inverse given by $\bigl(P(\s)^{-1}(f)\bigr)_\sigma=(\P(\s)_\sigma)^{-1}f_\sigma$.

Now we define $\P_\pi\in S$ for all paths $\pi$ in $T(\s,x)$ starting at a leaf as follows (cf.~Definition~\ref{definition:9}).

\begin{definition}\label{definition:11}
We set $\P_\emptyset:=1$. Let $\pi=\pi_1\cdots\pi_i$ be a path in $T(\s,x)$
of length $i>0$ starting at a leaf. 
We define
$$
\P_\pi:=\overline{\ev(\pi)(\r(\pi_i))}\,\P_{\pi'}.
$$
\end{definition}

\begin{lemma}\label{lemma:7}
$\P(\s)_{[\pi]}=\P_\pi$ for all maximal paths $\pi$ in $T(\s,x)$.
\end{lemma}
\begin{proof}
Induction on the length of $\s$. First suppose that $\s=\emptyset$. It suffices to consider
the case $x=e$, since otherwise $T(\s,x)$ is empty and there are no paths to consider.
By the previous definition and Lemma~\ref{lemma:12}, we have $\P(\emptyset)_\emptyset=\P_\emptyset=1$.

Now suppose that the length of $\s$ is nonzero.
Take a maximal path $\pi=\pi_1\cdots \pi_l$ in $T(\s,x)$ and consider its beginning $\pi'$.
We denote $x':=\ev(\pi')$. By Proposition~\ref{proposition:6}\ref{proposition:6:part:3},
we obtain that $\pi'$ is a maximal path in $T(\s',x')$. Therefore by the inductive hypothesis
$$
\P(\s')_{[\pi']}=\P_{\pi'}.
$$
From this equality, Lemma~\ref{lemma:12} and Definition~\ref{definition:11}, we get
$$
\P(\s)_{[\pi]}=\overline{\ev([\pi])(\alpha_l)\vphantom{A^A}}\,\P(\s')_{[\pi]'}
=\overline{\ev(\pi)(\alpha_l)}\,\P(\s')_{[\pi']}
=\overline{\ev(\pi)(\alpha_l)}\,\P_{\pi'}=\P_\pi.
$$
\end{proof}

The elements $\P_\pi$ we introduced are however very big. Our following aim is to divide all
$\P_\pi$ for paths $\pi$ ending at a fixed point by their common divisor as predicted by Proposition~\ref{proposition:7}.
To this end, we define elements $\D_\pi\in Q$ for any path $\pi$ in $T(\s,x)$ starting at a leaf as follows.

\begin{definition}\label{definition:12}
We set $\D_\emptyset:=1$. Let $\pi=\pi_1\cdots\pi_i$ be a path in $T(\s,x)$ of length $i>0$ starting at a leaf.
We define
$$
\D_\pi:=\left\{\arraycolsep=2pt
        \begin{array}{ll}
             \overline{\ev(\pi)(\r(\pi_i))}\,\D_{\pi'}&\text{ if }\c(\pi_i)=1;\\[6pt]
             \D_{\pi'}&\text{ if }\c(\pi_i)=0;\\[2pt]
             -\,\overline{\ev(\pi)(\r(\pi_i))}{\;\vphantom{A^{A^{A^A}}}}^{-1}\D_{\pi'}&\text{ if }\c(\pi_i)=-1.
        \end{array}\right.
$$
\end{definition}

For any element $x\in\widehat\W$, we define
$\D(x)$ to be the product of all labels of edges of $\widehat\G$ ending at $x$ each multiplied by $-1$ .
We make the following simple observation.

\begin{lemma}\label{lemma:13} Let $x\in\widehat\W$ and $\alpha\in\widehat\Pi$ such that $xs_\alpha<x$. Then
$\D(x)=\bar\gamma\,\D(xs_\alpha)$, where $\gamma$ is the unique negative root in the set $\{x(\alpha),-x(\alpha)\}$.
\end{lemma}
\begin{proof} Denote by $Y(x)$ the set of all $\tau\in \widehat R^-$ such that $s_\tau x<x$.
So $\D(x)=\overline{\prod Y(x)\vphantom{\prod^a}}$.
By the exchange property, $|Y(x)|=\ell(x)$. It is enough to prove that
$$
Y(x)=Y(xs_\alpha)\dotcup\{\gamma\}.
$$
First, by~(\ref{eq:0}), we have
$$
s_\gamma x=s_{x(\alpha)}x=xs_\alpha x^{-1}x=xs_\alpha<x.
$$
Hence and from $\gamma<0$, we get $\gamma\in Y(x)$. This formula also implies $s_\gamma xs_\alpha=x>xs_\alpha$,
whence $\gamma\notin Y(xs_\alpha)$.

Now take $\tau\in Y(xs_\alpha)$. By definition, we have $s_\tau xs_\alpha<xs_\alpha<x$.
Corollary~\ref{corollary:0}\ref{corollary:0:part:1} implies $s_\tau x\le x$.
The equality is impossible as $s_\tau\ne1$. Hence $s_\tau x<x$ and $\tau\in Y(x)$.

We have proved that $Y(x)\supset Y(xs_\alpha)\dotcup\{\gamma\}$. The inverse inclusion follows from the fact
that the cardinalities of both sets equal $\ell(x)=\ell(xs_\alpha)+1$.
\end{proof}

\begin{lemma}\label{lemma:14}
$\D_\pi=\D(\ev(\pi))$.
\end{lemma}
\begin{proof}
Induction on the length of $\pi$. We obviously have $\D_\emptyset=\D(\ev(\emptyset))=\D(e)=1$,
since no edge ends at $e$.

Now suppose that $\pi=\pi_1\cdots\pi_i$ is a path of length $i>0$. We set for brevity $x:=\ev(\pi)$.
We have $\alpha_i:=\r(\pi_i)$ and $s_i=s_{\alpha_i}$.

{\it Case 1: $\c(\pi_i)=1$.} In this case $x(\alpha_i)<0$. Indeed, if we had $x(\alpha_i)>0$
then $\ell(xs_i)>\ell(x)$ by Proposition~\ref{proposition:0}, which is a contradiction.
Hence by Definition~\ref{definition:12}, Lemma~\ref{lemma:13} and the inductive hypothesis,
we get
$$
\D_\pi=\overline{x(\alpha_i)}\,\D_{\pi'}=\overline{x(\alpha_i)}\,\D(xs_i)=\D(x).
$$

{\it Case 2: $\c(\pi_i)=0$.}  We have by Definition~\ref{definition:12} and the inductive hypothesis
$\D_\pi=\D_{\pi'}=\D(\ev(\pi'))=\D(x)$.

{\it Case 3: $\c(\pi_i)=-1$.} In this case $x(\alpha_i)>0$. Indeed, if we had $x(\alpha_i)<0$
then \linebreak$\ell(xs_i)<\ell(x)$ by Proposition~\ref{proposition:0}, which is a contradiction.
Lemma~\ref{lemma:13} implies $\D(xs_i)=-\overline{x(\alpha_i)}\D(x)$.
Hence by Definition~\ref{definition:12},  and the inductive hypothesis,
we get
$$
\D_\pi=-\,\overline{x(\alpha_i)}{\;\vphantom{A^{A^{A^A}}}}^{-1}\D_{\pi'}=-\,\overline{x(\alpha_i)}{\;\vphantom{A^{A^{A^A}}}}^{-1}\D(xs_i)=\D(x).
$$
\end{proof}
\noindent
This lemma shows that $\D_\pi\in S$ (although this is not obvious from the definition).

We will be especially interested in the quotient $\Q_\pi=\P_\pi/\D_\pi$.

\begin{lemma}\label{lemma:10.5}
$\Q_\emptyset=1$. Let $\pi=\pi_1\cdots\pi_i$ be a path in $T(\s,x)$ of length $i>0$ starting at a leaf.
Then
$$
\Q_\pi:=\left\{\arraycolsep=2pt
        \begin{array}{ll}
             \Q_{\pi'}&\text{ if }\c(\pi_i)=1;\\[6pt]
             \overline{\ev(\pi)(\r(\pi_i))}\,\Q_{\pi'}&\text{ if }\c(\pi_i)=0;\\[2pt]
             -\,\overline{\ev(\pi)(\r(\pi_i))}{\;\vphantom{A^{A^{A^A}}}}^2\Q_{\pi'}&\text{ if }\c(\pi_i)=-1.
        \end{array}\right.
$$
In particular, $\Q_\pi\in S$.
\end{lemma}
\begin{proof}
The result follows from Definitions~\ref{definition:11} and~\ref{definition:12}.
\end{proof}

{\bf Remark}. We can consider $\P_\pi$ for all maximal paths $\pi$ in a fixed tree $T(\s,x)$.
Lemmas~\ref{lemma:14} and~\ref{lemma:10.5} show that we can divide all these elements simultaneously by $\D(x)$ in~$S$
and calculate the quotient.

\subsection{Inclusion $\B(\s)_{[x]}\subset\B(\s)^x$} We briefly recall Fiebig's construction of the module $\Y(\s)$
dual to $\X(\s)$.

\begin{definition}[\text{\cite[Definition 6.6]{Fiebig_An_upper_bound}}]\label{definition:13}
We define for all sequences $\s$ in $\widehat\S$ the $S$-submodule $\Y(\s)\subset\bigoplus_{\sigma\in I(\s)}Q$
by the following inductive rule:
\begin{enumerate}
\itemsep=2pt
\item\label{definition:13:part:1} $\Y(\emptyset):=\bigoplus_{\sigma\in I(\emptyset)}S\cong S$;
\item\label{definition:13:part:2} if $\s=(s_1,\ldots,s_l)$ is not empty, then
$$
\Y(\s):=\Delta(\Y(\s'))+(c^{\alpha_l})^{-1}\Delta(\Y(\s')),
$$
where $\alpha_l\in\widehat\Pi$ is such that $s_l=s_{\alpha_l}$.
\end{enumerate}
\end{definition}

For any $S$-submodule $M\subset\bigoplus_{\sigma\in I(\s)}Q$, we define, following Fiebig, its dual by
$$
{\mathsf D}M:=\Biggl\{z\in\bigoplus_{\sigma\in I(\s)}Q\;\biggl|\;\sum\nolimits_{\sigma\in I(\s)}z_\sigma m_\sigma\in S\text{ for any }m\in M\Biggr\}.
$$
Similarly, for any $S$-submodule $N\subset\bigoplus_{\sigma\in I(\s)_x}Q$, we define its dual by
$$
{\mathsf D}N:=\Biggl\{z\in\bigoplus_{\sigma\in I(\s)_x}Q\;\biggl|\;\sum\nolimits_{\sigma\in I(\s)_x}z_\sigma n_\sigma\in S\text{ for any }n\in N\Biggr\}.
$$


\begin{proposition}[\text{\cite[Lemmas~6.8, 6.9 and 6.13]{Fiebig_An_upper_bound}}]\label{proposition:9}
${\mathsf D}\X(\s)=\Y(\s)$, ${\mathsf D}(\X(\s)^x)=\Y(\s)_x$ and $\X(\s)=P(\s)(\Y(\s))$.
\end{proposition}
\noindent
Here $P(\s)$ is the automorphism defined in Section~\ref{automorphism_P}. As we noted this automorphism
is diagonal. So it restricts to the automorphism $P(\s)_x$ of $\bigoplus_{\sigma\in I(\s)_x}Q$ defined by
$\bigl(P(\s)_x(f)\bigr)_\sigma:=\P(\s)_\sigma f_\sigma$ for any $f\in\bigoplus_{\sigma\in I(\s)_x}Q$
and $\sigma\in I(\s)_x$.
From Proposition~\ref{proposition:9}, we obviously get
\begin{equation}\label{eq:22}
P(\s)_x(\Y(\s)_x)=\X(\s)_x.
\end{equation}
Hence and from Proposition~\ref{proposition:9}, we can calculate the costalk $\X(\s)_x$ once we know the stalk $\X(\s)^x$.
This remarkable argument allowed Fiebig in~\cite{Fiebig_An_upper_bound} to obtain an upper bound for
the primes for which Lusztig's character formula does not hold.

Recall the basis $v_1,\ldots,v_n$ from Corollary~\ref{corollary:5}. Then there exist elements
$v'_1,\ldots,v'_n$ of ${\mathsf D}(\X(\s)^x)=\Y(\s)_x$ such that
$$
\sum_{\sigma\in I(\s)_x}(v_i)_\sigma(v'_j)_\sigma=\delta_{i,j}
$$
for any $i,j=1,\ldots,n$. A simple calculation show that $v'_1,\ldots,v'_n$ is a basis of $\Y(\s)_x$ and
$$
\left(\arraycolsep=0pt\begin{array}{c} v'_1\\\vdots\\v'_n\end{array}\right)=\Bigl(E(\s,x)^{-1}\Bigr)^{T}\!\left(\arraycolsep=0pt\begin{array}{c} e_1\\\vdots\\ e_n\end{array}\right).
$$
Consider the diagonal matrix $P(\s,x)$ whose $ii^{\text{th}}$-entry is $\P_{\pi^{(i)}}$ and elements
$v''_i:=P(\s)_x(v'_i)$. By~(\ref{eq:22}), we get that $v''_1,\ldots,v''_n$ is a basis of $\X(\s)_x$ and, using~(\ref{eq:21.5}), we get
$$
\left(\arraycolsep=0pt\begin{array}{c} v''_1\\\vdots\\v''_n\end{array}\right)=\Bigl(E(\s,x)^{-1}\Bigr)^TP(\s,x)\!\left(\arraycolsep=0pt\begin{array}{c} e_1\\\vdots\\ e_n\end{array}\right)
=\Bigl(E(\s,x)^{-1}\Bigr)^TP(\s,x)\;E(\s,x)^{-1}\left(\arraycolsep=0pt\begin{array}{c} v_1\\\vdots\\ v_n\end{array}\right).
$$

By Corollary~\ref{corollary:2}, we get that the matrix in the right-hand side is the transition matrix
from a basis of $\B(\s)^x$ to a basis of $\B(\s)_x$ if $\widehat\G_{J(\s)}$ satisfies the GKM-property.
Proposition~\ref{proposition:7} shows that to obtain the transition matrix from the same
basis of $\B(\s)^x$ to a basis of $\B(\s)_{[x]}$, we have to divide the above matrix by $\D(x)$.
Fortunately, we can divide $P(\s,x)$ by $\D(x)$
and calculate the quotient with the help of the function $\Q$ inductively defined in Lemma~\ref{lemma:10.5}.

Denote by $Q(\s,x)$ the diagonal matrix whose $ii^{\rm th}$-entry is $\Q_{\pi^{(i)}}$.

\begin{theorem}\label{theorem:4}
Let $\s$ be a sequence in $\widehat\S$ such that $\widehat\G_{J(\s)}$ satisfies the GKM-property.
Let $\pi^{(1)},\ldots,\pi^{(n)}$ be maximal paths in $T(\s,x)$ counted from left to right.
Then there exist a homogeneous basis of $\B(\s)^x$ with elements of degrees
$$
\deg\pi^{(1)},\ldots,\deg\pi^{(n)}
$$
and a homogeneous basis of $\B(\s)_{[x]}$
with elements of degrees
$$
2(|s|-\ell(x))-\deg\pi^{(1)},\ldots,2(|s|-\ell(x))-\deg\pi^{(n)}
$$
such that the transition matrix from the first basis to the second one is
$$
\Phi(\s,x):=(E(\s,x)^{-1})^TQ(\s,x)E(\s,x)^{-1}.
$$
\end{theorem}

\section{Low rank cases}\label{low_rank_cases}

\subsection{Exchange and comparison of roots} For the calculations in this section, we shall use the following simple arguments.

\begin{lemma}[Exchange of roots]\label{lemma:15} Let $\s=(s_1,\ldots,s_l)$ be a sequence in $\widehat\S$.
Let $\pi=\pi_1\cdots\pi_m$ be a path in $T(\s,x)$ starting at a leaf such that $\pi_m$ is left tilted.
Then there exists a path $\rho=\rho_1\cdots\rho_m$ in $T(\s,x)$ such that
{\leftmargini=22pt
\begin{enumerate}
\itemsep=4pt
\item\label{lemma:15:property:1} $\rho_m$ is right tilted, starts at a leaf and ends at the same vertex as $\pi$.
\item\label{lemma:15:property:2} $[\rho']$ is obtained from $[\pi']$ by replacing the simple reflection at some position $i$ with~$\sp$.
\end{enumerate}
If $\rho$ is a path satisfying these conditions, then $\rho$ is called a descendant of $\pi$ and
the following equality is satisfied:
\begin{enumerate}
\setcounter{enumi}{2}
\item\label{lemma:15:property:3} $\ev(\pi_1\cdots\pi_{m-1})\bigl(\r(\pi_m)\bigr)=-\c(\pi_i)\ev(\rho_1\cdots\rho_{i-1})\bigl(\r(\rho_i)\bigr)$.
\end{enumerate}}
\end{lemma}
\begin{proof}
As usual, we assume that $s_i=s_{\alpha_i}$, where $\alpha_i\in\widehat\Pi$.
Denote by $\rho_m$ the unique right tilted edge ending at the same vertex as $\pi_m$.
Let $a$ and $b$ be the beginnings of $\rho_m$ and $\pi_m$ respectively.
We set $y:=\ev(a)=$ and $z:=\ev(b)=\ev(\pi')$. By our construction $y<z$ and $zs_m=y$.
So we get
\begin{equation}\label{eq:20.5}
\ev(\pi')s_m<\ev(\pi').
\end{equation}
By the exchange property,
$\ev(\pi')s_m=\ev(u)$, where $u$ is the sequence obtained from $[\pi']$
by replacing its $i^{\text{th}}$ entry with $\sp$ for some $i$. Hence $u\in I((s_1,\ldots,s_{m-1}))_y$.
By Proposition~\ref{proposition:6}\ref{proposition:6:part:1}, there exists a maximal path $\rho'$ in
$T((s_1,\ldots,s_{m-1}),y)$ such that $[\rho']=u$.

By Proposition~\ref{proposition:6}\ref{proposition:6:part:3} the full subtree of $T(\s,x)$ with root $a$
is $T((s_1,\ldots,s_{m-1}),y)$. So we can consider the path $\rho:=\rho'\rho_m$ in $T(\s,x)$.
Properties~\ref{lemma:15:property:1} and~\ref{lemma:15:property:2} are automatically satisfied
for this choice of $\rho$.

Let us write $u=v\sp w$, where $v$ is the beginning of $u$ of length $i-1$. Then $[\pi']=vs_iw$.
We get
$$
\ev(v)s_i\ev(w)s_m=\ev(\pi')s_m=zs_m=y=\ev(u)=\ev(v)\ev(w).
$$
Hence $\ev(w)s_m\ev(w)^{-1}=s_i$. By~(\ref{eq:0}), we get
$s_{\ev(w)(\alpha_m)}=s_{\alpha_i}$. As we noted in Section~\ref{Affine_root_system}, we get from this equality
that $\ev(w)(\alpha_m)=\epsilon\alpha_i$ for some $\epsilon\in\{1,-1\}$.
Multiplying this equality first by $s_i$ and then by $\ev(v)$ on the left,
we get
\begin{equation}\label{eq:21.75}
\ev(\pi')(\alpha_m)=\ev(v)s_i\ev(w)(\alpha_m)=-\epsilon\ev(v)(\alpha_i).
\end{equation}
By~(\ref{eq:20.5}) and Proposition~\ref{proposition:0}, $\ev(\pi_1\cdots\pi_{m-1})(\alpha_m)<0$.
On the other hand, $\ev(v)=\ev(\pi_1\cdots\pi_{i-1})$. Hence $\ev(v)(\alpha_i)>0$ if $\c(\pi_i)=1$
and $\ev(v)(\alpha_i)<0$ if $\c(\pi_i)=-1$, the case $\c(\pi_i)=0$ being impossible.
In the first case $\epsilon=1$ and in the second case $\epsilon=-1$.
It remains to apply~(\ref{eq:21.75}) and recall that $\ev(v)=\ev(\rho_1\cdots\rho_{i-1})$.
\end{proof}

We shall also use the following method to compare roots. Let $x,y\in\widehat\W$ and $\alpha,\beta\in\widehat R$.
Then by~(\ref{eq:1}), we get
$$
xs_\alpha y(\beta)-xy(\beta)=x\Bigl(s_\alpha\bigl(y(\beta)\bigr)-y(\beta)\Bigr)=-\<y(\beta),\alpha\>'x(\alpha).
$$
We apply~(\ref{eq:0.25}) and get
\begin{equation}\label{eq:23}
xs_\alpha y(\beta)-xy(\beta)=-\<xy(\beta),x(\alpha)\>'x(\alpha)=\<xs_\alpha y(\beta),x(\alpha)\>'x(\alpha).
\end{equation}


\subsection{$\mathbf{2\times2}$-matrices}\label{2times2} Let $\s=(s_1,\ldots,s_l)$ be a sequence in $\widehat\S$ such that
the expression $w:=s_1\cdots s_l$ is reduced and $\widehat\G_{\le w}$ satisfies the GKM-property.
As usual, we assume that $s_i=s_{\alpha_i}$, where $\alpha_i\in\widehat\Pi$.

We want to calculate the defect of the projection $\rho_{x,\delta x}:\B(\s)^x\to\B(\s)^{\delta x}$
in the case when the ungraded rank of $\B(\s)^x$ is 2.
In this case, $|I(\s)_x|=2$ and the tree $T(\s,x)$ has two maximal paths
$\pi^{(1)}=\pi^{(1)}_1\cdots\pi^{(1)}_l$ and $\pi^{(2)}=\pi^{(2)}_1\cdots\pi^{(2)}_l$ labelled from left to right.
Denote by $k$ be the level of the vertex where $\pi^{(1)}$ and $\pi^{(2)}$ merge.
By our construction $\ell\bigl(\ev(\pi^{(2)}_1\cdots\pi^{(2)}_{k-1})s_k\bigr)<\ell\bigl(\ev(\pi^{(2)}_1\cdots\pi^{(2)}_{k-1})\bigr)$.
This inequality and the fact that the expression $w=s_1\cdots s_l$ is reduced imply that
there is some $i=1,\ldots,k-1$ with $\c(\pi^{(2)}_i)=0$.

By Corollary~\ref{corollary:1} and Theorem~\ref{theorem:4} the defect $\d(\rho_{x,\delta x})$ can be read off the zero degree entries
of the matrix $\Phi(\s,x)=(E(\s,x)^{-1})^TQ(\s,x)E(\s,x)^{-1}$.
The multiplication rule for paths given in Lemma~\ref{lemma:11} implies that
$$
E(\s,x)=\(
\arraycolsep=2pt
\begin{array}{cc}
1&1\\[3pt]
0&\;\overline{z(\alpha_k)}
\end{array}\),
$$
where $z=\ev(\pi^{(2)}_1\cdots\pi^{(2)}_{k-1})$. We shall consider only the case when $\Phi(\s,x)$
has entries of degree $0$, since otherwise $\d(\rho_{x,\delta x})=0$. Then $\deg\Q_{\pi^{(1)}}=\deg\Q_{\pi^{(2)}}\le 4$.
Since $\c(\pi^{(2)}_i)=0$, we get $\c(\pi^{(2)}_k)=0$ and the colors of all other edges of $\pi^{(2)}$ equal $1$.
Lemma~\ref{lemma:15} implies that $\pi^{(1)}_1\cdots\pi^{(1)}_k$ is a descendant of
$\pi^{(2)}_1\cdots\pi^{(2)}_k$. Hence $\c(\pi^{(1)}_i)=0$ and $\c(\pi^{(1)}_j)=0$ for some
$j\in\{1,\ldots,k-1\}\setminus\{i\}$.

We claim that $j<i$. Indeed suppose that on the contrary $i<j$.
Then $s_1\cdots \hat s_i\cdots s_{k-1}=z=\ev(\pi^{(1)}_1\cdots\pi^{(1)}_{k-1})s_k=s_1\cdots \hat s_i\cdots \hat s_j\cdots s_k$.
Hence $s_j\cdots s_{k-1}=s_{j+1}\cdots s_k$, which contradicts the fact that $w=s_1\cdots s_l$ is a reduced expression.
Graphically, our situation is as follows:

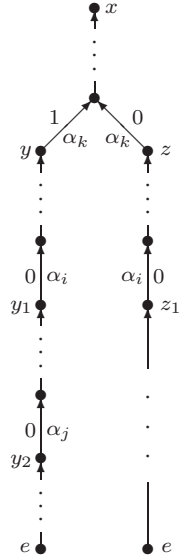
\def\emptyshort{
\put(0,0){\line(0,1){6}}
\put(0,11){\circle*{1}}\put(0,16){\circle*{1}}\put(0,21){\circle*{1}}
\put(0,26){\vector(0,1){7}}
}

\def\emptylong{
\put(0,-170){\line(0,1){25}}
\put(0,-135){\circle*{1}}
\put(0,-124){\circle*{1}}
\put(0,-113){\circle*{1}}
\put(0,-102){\vector(0,1){23}}
}

\def\emptylongg{
\put(0,-168){\line(0,1){23}}
\put(0,-135){\circle*{1}}
\put(0,-124){\circle*{1}}
\put(0,-113){\circle*{1}}
\put(0,-102){\vector(0,1){23}}
}

\def\casetwobytwo{
\put(0,34){\circle*{4}}
\put(4,32){$\scriptstyle x$}
\put(0,0){\circle*{4}}
\put(0,0)\emptyshort
\put(-20,-20){\vector(1,1){19}}
\put(-20,-20){\circle*{4}}
\put(-28,-21.5){$\scriptstyle y$}
\put(20,-20){\vector(-1,1){19}}
\put(20,-20){\circle*{4}}
\put(24.5,-21.5){$\scriptstyle z$}
\put(-17,-10){$\scriptstyle1$}
\put(-12,-17){$\scriptstyle\alpha_k$}
\put(14,-10){$\scriptstyle0$}
\put(4,-17){$\scriptstyle\alpha_k$}
%
\put(20,-54)\emptyshort
\put(20,-54){\circle*{4}}
\put(20,-78){\vector(0,1){23}}
\put(22,-70){$\scriptstyle0$}
\put(10,-69){$\scriptstyle\alpha_i$}
\put(20,-78){\circle*{4}}
\put(24.5,-80){$\scriptstyle z_1$}
%
\put(-20,-54)\emptyshort
\put(-20,-54){\circle*{4}}
\put(-20,-78){\vector(0,1){23}}
\put(-26,-70){$\scriptstyle0$}
\put(-18,-69){$\scriptstyle\alpha_i$}
\put(-20,-78){\circle*{4}}
\put(-31.5,-80){$\scriptstyle y_1$}
\put(-20,-112)\emptyshort
%
%
\put(-20,-136){\circle*{4}}
\put(-20,-136){\vector(0,1){23}}
\put(-26,-128){$\scriptstyle0$}
\put(-18,-127){$\scriptstyle\alpha_j$}
\put(-20,-112){\circle*{4}}
\put(-31.5,-138){$\scriptstyle y_2$}
%
%
\put(-20,-170)\emptyshort
\put(-20,-170){\circle*{4}}
\put(-28,-171.5){$\scriptstyle e$}
\put(20,0)\emptylong
\put(20,-170){\circle*{4}}
\put(25,-171.5){$\scriptstyle e$}
%
}

\def\casetwobytwoprime{
\put(0,0){\circle*{4}}
\put(0,0)\emptyshort
\put(-20,-20){\vector(1,1){19}}
\put(-20,-20){\circle*{4}}
\put(-28,-21.5){$\scriptstyle y'$}
\put(20,-20){\vector(-1,1){19}}
\put(20,-20){\circle*{4}}
\put(24.5,-21.5){$\scriptstyle z'$}
\put(-17,-10){$\scriptstyle1$}
\put(-12,-17){$\scriptstyle\alpha_k$}
\put(14,-10){$\scriptstyle0$}
\put(4,-17){$\scriptstyle\alpha_k$}
%
\put(20,-54)\emptyshort
\put(20,-54){\circle*{4}}
\put(20,-78){\vector(0,1){23}}
\put(22,-70){$\scriptstyle0$}
\put(10,-69){$\scriptstyle\alpha_i$}
\put(20,-78){\circle*{4}}
\put(24.5,-80){$\scriptstyle z'_1$}
%
\put(-20,-54)\emptyshort
\put(-20,-54){\circle*{4}}
\put(-20,-78){\vector(0,1){23}}
\put(-26,-70){$\scriptstyle0$}
\put(-18,-69){$\scriptstyle\alpha_i$}
\put(-20,-78){\circle*{4}}
\put(-31.5,-80){$\scriptstyle y'_1$}
\put(-20,-112)\emptyshort
%
%
\put(-20,-136){\circle*{4}}
\put(-20,-136){\vector(0,1){23}}
\put(-26,-128){$\scriptstyle0$}
\put(-18,-127){$\scriptstyle\alpha_j$}
\put(-20,-112){\circle*{4}}
\put(-31.5,-138){$\scriptstyle y'_2$}
%
%
\put(-20,-170)\emptyshort
\put(-20,-170){\circle*{4}}
\put(-28,-171.5){$\scriptstyle e$}
\put(20,0)\emptylong
\put(20,-170){\circle*{4}}
\put(24,-171.5){$\scriptstyle e$}
%
}

\vspace{32pt}
\begin{center}
\begin{picture}(0,0)
\put(0,0)\casetwobytwo
\end{picture}
\end{center}

\vspace{180pt}

\noindent
Here $y,y_1,y_2,z_1,e$ are the elements of $\widehat\W$ corresponding to the vertices closest to them:
$$
\arraycolsep=20pt
\begin{array}{lll}
y=s_1\ldots\hat s_j\cdots\hat s_i\cdots s_{k-1},& y_1=s_1\cdots \hat s_j\cdots s_{i-1},& y_2=s_1\cdots s_{j-1}, \\[12pt]
z_1=s_1\cdots s_{i-1},                          & z=s_1\cdots\hat s_i\cdots s_{k-1}.   &
\end{array}
$$
All unmarked edges in the picture above have color $1$.
By Lemma~\ref{lemma:10.5}, we get
$$
Q(\s,x)=\(
\arraycolsep=2pt
\begin{array}{cc}
\overline{y_1(\alpha_i)}\,\overline{y_2(\alpha_j)}\;&0\\[3pt]
0&\;\overline{z(\alpha_k)}\,\overline{z_1(\alpha_i)}
\end{array}\).
$$
Hence
$$
\Phi(\s,x)=\(
\arraycolsep=2pt
\begin{array}{cc}
\overline{y_1(\alpha_i)}\,\overline{y_2(\alpha_j)}\;&-\frac{\overline{y_1(\alpha_i)}\,\overline{y_2(\alpha_j)}}{\overline{z(\alpha_k)}^{\vphantom{2}}}\\[10pt]
-\frac{\overline{y_1(\alpha_i)}\,\overline{y_2(\alpha_j)}}{\overline{z(\alpha_k)}^{\vphantom{2}}}&\;\frac{\overline{y_1(\alpha_i)}\,\overline{y_2(\alpha_j)}+\overline{z(\alpha_k)}\,\overline{z_1(\alpha_i)}}{\overline{z(\alpha_k)}^2\vphantom{A^{A^{A^a}}}}
\end{array}\).
$$
By part~\ref{lemma:15:property:3} of Lemma~\ref{lemma:15}, we get $z(\alpha_k)=-y_2(\alpha_j)$. Therefore,
we can simplify the above matrix as follows
$$
\Phi(\s,x)=\(
\arraycolsep=2pt
\begin{array}{cc}
\overline{y_1(\alpha_i)}\,\overline{y_2(\alpha_j)}\;&\overline{y_1(\alpha_i)}\\[10pt]
\overline{y_1(\alpha_i)}&\;\frac{-\overline{y_1(\alpha_i)}+\overline{z_1(\alpha_i)}}{\overline{z(\alpha_k)}^{\vphantom{2}}}
\end{array}\).
$$
Now~(\ref{eq:23}) implies that
\begin{equation}\label{eq:25}
z_1(\alpha_i)-y_1(\alpha_i)=-\<y_1(\alpha_i),y_2(\alpha_j)\>'y_2(\alpha_j)=\<y_1(\alpha_i),y_2(\alpha_j)\>'z(\alpha_k).
\end{equation}
Hence the only entry of $\Phi(\s,x)$ having degree $0$ is
$$
\frac{-\overline{y_1(\alpha_i)}+\overline{z_1(\alpha_i)}}{\overline{z(\alpha_k)}^{\vphantom{2}}}=\overline{\<y_1(\alpha_i),y_2(\alpha_j)\>'}.
$$
\noindent
Consider the following elements of $\widehat\W_{<w}$:
$$
u=s_1\cdots\hat s_j\cdots s_l,\qquad\quad v=s_1\cdots\hat s_k\cdots s_l.
$$
We have $u\ne v$, since otherwise $s_{j+1}\cdots s_k=s_j\cdots s_{k-1}$, which contradicts
the fact that the expression $w=s_1\cdots s_l$ is reduced.
So we have two different edges $\edgeright x{\pm\overline{z_1(\alpha_i)}}v$
and $\edgeright x{\pm\overline{y_1(\alpha_i)}}u$ of $\widehat\G_{\le w}$.
By~(\ref{eq:25}), the labels of this edges are equal unless $\overline{\<y_1(\alpha_i),y_2(\alpha_j)\>'}\ne0$.
Hence and from Corollary~\ref{corollary:1}, we get the following result.

\begin{lemma}\label{lemma:2by2} Let $\s=(s_1,\ldots,s_l)$ be a sequence in $\widehat\S$ such that
the expression $w:=s_1\cdots s_l$ is reduced and $\widehat\G_{\le w}$ satisfies the GKM-property.
If the ungraded rank of $\B(\s)^x$ is $2$ {\rm(}i.e. $|I(\s)_x|=2${\rm)},
then the defect of the projection $\rho_{x,\delta x}:\B(\s)^x\to\B(\s)^{\delta x}$ is $v^{-2}$
if $\ell(x)=l-2$ and $0$ otherwise.
\end{lemma}

If the product $s_1\cdots s_l$ is not reduced, then there are two additional cases. In both of them
$\d(\rho_{x,\delta x})=0$.

\subsection{$\mathbf{3\times3}$-matrices} Consider the same situation as in Section~\ref{2times2} with the only difference that
the ungraded rank of $\B(\s)^x$ is $3$. A priori, the following cases are possible:

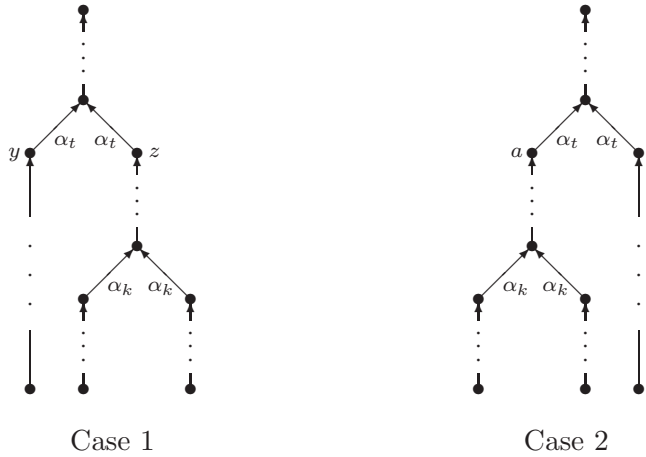
\vspace{33pt}
\begin{center}
\begin{picture}(0,0)
\put(0,34){\circle*{4}}
\put(0,0)\emptyshort
%
\put(0,0){\circle*{4}}
\put(-20,-20){\vector(1,1){19}}
\put(-20,-20){\circle*{4}}
\put(-28,-21.5){$\scriptstyle y$}
\put(20,-20){\vector(-1,1){19}}
\put(20,-20){\circle*{4}}
\put(24.5,-21.5){$\scriptstyle z$}
\put(-11,-17){$\scriptstyle\alpha_t$}
\put(4,-17){$\scriptstyle\alpha_t$}
%
\put(20,-54)\emptyshort
\put(20,-55){\circle*{4}}
\put(40,-75){\vector(-1,1){19}}
\put(24,-72){$\scriptstyle\alpha_k$}
\put(0, -75){\vector(1,1){19}}
\put(9, -72){$\scriptstyle\alpha_k$}
\put(40, -75){\circle*{4}}
\put(0, -75){\circle*{4}}
%
\put(40,-109)\emptyshort
\put(40,-109){\circle*{4}}
%
%
\put(0,-109)\emptyshort
\put(0,-109){\circle*{4}}
\put(-20,-109){\circle*{4}}
\put(-20,58)\emptylongg
\put(-5,-132){\small Case~1}
\end{picture}
\hspace{180pt}
\begin{picture}(0,0)
\put(0,34){\circle*{4}}
\put(0,0)\emptyshort
\put(0,0){\circle*{4}}
\put(-20,-20){\vector(1,1){19}}
\put(-20,-20){\circle*{4}}
\put(-28,-21.5){$\scriptstyle a$}
\put(20,-20){\vector(-1,1){19}}
\put(20,-20){\circle*{4}}
\put(-11,-17){$\scriptstyle\alpha_t$}
\put(4,-17){$\scriptstyle\alpha_t$}
\put(-20,-54)\emptyshort
\put(-20,-55){\circle*{4}}
  \put(0,-75){\vector(-1,1){19}}
\put(-40,-75){\vector(1,1){19}}
\put(-40,-75){\circle*{4}}
  \put(0,-75){\circle*{4}}
\put(-16,-72){$\scriptstyle\alpha_k$}
\put(-31,-72){$\scriptstyle\alpha_k$}
%
%
\put(-40,-109)\emptyshort
\put(-40,-109){\circle*{4}}
%
%
%
\put(0,-109)\emptyshort
\put(0,-109){\circle*{4}}
%
\put(-23,-132){\small Case~2}

\put(20,-109){\circle*{4}}
\put(20,58)\emptylongg
\end{picture}
\end{center}
\vspace{140pt}

We claim that the first case is impossible. Indeed, we have two elements $\sigma$ and $\tau$ of $I((s_1,\ldots,s_{t-1}))_z$
such that $\sigma=\tilde\sigma s_k\rho$ and $\tau=\tilde\tau\sp\rho$, where $|\tilde\sigma|=|\tilde\tau|=k-1$.
Since $y=z s_t<z$, we can apply the exchange property to both representations $z=\ev(\sigma)$ and $z=\ev(\tau)$.
We get that there are numbers $a,b=1,\ldots,t-1$ such that $y=\ev(\sigma_a)$ and $y=\ev(\tau_b)$,
where $\sigma_a$ is obtained from $\sigma$ by replacing its $a^{\text{th}}$ entry $s_a$ by $\sp$ and
$\tau_b$ is obtained from $\tau$ by replacing its $b^{\text{th}}$ entry $s_b$ by $\sp$.

First notice that if $a>k$ or $b>k$ then $\ev(\rho)s_t=\ev(\hat\rho)$, where $\hat\rho$ is obtained from $\rho$
by replacing one simple reflection with $\sp$. Then we get two different representations $y=\ev(\tilde\sigma s_k\hat\rho)$
and $y=\ev(\tilde\tau\sp\hat\rho)$ contrary to our picture. So $a\le k$ and $b\le k$.

The case where $a<k$ and $b<k$ is impossible, as we would get $\sigma_a\ne\tau_b$, whence $y=\ev(\sigma_a)$ and
$y=\ev(\tau_b)$ are different representations. On the other hand, $b<k$ in any case,
as the $k^{\text{th}}$ entry of $\tau$ is already $\sp$. Hence we get $a=k$.
This equality implies that $\ev(\rho)s_t=s_k\ev(\rho)$. We have $y=\ev(\tilde\sigma\sp\rho)$ and
$$
y=zs_t=\ev(\tilde\tau\sp\rho)s_t=\ev(\tilde\tau)\ev(\rho)s_t=\ev(\tilde\tau)s_k\ev(\rho)=\ev(\tilde\tau s_k\rho),
$$
which are again different representations of $y$. This contradicts the picture of case~1.

So only case~2 is possible. Consider the projection $\rho_{x,\delta x}:\B(\s)^x\to\B(\s)^{\delta x}$.
Clearly $d(\rho_{x,\delta x})=0$ if $\Phi(\s,x)$ does not contain entries of degree $0$.
Hence we consider only the case where $\deg\Q_\pi\le4$ for any maximal path $\pi$ in $T(\s,x)$.

Denote by $\pi^{(1)},\pi^{(2)},\pi^{(3)}$ the maximal paths in $T(\s,x)$ labelled
from left to right. Let us apply the arguments of Section~\ref{2times2} to the subtree with root $a$
(see the right picture above). Hence we get $\c(\pi^{(1)}_j)=0$, $\c(\pi^{(1)}_i)=0$, $\c(\pi^{(2)}_i)=0$,
$\c(\pi^{(2)}_k)=0$ for some $j<i<k$. Taking into account $\deg\Q_\pi\le4$, we get that the color of any other
edge of $\pi^{(1)}$ and $\pi^{(2)}$ has color $1$. In particular,
$\deg\Q_{\pi^{(1)}}=\deg\Q_{\pi^{(2)}}=\deg\Q_{\pi^{(3)}}=4$ and $\c(\pi^{(3)}_t)=0$.

Now the only possibility to get $\deg\Q_{\pi^{(3)}}=4$ is that $\c(\pi^{(3)}_q)=0$ for some $q<t$ and the
colors of all edges of $\pi^{(3)}$ except $\pi^{(3)}_q$ and $\pi^{(3)}_t$ is $1$.
As $\pi^{(1)}_1\cdots\pi^{(1)}_{t-1}$ or $\pi^{(2)}_1\cdots\pi^{(2)}_{t-1}$
is a descendant of $\pi^{(3)}_1\cdots\pi^{(3)}_{t-1}$, we have $q=k$ or $q=i$ or $q=j$.
The last two cases are impossible, as the expression $w=s_1\cdots s_l$ is reduced.
So $q=k$ and we get the only possible picture

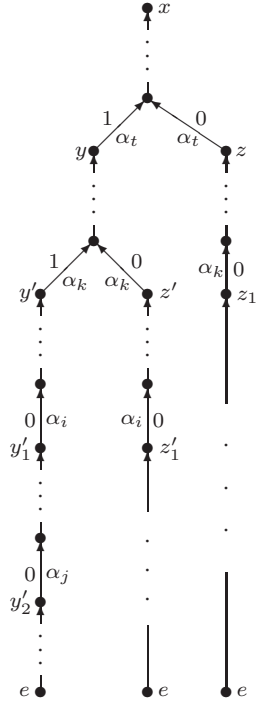
\vspace{33pt}
\begin{center}
\begin{picture}(0,0)
\put(0,0)\emptyshort
\put(0,0){\circle*{4}}
\put(4,32){$\scriptstyle x$}
\put(0,34){\circle*{4}}
\put(-27,-22){$\scriptstyle y$}
\put(-20,-20){\vector(1,1){19}}
\put(-20,-20){\circle*{4}}
\put(29.5,-20){\vector(-3,2){28.5}}
\put(29.5,-20){\circle*{4}}
\put(33,-22){$\scriptstyle z$}
\put(-17,-10){$\scriptstyle1$}
\put(-12,-17){$\scriptstyle\alpha_t$}
\put(18,-10){$\scriptstyle0$}
\put(11,-17){$\scriptstyle\alpha_t$}
\put(29.5,-54)\emptyshort
\put(29.5,-54){\circle*{4}}
\put(29.5,-74){\circle*{4}}
\put(29.5,-74){\vector(0,1){19}}
\put(32,-67){$\scriptstyle0$}
\put(19,-66){$\scriptstyle\alpha_k$}
\put(33.5,-76){$\scriptstyle z_1$}
\put(29.5,-224){\circle*{4}}
\put(33.5,-225.5){$\scriptstyle e$}
\put(29.5,-224){\line(0,1){45}}
\put(29.5,-163){\circle*{1}}\put(29.5,-147){\circle*{1}}\put(29.5,-131){\circle*{1}}
\put(29.5,-115){\vector(0,1){40}}
\put(-20,-54)\casetwobytwoprime
\end{picture}
\end{center}

\vspace{230pt}

\noindent
Here again all unmarked edges have color $1$.
The multiplication rule for paths given in Lemma~\ref{lemma:11} implies that
$$
E(\s,x)=\(
\arraycolsep=2pt
\begin{array}{ccc}
1&1&1\\[6pt]
0&\;\overline{z'(\alpha_k)}&\;\frac{\overline{z'(\alpha_k)}+\overline{z_1(\alpha_k)}}2\\[6pt]
0&0\;&\overline{z(\alpha_t)}
\end{array}\).
$$
By Lemma~\ref{lemma:10.5}, we get
$$
Q(\s,x)=\(
\arraycolsep=2pt
\begin{array}{ccc}
\overline{y'_1(\alpha_i)}\,\overline{y'_2(\alpha_j)}\;&0\;&0\\[6pt]
0&\;\overline{z'(\alpha_k)}\,\overline{z'_1(\alpha_i)}\;&0\\[6pt]
0&\;0\;&\overline{z(\alpha_t)}\,\overline{z_1(\alpha_k)}
\end{array}\).
$$
\noindent
Lemma~\ref{lemma:15} implies the following relations:
$$
z(\alpha_t)=-z'_1(\alpha_i),\qquad z'(\alpha_k)=-y'_2(\alpha_j)
$$
and formula~(\ref{eq:23}) the following:
$$
z_1(\alpha_k)-z'(\alpha_k)=-\<z'(\alpha_k),z'_1(\alpha_i)\>'z'_1(\alpha_i),\quad
z'_1(\alpha_i)-y'_1(\alpha_i)=\<z'_1(\alpha_i),y'_2(\alpha_j)\>'y'_2(\alpha_j).
$$
We can express all the roots occurring in $E(\s,x)$ and $Q(\s,x)$ via $y'_1(\alpha_i)$ and $y'_2(\alpha_j)$
as follows:
\begin{equation}\label{eq:26}
\begin{array}{l}
z'(\alpha_k)=-y'_2(\alpha_j),\quad z'_1(\alpha_i)=y'_1(\alpha_i)+by'_2(\alpha_j),\\[8pt]
\hspace{90pt}z(\alpha_t)=-y'_1(\alpha_i)-by'_2(\alpha_j),\quad z_1(\alpha_k)=ay'_1(\alpha_i)+(ab-1)y'_2(\alpha_j),
\end{array}
\end{equation}{
where $a=\<y'_2(\alpha_j),z'_1(\alpha_i)\>'$ and $b=\<z'_1(\alpha_i),y'_2(\alpha_j)\>'$.
Plugging this into the formulas for $E(\s,x)$ and $Q(\s,x)$, we get
$$
\Phi(\s,x)=\left(
\begin{array}{ccc}
\overline{y'_1(\alpha_i)}\,\overline{y'_2(\alpha_j)}&\overline{y'_1(\alpha_i)}&\frac{a\overline{y'_1(\alpha_i)}}2\\[14pt]
\overline{y'_2(\alpha_j)}&-\bar b&1-\frac{\bar a\bar b}2\\[14pt]
\frac{a\overline{y'_1(\alpha_i)}}2&1-\frac{\bar a\bar b}2&-\frac{\bar a^2\bar b}4
\end{array}
\right)
$$
The determinant of the zero degree part is $\bar a\bar b-1$. 
Consider the following elements of $\widehat\W_{<w}$:
$$
u=s_1\cdots\hat s_j\cdots s_l,\qquad\quad v=s_1\cdots\hat s_t\cdots s_l.
$$
As the expression $w=s_1\cdots s_l$ is reduced, we get $u\ne v$.
So we have two different edges $\edgeright x{\pm\overline{z_1(\alpha_k)}}v$
and $\edgeright x{\pm\overline{y'_1(\alpha_i)}}u$ of $\widehat\G_{\le w}$.
By the last formula of~(\ref{eq:26}), the labels of these edges are proportional unless $\bar a\bar b-1\ne0$.
Hence and from Corollary~\ref{corollary:1}, we get the following result.

\begin{lemma}\label{lemma:3by3} Let $\s=(s_1,\ldots,s_l)$ be a sequence in $\widehat\S$ such that
the expression $w:=s_1\cdots s_l$ is reduced and $\widehat\G_{\le w}$ satisfies the GKM-property.
If the ungraded rank of $\B(\s)^x$ is $3$ {\rm(}i.e. $|I(\s)_x|=3${\rm)},
then the defect of the projection $\rho_{x,\delta x}:\B(\s)^x\to\B(\s)^{\delta x}$ is $2v^{-2}$
if $\ell(x)=l-2$ and $0$ otherwise.
\end{lemma}

\subsection{Decomposition of $\B(\s)$} In view of~(\ref{eq:def}), Lemmas~\ref{lemma:2by2} and~\ref{lemma:3by3}
imply the following result.

\begin{corollary}\label{corollary:6}
Let $\s=(s_1,\ldots,s_l)$ be a sequence in $\widehat\S$ such that
the expression $w:=s_1\cdots s_l$ is reduced and $\widehat\G_{\le w}$ satisfies the GKM-property.
Then
$$
\B(\s)\cong\B(w)\,\oplus\!\sum_{l(x)=l(w)-2\atop|I(\s)_x|=2\vphantom{A^{A^a}}}\!\B(x)\<-2\>\,\oplus\!\sum_{l(x)=l(w)-2\atop|I(\s)_x|=3\vphantom{A^{A^a}}}\!2\,\B(x)\<-2\>\oplus\B(x_1)\<r_1\>\oplus\cdots\oplus\B(x_n)\<r_n\>,
$$
where $|I(\s)_{x_i}|>3$ for any $i=1,\ldots,n$.
\end{corollary}

Note that our constructions do not depend on the characteristic of $\F$ as long as it is not~$2$ and we have
the corresponding GKM-property. Therefore, we can convert our knowledge of the characters of Bott-Samelson
modules for $\mathop{\rm char}\F=0$ into the statement on these characters for positive characteristic.

\begin{definition}\label{definition:14}
Fix some $n\in\mathbb N$. The unit $e\in\widehat\W$ is always $n$-reachable.

An element $w\in\widehat\W$
is $n$-reachable if and only if there exists some reduced expression $w=s_1\cdots s_l$ with $s_i\in\widehat\S$
such that $(H_{s_1}+v)\cdots(H_{s_l}+v)=H_w+\sum_{x<w}f_{x,w}H_x$ and one of the following conditions holds for any $x<w$:
\begin{enumerate}
\itemsep=3pt
\item $f_{x,w}(1)\le n$ and $x$ is $n$-reachable.
\item $f_{x,w}\in v\Z[v]$.
\end{enumerate}
\end{definition}
\noindent
Obviously, any element of $\widehat\W$ is $n$-reachable for suitable $n$.

\begin{corollary}\label{corollary:7}
If $w\in\widehat\W$ is $3$-reachable and such that
$\widehat\G_{\le w}$ satisfies the GKM-property, then $v^{\ell(w)}h(\B(w))=\underline H_w$.
\end{corollary}
\begin{proof} We apply induction on $\ell(w)$, the case $w=e$ being obvious. Suppose that $w\in\widehat\W$
is \linebreak$3$-reachable and $w\ne e$. Consider a reduced representation $w=s_1\cdots s_l$ as
in Definition~\ref{definition:14}. Our aim is to prove that $\B(\s)$ decomposes into a sum of
Braden-MacPherson sheaves as in the case $\mathop{\rm char}\F=0$ and that each summand except $\B(w)$ has
the form $\B(x)$ for some $3$-reachable $x<w$.
To this end, we calculate the defect $\d(\B(\s))$ and apply~(\ref{eq:def}).

Under the notation of Definition~\ref{definition:14}, we have $f_{x,w}=v^{l-\ell(x)}\rk'\B(\s)^x$ for any $x<w$
by Lemma~\ref{lemma:9}.
So the first possible case is that the ungraded rank of $\B(\s)^x$ is not greater than $3$ and $x$
is $3$-reachable. In this case, the defect $\d(\rho_{x,\delta x})$ is given by
Lemmas~\ref{lemma:2by2} and~\ref{lemma:3by3} and we can apply induction to obtain
the character of the possible direct summand $\B(x)$. Note that $\widehat\G_{\le x}$
satisfies the GKM-property, as $\widehat\G_{\le w}$ does so.

In the second possible case, we have $f_{x,w}\in v\Z[v]$. It follows that all generators of $\B(\s)^x$ have degrees
less than $l-\ell(x)$. By Theorem~\ref{theorem:4}, the generators of $\B(\s)_{[x]}$ have degrees more than
$l-\ell(x)$. Hence the matrix $\Phi(\s,x)$ does not have entries of degree $0$ and the defect $\d(\rho_{x,\delta x})$
is zero. This means that $\B(x)$ can not be a direct summand of $\B(s)$.
\end{proof}

This corollary is trivially true for $1$-reachable elements, since they are just those $w\in\widehat\W$
that have a reduced representation $w=s_1\cdots s_l$ such that
$(H_{s_1}+v)\cdots(H_{s_l}+v)=\underline H_w$.
On the other hand, there are a lot of 3-reachable elements that are not 1-reachable,
see, for example, the table below.
\bigskip
\begin{center}
\begin{tabular}{|c||c|c|c|c|c|}
\hline
Type of the root system  \vphantom{$A^{A^A}$}        &$A_1$&$A_2$&$A_3$&$A_4$&$A_5$\\
\hline
number of 1-reachable elements \vphantom{$A^{A^A}$} & 2   &  5  &  14 &  42 &  132\\
\hline
number of 3-reachable elements \vphantom{$A^{A^A}$} & 2   &  6  &  22 &  83 &  310\\
\hline
\end{tabular}
\end{center}

\end{document}